\documentclass[12pt]{amsart}
\usepackage{amsmath}
\usepackage{amssymb}
\usepackage{amsthm}
\usepackage[all]{xy}
\usepackage{graphicx}
\usepackage[margin=2.5cm]{geometry}
\usepackage{color}
\usepackage{hyperref}
\usepackage{verbatim}

\numberwithin{equation}{section}

\newcommand{\N}{\ensuremath{\mathbb{N}}}
\newcommand{\R}{\ensuremath{\mathbb{R}}}

\newtheorem{theorem}{Theorem}[section]
\newtheorem{lemma}[theorem]{Lemma}

\newtheorem{proposition}[theorem]{Proposition}

\newtheorem{corollary}[theorem]{Corollary}
\newtheorem{claim}[theorem]{Claim}
\newtheorem{Notation}{Notational Convention}[section]

\theoremstyle{definition} 
\newtheorem{example}[theorem]{Example}
\newtheorem{defn}{Definition}[section]
 
\newtheorem{remark}{Remark}[section] 
\newtheorem{notation}{Notation}[section]

\newcommand{\Th}{{\mathsf T}\!{\mathsf h}}
\newcommand{\MT}{{\mathsf M}{\mathsf T}}
\newcommand{\Mor}{{\mathcal M}or}
\newcommand{\Ob}{\mathcal{O}b}

\newcommand{\MO}{{\mathsf M}{\mathsf O}}
\newcommand{\MTO}{{\mathsf M}{\mathsf T}{\mathsf O}}
\newcommand{\Iso}{{\mathrm I}{\mathrm s}{\mathrm o}}
\DeclareMathOperator*{\Cofibre}{Cofibre}

\DeclareMathOperator*{\colim}{colim}

\DeclareMathOperator*{\Diff}{Diff}
\DeclareMathOperator*{\Emb}{Emb}
\DeclareMathOperator*{\BDiff}{BDiff}
\DeclareMathOperator*{\Sub}{Sub}

\DeclareMathOperator*{\Maps}{Maps}

\author{Nathan Perlmutter}

\address{Department of Mathematics, University of Oregon, Eugene, OR,
  97403, USA} 
  
  \email{nperlmut@uoregon.edu}

\title[Cobordism Category of Manifolds With
Singularities]{Cobordism category of manifolds \\ with Baas-Sullivan
  singularities, Part I}
  
\begin{document}
  
  \maketitle
  \begin{abstract} For a fixed closed manifold $P$, we construct a
    cobordism category of embedded manifolds with
    Baas-Sullivan singularities modeled on $P$. Our main theorem identifies
    the homotopy type of the classifying space of this cobordism
    category with that of the infinite loop-space of a certain Thom spectrum,
     related to the spectrum $\MT(d)$ introduced in \cite{GMTW
      09}. We obtain an analogue of the Bockstein-Sullivan exact
    couple that arises between the classical bordism theories $MO$ and
    $MO_{P}$ on the level of cobordism categories and
      their classifying spaces. 
    \end{abstract}
    
\thispagestyle{empty}
\section{Introduction and Statement of Main Results}  \label{Introduction}
 Fix a closed, smooth manifold $P$. 
Following \cite{Ba 73} and \cite{B 92}, a manifold with Baas-Sullivan singularities modeled on $P$ is a smooth manifold $W$ equipped with the following:
 \begin{enumerate}
 \item[i.] The boundary is given a decomposition as the union of two \textit{faces}, $\partial W = \partial_{0}W\cup\partial_{1}W$ 
 such that $\partial(\partial_{1}W) = \partial_{0}W\cap\partial_{1}W = \partial(\partial_{0}W)$ is a closed manifold.
 \item[ii.] The face $\partial_{1}W$ has the factorization, $\partial_{1}W = \beta_{1}W\times P$ for some manifold $\beta_{1}W$. 
 \end{enumerate}
 We will call such manifolds $P$-manifolds. 
  The face $\partial_{0}W$ is said to be the \textit{boundary} of $W$. 
  If $W$ is compact and $\partial_{0}W = \emptyset$ then $W$ is said to be a closed $P$-manifold. 
  Two closed $d$-dimensional $P$-manifolds $M_{a}$ and $M_{b}$ are said to be cobordant if there is a $(d+1)$-dimensional $P$-manifold $W$ such that $\partial_{0}W = M_{a}\sqcup M_{b}$. 
  
  We are interested in the cobordism theory of $P$-manifolds.
 To simplify our presentation
we will assume that all manifolds are unoriented. However, the same constructions work
in the same way for an arbitrary tangential structure, $\theta: B \rightarrow BO$.  
  We denote by
$\Omega_{*}$ the graded cobordism group of unoriented
manifolds.
Using the above definitions of $P$-manifolds and cobordism of $P$-manifolds, one can define the graded cobordism
  group $\Omega^{P}_{*}$ of unoriented $P$-manifolds.

The group $\Omega^{P}_*$ is related to $\Omega_*$ by means of
the well-known Bockstein-Sullivan exact couple:
\begin{equation} \label{couple}
\xymatrix{
\Omega_{*} \ar[drr]^{i} &&&& \Omega_{*} \ar[llll]_{\times P}\\
&& \Omega^{P}_{*} \ar[urr]_{\beta_{1}} &&}
\end{equation}
The map $\times P$ is the homomorphism of degree $\dim(P)$ given by
multiplication by $P$. The map $i$ is given by inclusion and $\beta_{1}$ is the degree $-1$ homomorphism given
by $M \mapsto \beta_{1}M$. 
This exact couple arises from a cofibre sequence of spectra,
$
\Sigma^{p}\MO \longrightarrow \MO \longrightarrow \MO_{P}
$
where $\MO_{P}$ is the classifying spectrum for $\Omega^{P}_{*}$. 
Details on the construction of this exact couple can be found in \cite{Ba 73} and \cite{B 92}. 

Motivated by the ideas in \cite{Ba 09}, we construct a
cobordism category of manifolds with Baas-Sullivan singularities which generalizes the cobordism category of \cite{GMTW 09},
 and then determine the homotopy-type of its classifying space.  
 In \cite{GMTW 09}, the authors construct a topological category $\mathbf{Cob}_{d+1}$
whose morphisms are $(d+1)$-dimensional submanifolds
$W \subseteq [a,b]\times \R^{d+\infty}$ that intersect
  the walls $\{a,b\}\times\R^{d+\infty}$
orthogonally in $\partial W$. 
This category is topologized in such a way so that there are weak homotopy equivalences,
\begin{equation} \label{eq: GMTW category}
\xymatrix{
\Ob(\mathbf{Cob}_{d+1}) \simeq \bigsqcup_{M} \BDiff(M), & \Mor(\mathbf{Cob}_{d+1}) \simeq \bigg(\bigsqcup_{W}\BDiff(W)\bigg)\bigsqcup \Ob(\mathbf{Cob}_{d+1})
}
\end{equation}
 where $M$ varies
over diffeomorphism classes of $d$-dimensional closed manifolds and $W$ varies over diffeomorphism classes of
cobordisms. Above, the space of identity morphisms is identified with the space of objects.
 In \cite{GMTW 09}, the authors determine the
homotopy type of the classifying space of $\mathbf{Cob}_{d+1}$, namely they prove that there is a  
weak homotopy equivalence
\begin{equation} \label{main GMTW} B\mathbf{Cob}_{d+1} \simeq \Omega^{\infty -1}\MT(d+1). \end{equation}
On the right-hand side, $\MT(d+1)$ is the spectrum
whose $(n+d+1)$-st space is the Thom-space $\Th(U_{d+1,n}^{\perp})$, where
$U_{d+1,n}^{\perp}$ is the orthogonal compliment to the canonical
$(d+1)$-plane bundle over the Grassmannian $G(d+1,n)$, of $(d+1)$-dimensional vector subspaces of $\R^{d+1+n}$. 

Following this work from \cite{GMTW 09}, we construct an analogous cobordism
category of $P$-manifolds. We fix once and for all an embedding
\begin{equation} \label{P-emb intro} i_{P}: P \hookrightarrow \R^{p+m} \end{equation}
 with $p = \dim(P)$ and $m >> p$. 
 We construct a topological category $\mathbf{Cob}^{P}_{d+1}$ whose morphisms are given by $(d+1)$-dimensional embedded $P$-manifolds,
$$
W\subseteq
[a,b]\times\R_{+}\times\R^{d-1+\infty}\times\R^{p+m}
$$ 
such that:
$$\begin{aligned}
W \cap (\{a,b\}\times\R_{+}\times\R^{d-1+\infty}\times\R^{p+m}) &= \partial_{0}W, \\
W \cap ([a,b]\times\{0\}\times\R^{d-1+\infty}\times\R^{p+m}) &= \partial_{1}W, 
\end{aligned}$$
 and $\partial_{1}W$ has the factorziation,
 $$\partial_{1}W \; = \; \beta_{1}W\times i_{P}(P) $$
 where $\beta_{1}W \subset [a,b]\times\{0\}\times\R^{d-1+\infty}$ is a submanifold and $i_{P}(P) \subset \R^{p+m}$ is
the submanifold given by the embedding specified in (\ref{P-emb intro}).  
  Here and throughout this paper $\R_{+}$ denotes the half-open interval $[0,\infty)$. 
  We topologize this category in a way similar to as in \cite{GMTW 09} so that there are homotopy equivalences,
$$\xymatrix{
\Ob(\mathbf{Cob}^{P}_{d+1}) \; \simeq \; {\displaystyle\bigsqcup_{M}} \BDiff_{P}(M), & \Mor(\mathbf{Cob}^{P}_{d+1}) \; \simeq \;
\bigg({\displaystyle \bigsqcup_{W}} \BDiff_{P}(W)\bigg) \bigsqcup \Ob(\mathbf{Cob}^{P}_{d+1}) } $$ 
where $M$ varies over diffeomorphism classes of closed
$d$-dimensional $P$-manifolds and $W$ varies over diffeomorphism
classes of $(d+1)$-dimensional $P$-manifold cobordisms. 
For a $P$-manifold $W$, $\Diff_{P}(W)$ is defined to be the group of diffeomorphisms $g:
W\rightarrow W$ such that the restriction $g|_{\partial_{1} W}$ is
equal to the product $g_{\beta_{1}W}\times Id_{P}$ where
$g_{\beta_{1}W}$ is a diffeomorphism of $\beta_{1}W$. 

The main goal of this paper is to determine the homotopy type of
the classifying space, $B\mathbf{Cob}^{P}_{d+1}$. To do so we 
construct a new spectrum $\MT_{P}(d+1)$ as
follows.
From the embedding in (\ref{P-emb intro}) used to
  construct $\mathbf{Cob}^{P}_{d+1}$, we obtain a
Pontryagin-Thom map $c_{P}: S^{p+m} \longrightarrow
\Th(U^{\perp}_{p,m}).  $ The natural multiplication map given by
sending a pair of vector sub-spaces to their product,
$$\xymatrix{
U_{d-p,n-m}^{\perp}\times U_{p,m}^{\perp} \ar[rrr]^{\hat{\mu}} \ar[d] &&& U_{d,n}^{\perp} \ar[d] \\
G(d-p,n-m)\times G(p,m) \ar[rrr]^{\mu} &&& G(d,n),}$$
yields a map of Thom-spaces,
$\Th(\hat{\mu}): \Th(U^{\perp}_{d-p,n-m})\wedge \Th(U^{\perp}_{p,m}) \longrightarrow \Th(U^{\perp}_{d,n}).$
The composition
$$\xymatrix{\Th(U^{\perp}_{d-p,n-m})\wedge S^{p+m}
  \ar[rr]^{c_{P}\wedge Id} &&
  \Th(U^{\perp}_{d-p,n-m})\wedge\Th(U^{\perp}_{p,m})
  \ar[rr]^{\ \ \ \ \Th(\hat{\mu})} && \Th(U^{\perp}_{d,n})}$$ then induces a
map of spectra which we denote by $\tau_{P}:
\MT(d-p) \longrightarrow \MT(d).$

There is another map of spectra $\hat{j}_{d}: \Sigma^{-1}\MT(d)
\longrightarrow \MT(d+1)$, induced by the bundle map covering the
standard embedding $G(d,n) \hookrightarrow G(d+1,n)$ of
Grassmannians. We define $\MT_{P}(d+1)$ to be
the cofibre of the composition of spectrum maps,
\begin{equation} \label{composition} \xymatrix{\Sigma^{-1}\MT(d-p)
    \ar[rr]^{\Sigma^{-1}\tau_{P}} && \Sigma^{-1}\MT(d)
    \ar[rr]^{\hat{j}_{d}} && \MT(d+1).} \end{equation} Details of
this construction are covered in Section \ref{Thom Spectrum}.  We now
state our main result:
\begin{theorem}[Main Theorem] \label{thm: Main} 
There is a weak homotopy equivalence
$$B\mathbf{Cob}^{P}_{d+1} \simeq \Omega^{\infty -1}\MT_{P}(d+1).$$
\end{theorem}
The spectrum
$\MT_{P}(d+1)$ was constructed using the
Pontryagin-Thom map for a particular embedding of our manifold
$P$. The homotopy class of this Pontryagin-Thom map only depends on
the cobordism class of $P$. This observation leads to
  the following:
\begin{corollary} Let $P_{1}$ and $P_{2}$ be smooth closed manifolds
  of the same dimension. Suppose that $P_{1}$ and $P_{2}$ are
  cobordant. Then there is a weak homotopy equivalence,
$$B\mathbf{Cob}^{P_{1}}_{d+1} \simeq B\mathbf{Cob}^{P_{2}}_{d+1}.$$
\end{corollary}
We consider the functors
$
\xymatrix{ 
\mathbf{Cob}_{d+1} \ar[rr]^{i} && \mathbf{Cob}^{P}_{d+1} \ar[rr]^{\beta_{1}} && \mathbf{Cob}_{d-p}\ ,
} 
$
where $i$ is given by inclusion and $\beta_{1}$ sends a
$(d+1)$-dimensional $P$-manifold $W$ to the $(d-p)$-dimensional
manifold $\beta_{1}W$. 
We have the following theorem. 
\begin{theorem} \label{thm: homotopy fibre} 
Passing to classifying spaces, the sequence of functors given above induces a homotopy fibre-sequence,
$\xymatrix{
B\mathbf{Cob}_{d+1} \ar[rr]^{B(i)} && B\mathbf{Cob}^{P}_{d+1} \ar[rr]^{B(\beta_{1})} && B\mathbf{Cob}_{d-p}.
}$
\end{theorem}

\subsection{Outline of Paper}
This paper is structured as follows. 
Sections $2$ and $3$ are devoted to carefully defining $P$-manifolds and the different mapping spaces associated to them which include diffeomorphism groups and certain spaces of embeddings. 
In Section $4$ we give a rigorous definition of the cobordism category. 
 In Section $5$ we give a recollection of sheaves and define the main sheaf $\mathbf{D}^{P}_{d+1}$ whose representing space is latter shown to be weakly equivalent to $B\mathbf{Cob}^{P}_{d+1}$.
  In Section $6$ we construct the spectrum $\MT_{P}(d\!+\!1)$ and Section $7$ is devoted to
proving the weak homotopy equivalence
$|\mathbf{D}^{P}_{d+1}| \simeq \Omega^{\infty-1}\MT_{P}(d+1)$. 
In Section 8, we complete the proofs of Theorems \ref{thm: Main} and \ref{thm: homotopy fibre}. 
Sections 9, 10, 11 and the appendix are devoted to the proofs of technical results used earlier in the paper. 
  
To simplify the exposition we will only treat unoriented manifolds with
Baas-Sullivan singularities modeled on a single fixed manifold $P$.  One could easily
adapt our proofs to derive a corresponding theorem for $P$-manifolds
with arbitrary tangential structure. 

\subsection{Acknowledgments}
The author would like to thank Boris Botvinnik for suggesting this
particular problem and for numerous helpful discussions on the subject
of this paper.  The author is also grateful to Nils Baas
  for his encouraging remarks and to Oscar Randal-Williams for very helpful
  critical comments on the earlier version of this work.

  \section{Manifolds With Singularities}
  \label{Singularities}
  We begin with a definition of manifolds with Baas-Sullivan
  singularities modeled on a fixed manifold $P$. Fix once and for all a closed, smooth manifold $P$ and let $p$ denote the dimension of $P$. 
  Throughout the paper we will let $\R_{+}$ denote the half-open interval $[0, \infty)$. 

  \begin{defn}\label{P mfd} Let $M$ be a $d$-dimensional smooth
    manifold with corners, equipped with the following extra structure:
 \begin{enumerate} 
  \item[i.] The boundary of $M$ is given a decomposition,
  $\partial M = \partial_{0}M \cup \partial_{1}M$
 into a union of $(d-1)$-dimensional manifolds such that 
  $$\partial_{0}M \cap \partial_{1}M =   \partial (\partial_{0} M) =  \partial (\partial_{1} M)$$
  is a closed $(d-2)$-dimensional manifold. 
  We denote, $\partial_{0,1}M := \partial_{0}M \cap \partial_{1}M$.
\item[ii.] 
There are embeddings,
$$h_{0}: \partial_{0}M\times\R_{+} \longrightarrow M \quad \text{and} \quad h_{1}: \partial_{1}M\times\R_{+} \longrightarrow M$$
which satisfy:
\\
\begin{enumerate} \itemsep8pt
\item[(a)] $h^{-1}_{0}(\partial_{0}M) = \partial_{0}M\times\{0\}$ and $h^{-1}_{1}(\partial_{1}M) = \partial_{1}M\times\{0\}$,
\item[(b)] $h_{0}(\partial_{0,1}M\times[0,\infty))\subset \partial_{1}M$ and $h_{1}(\partial_{0,1}M\times[0,\infty)) \subset \partial_{0}M$,
\item[(c)] for all $(x, t_{0}, t_{1}) \in \partial_{0,1}M\times\R_{+}^{2}$ the following equation is satisfied,
$$h_{0}(h_{1}(x, t_{1}), t_{0}) = h_{1}(h_{0}(x, t_{0}), t_{1}).$$
\end{enumerate}
\item[iii.] There is a manifold $\beta_{1}M$ and diffeomorphism,
$$\phi_{1}: \partial_{1}M \stackrel{\cong} \longrightarrow \beta_{1}M\times P.$$
We let $\beta_{0,1}M$ denote the boundary $\partial(\beta_{1}M)$ and let, 
$$\phi_{0,1}: \partial_{0,1}M \stackrel{\cong} \longrightarrow \beta_{0,1}M\times P$$
denote the diffeomorphism obtained by restricting $\phi_{1}$ to $\partial_{0,1}M = \partial(\partial_{1}M)$. 
\end{enumerate}
With the above conditions satisfied, the triple $(M, (\phi_{1}, \phi_{0,1}), (h_{0}, h_{1}))$ is called a \textit{$P$-manifold.}
The manifold $\beta_{1}M$ is called the \textit{Bockstein}, the pair of diffeomorphisms $(\phi_{1}, \phi_{0,1})$ is called the \textit{structure map}, and the pair of embeddings $(h_{0}, h_{1})$ is called the \textit{collar}. 
\end{defn}  
\begin{notation}
When denoting a $P$-manifold we will usually drop the structure maps and collar from the notation and denote the $P$-manifold by its underlying manifold. We will denote, $M := (M, (\phi_{1}, \phi_{0,1}), (h_{0}, h_{1}))$. 
\end{notation}
Let $M$ be a $P$-manifold as described in the above definition. 
By setting,
$$\partial_{1}(\partial_{0}M) = \partial_{0,1}M, \quad \partial_{0}(\partial_{0}M) = \emptyset, \quad \text{and} \quad \beta_{1}(\partial_{0}M) = \beta_{0,1}M,$$ 
and restricting the structure map and collar associated to $M$, $\partial_{0}M$ obtains the structure of a $P$-manifold. 
We call $\partial_{0}M$ the \textit{boundary} of $M$. 
If $\partial_{0}M = \emptyset$ then $M$ is said to be a $P$-manifold without boundary. 
If $M$ is compact and $\partial_{0}M = \emptyset$ then $M$ is said to be a \textit{closed $P$-manifold}.

We will need to consider maps from $P$-manifolds to arbitrary topological spaces.
\begin{defn} \label{defn: P-map}
If $M$ is a $P$-manifold and $X$ is a topological space then a continuous map $f: M \longrightarrow X$ is said to be a \textit{$P$-map} if there exists 
a map $f_{\beta_{1}}: \beta_{1}M \longrightarrow X$ such that
the restriction of $f$ to $\partial_{1}M$ factors as,
\begin{equation} \label{eq: P-map factorization}
\xymatrix{
\partial_{1}M \ar[rr]^{\Phi} && \beta_{1}M\times P \ar[rr]^{\ \ \ \ \text{proj}_{\beta_{1}M}} && \beta_{1}M \ar[rr]^{f_{\beta_{1}}} && X.
}
\end{equation}
If $X$ is a smooth manifold then a $P$-map $f: M \longrightarrow X$ is said to be \text{smooth} if $f$ is a smooth map when considering $M$ as a smooth manifold with corners. 
\end{defn}

We will have to consider vector bundles over $P$-manifolds. 
\begin{defn} \label{defn: P-vector bundles}
Let $M$ be a $P$-manifold. Let $\pi: E \longrightarrow M$,  $\pi_{\beta_{1}}: E_{\beta_{1}} \longrightarrow \beta_{1}M$, and $\pi_{P}: E_{P} \longrightarrow P$ be vector bundles and let
$$\hat{\phi}_{E}: E|_{\partial_{1}M} \stackrel{\cong} \longrightarrow (E_{\beta_{1}}\times E_{P})\oplus\epsilon^{1}$$
be a vector bundle isomorphism that covers the structure map $\phi_{1}: \partial_{1}M \stackrel{\cong} \longrightarrow \beta_{1}M\times P$ (the bundle on the right hand side is assumed to be over $\beta_{1}M\times P$).
The pair $(E, \hat{\phi}_{E})$ is said to be a \textit{$P$-vector} bundle over $M$.
We refer to $\hat{\phi}_{E}$ as the structure map. 
\end{defn}
\begin{notation}
When working with a $P$-vector bundle as in the previous definition, we will drop the bundle isomorphism $\hat{\phi}_{1}: E|_{\partial_{1}M} \longrightarrow (E_{\beta_{1}}\times E_{P})\oplus\epsilon^{1}$ from the notation and denote $E := (E, \hat{\phi}_{1})$. 
We will always use the same greek letter to denote the structure map and will always use the same notational convention to denote the auxiliary bundles $E_{\beta_{1}M}$ and $E_{P}$. 
\end{notation}
\begin{example} \label{example: tangent bundle}
For any $P$-manifold, the tangent bundle $TM \longrightarrow M$ naturally has the structure of a $P$-vector bundle as follows. 
The collar embedding $h_{1}: \partial_{1}M\times\R_{+} \longrightarrow M$ induces a bundle isomorphism 
$$TM|_{\partial_{1}M} \stackrel{\cong} \longrightarrow T\partial_{1}M\oplus\epsilon^{1}$$
which covers the identity on $M$. 
Using this bundle isomorphism we obtain the bundle isomorphism,
\begin{equation} \label{eq: P-tangent bundle}
\xymatrix{
TM|_{\partial_{1}M} \ar[rr]^{\cong} && T\partial_{1}M\oplus\epsilon^{1} \ar[rrr]^{d\phi_{1}\oplus Id_{\epsilon^{1}}}_{\cong} &&& (T\beta_{1}M\times TP)\oplus\epsilon^{1}
}
\end{equation}
where $d\phi_{1}$ denotes the differential of the structure map.
In this way the bundle isomorphism (\ref{eq: P-tangent bundle}) endows $TM \longrightarrow M$ with the structure of a $P$-vector bundle. 
\end{example}

Transversality will play an important role in the constructions used to prove the main theorem. 
\begin{defn} \label{defn: P-transversality}
Let $U$ be a smooth manifold, let $K \subset U$ be a submanifold, and let $M$ be a $P$-manifold. 
A smooth $P$-map $f: M \longrightarrow U$ is said to be \textit{transverse} to the submanifold $K$ if both $f$ and the map $f_{\beta_{1}}: \beta_{1}M \rightarrow U$ are transverse as smooth maps to $K$. 
In this case we write $f\pitchfork K$ to indicate transversality. 
\end{defn}
The following proposition is easy to verify. 
\begin{proposition} \label{prop: P-transverse}
Let $U$ be a manifold of dimension $r$, $K \subset U$ a submanifold of dimension $k$, and let $M$ be a $P$-manifold of dimension $d$. 
Let $f: M \longrightarrow U$ be a smooth $P$-map transverse to $K$. 
Then the space $f^{-1}(K)$ is $P$-manifold of dimension $d+ k - r$ with Bockstein given by $\beta_{1}(f^{-1}(K)) = f_{\beta_{1}}^{-1}(K)$. 
\end{proposition}

On a similar note, submersions will play an important role in the constructions used in the proof of the main theorem. 
If $U$ is a smooth manifold and $M$ is a $P$-manifold, a smooth $P$-map $f: M \longrightarrow U$ is said to be a \textit{$P$-submersion} if both $f$ and $f_{\beta_{1}}$ are submersions when treating $M$ as a smooth manifold. 
It follows immediately from Proposition \ref{prop: P-transverse} that for $x \in U$, the space $f^{-1}(x)$ is a $P$-manifold of dimension $\dim(M)-\dim(U)$ with $\beta_{1}(f^{-1}(x)) = f_{\beta_{1}}^{-1}(x)$.
\begin{example} \label{example: kernel bundles}
Let $M$ be a $P$-manifold and let $X$ be a smooth manifold. 
Let $\pi: M \longrightarrow X$ be a $P$-submersion. 
Denote by $T^{\pi}M \rightarrow M$ the sub-vector bundle of the tangent bundle $TM$ given by the kernel of the differential of the submersion $\pi$.
Denote by $T^{\pi}\beta_{1}M \rightarrow \beta_{1}M$ the sub-bundle of $T\beta_{1}M$ given by the kernel of the submersion $\pi_{\beta_{1}}: \beta_{1}M \longrightarrow X$. 
The factorization from (\ref{eq: P-map factorization}) of the restriction of $\pi$ to $\partial_{1}M$ implies that there is an isomorphism 
$$T^{\pi}M|_{\partial_{1}M} \stackrel{\cong} \longrightarrow (T^{\pi}\beta_{1}M\times TP)\oplus\epsilon^{1}$$
that covers the structure map, $\phi_{1}: \partial_{1}M \stackrel{\cong} \longrightarrow \beta_{1}M\times P$. 
It follows that the kernel bundle of any $P$-submersion has the structure of a $P$-vector bundle. 
\end{example}

We are interested in the cobordism theory of $P$-manifolds. For this
we make the following definition.
\begin{defn} Let $M_{a}$ and $M_{b}$ be closed $P$-manifolds of dimension $d$ and let $W$ be a compact $P$-manifold of dimension $d+1$. 
If $\partial_{0}W = M_{a}\sqcup M_{b}$ then the triple $(W; M_{a}, M_{b})$ is said to be a \textit{$P$-manifold cobordism triple}. 
Two closed $P$-manifolds $M_{a}$ and $M_{b}$ of the same dimension are said to be cobordant if there exists a $(d+1)$-dimensional $P$-manifold $W$ such that $\partial_{0}W =  M_{a}\sqcup M_{b}$.   
\end{defn}

\section{Mapping Spaces} \label{section: mapping spaces}
We will need to consider certain spaces of maps between $P$-manifolds. 
\subsection{Diffeomorphisms}
For what follows let $M_{a}$ and $M_{b}$ be $P$-manifolds. 
For $i = 0, 1$, we denote by $h^{a}_{i}$ and $h^{b}_{i}$ the collar embeddings associated to $M_{a}$ and $M_{b}$. 
We denote by $\phi^{a}_{1}$ and $\phi^{b}_{1}$ the structure maps. 
\begin{defn} \label{defn: P-morphism}
A smooth map $f: M_{a} \longrightarrow M_{b}$ is said to be a \textit{$P$-morphism} if the following conditions are satisfied:
\begin{enumerate}
\item[i.] $f(\partial_{0}M_{a}) \subset \partial_{0}M_{b}$ and $f(\partial_{1}M_{a}) \subset \partial_{1}M_{b}$.
\item[ii.] There exists a real number $\varepsilon > 0$ such that 
$$\begin{aligned}
f(h^{a}_{0}(x, t)) & = h^{b}_{0}(f(x), t) \quad \text{for $(x, t) \in \partial_{0}M_{a}\times[0,\varepsilon)$,}\\
f(h^{a}_{1}(y, s)) & = h^{b}_{1}(f(y), s) \quad \text{for $(y, s) \in \partial_{1}M_{a}\times[0,\varepsilon)$.}
\end{aligned}$$
\item[iii.] There exists a smooth map $f_{\beta_{1}}: \beta_{1}M_{a} \longrightarrow \beta_{1}M_{b}$ such that the restriction of $f$ to $\partial_{1}M_{a}$ has the factorization,
$$\xymatrix{
\partial_{1}M_{a} \ar[rr]^{\phi^{a}_{1}}_{\cong} && \beta_{1}M_{a}\times P \ar[rr]^{f_{\beta_{1}}\times Id_{P}} && \beta_{1}M_{b}\times P \ar[rr]^{(\phi^{b})^{-1}}_{\cong} && \partial_{1}M_{b}
}$$
\end{enumerate}
\end{defn}
We denote by $C^{\infty}_{P}(M_{a}, M_{b})$ the space of $P$-morphisms $M_{a} \rightarrow M_{b}$, topologized as a subspace of the space of smooth maps $M_{a} \rightarrow M_{b}$, in the $C^{\infty}$-topology.
For a $P$-manifold $M$, we let $C^{\infty}_{P}(M)$ denote the space $C^{\infty}_{P}(M, M)$. 
We will need to consider diffeomorphisms of $P$-manifolds as well. 
\begin{defn} A smooth map between $P$-manifolds $f: M_{a} \longrightarrow M_{b}$, is said to be a \textit{$P$-diffeomorphism} if it is both a diffeomorphism as a map of smooth manifolds and a $P$-morphism, i.e. it satisfies all conditions of Definition \ref{defn: P-morphism}. 
\end{defn}
 We denote by $\Diff_{P}(M_{a}, M_{b})$ the space of $P$-diffeomorphisms from $M_{a}$ to $M_{b}$, where
 $\Diff_{P}(M_{a}, M_{b})$ is topologized as a subspace of $C^{\infty}_{P}(M_{a}, M_{b})$. 
 For a $P$-manifold $M$, we let $\Diff_{P}(M)$ denote the space $\Diff_{P}(M, M)$ of 
 self-$P$-diffeomorphisms $M \rightarrow M$.
 The space $\Diff_{P}(M)$ has the structure of a topological group with product given by composition.

\begin{proposition} \label{prop: P-diff open set}
For any two compact $P$-manifolds $M_{a}$ and $M_{b}$, $\Diff_{P}(M_{a}, M_{b})$ is an open subset of $C^{\infty}_{P}(M_{a}, M_{b})$. 
\end{proposition}
\begin{proof}
Denote by $C^{\infty}_{\partial}(M_{a}, M_{b})$ the space of smooth maps $M_{a} \rightarrow M_{b}$ that satisfy conditions i. and ii. of Definition \ref{defn: P-morphism} but which may fail to satisfy condition iii. The space $C^{\infty}_{\partial}(M_{a}, M_{b})$ is topologized in the $C^{\infty}$-topology. 
Similarly, we denote by $\Diff_{\partial}(M_{a}, M_{b})$ the subspace of $C^{\infty}_{\partial}(M_{a}, M_{b})$ which consists of all smooth maps $f: M_{a} \rightarrow M_{b}$ satisfying conditions i. and ii. of Definition \ref{defn: P-morphism} such that $f$ is also a diffeomorphism of smooth manifolds. 
It follows from \cite[Theorem 1.7]{Hi 76} that $\Diff_{\partial}(M_{a}, M_{b}) \subset C^{\infty}_{\partial}(M_{a}, M_{b})$ is an open subset.
It follows from
$$\xymatrix{
\Diff_{P}(M_{a}, M_{b}) = \Diff_{\partial}(M_{a}, M_{b})\cap C^{\infty}_{P}(M_{a}, M_{b}).
}$$
that $\Diff_{P}(M_{a}, M_{b})$ is an open subset of $C^{\infty}_{P}(M_{a}, M_{b})$.
\end{proof}

\subsection{Embeddings} \label{eq: P-embeddings}
We will need to consider certain spaces of embeddings of $P$-manifolds. 
Fix once and for all a smooth embedding
\begin{equation} \label{eq: chosen embedding of P}
i_{P}: P \longrightarrow \R^{p+m}
\end{equation}
with $m > p$. 
We will use this choice of embedding in all of our constructions to come. 
Throughout this section and the sections to come, we will use the following notational convention. 
For $n \in \N$, we will denote
\begin{equation}
\bar{n} := n - p - m - 1.
\end{equation}

For what follows let $X$ be a smooth manifold without boundary. 
For a positive integer $n$, let 
$$q: \R_{+}\times X\times\R^{\bar{n}}\times\R^{p+m} \longrightarrow X\times\R_{+}\times\R^{\bar{n}}\times\R^{p+m}$$
be the ``permutation'' map defined by, $q(t, x, y, z) = (x, t, y, z)$. 
\begin{defn} \label{defn: i-P maps}
Let $M$ be a $P$-manifold with $\partial_{0}M = \emptyset$ (we allow for $M$ to be non-compact). 
We define $\mathcal{E}_{P, n}(M, X)$ to be the space of smooth embeddings
$$\varphi: M \longrightarrow X\times\R_{+}\times\R^{\bar{n}}\times\R^{p+m}$$ 
which satisfy the following conditions:
\begin{enumerate}
\item[i.] $\varphi(\partial_{1}M) \subset X\times\{0\}\times\R^{\bar{n}}\times\R^{p+m}$. 
\item[ii.] 
There exists a real number $\varepsilon > 0$ such that if $(y, s) \in \partial_{1}M\times[0,\varepsilon)$ then,
$$\varphi(h_{1}(y, s)) \; = \; q(\varphi(y), s).$$
\item[iii.] There exists a map $\varphi_{\beta_{1}}: \beta_{1}M \longrightarrow X\times\R^{\bar{n}}$ such that 
the restriction of $\varphi$ to $\partial_{1}M$ has the factorization,
$$\xymatrix{
\partial_{1}M \ar[rr]_{\cong}^{\phi_{1}} && \beta_{1}M\times P \ar[rrr]^{\varphi_{\beta_{1}}\times i_{P}} &&& (X\times\R^{\bar{n}})\times\R^{p+m}.
}$$
\end{enumerate}
The space $\mathcal{E}_{P, n}(M, X)$ is topologized as a subspace of the space of smooth maps from $M$ to $X\times\R_{+}\times\R^{\bar{n}}\times\R^{p+m}$  in the $C^{\infty}$-topology.
We let $\mathcal{E}_{P, n}(M)$ denote the space $\mathcal{E}_{P, n}(M, \text{pt.})$, i.e.\ the space of embeddings $M \rightarrow \R_{+}\times\R^{\bar{n}}\times\R^{p+m}$ that satisfy the conditions given above. 
\end{defn} 

By the following proposition, we are justified in excluding the embedding (\ref{eq: chosen embedding of P}) from the notation. 
\begin{proposition} \label{prop: independence of embedding}
For any $P$-manifold $M$ with $\partial_{0}M = \emptyset$, smooth manifold $X$, and positive integer $n$, the homeomorphism type of the space $\mathcal{E}_{P, n}(M, X)$ does not depend on the embedding $P \hookrightarrow \R^{p+m}$ used to define it. 
\end{proposition}
\begin{proof}
Let $i'_{P}: P \hookrightarrow \R^{p+m}$ be another embedding and let $\mathcal{E}_{P', n}(M, X)$ denote the space of embeddings 
$$M \longrightarrow X\times\R_{+}\times\R^{\bar{n}}\times\R^{p+m}$$
that satisfy all conditions from Definition \ref{defn: i-P maps} with respect to the embedding $i'_{P}$. 
Since $m > p$, there exists an isotopy $\psi_{t}: P \longrightarrow \R^{p+m}$ through embeddings such that $\psi_{0} = i_{P}$ and $\psi_{1} = i'_{P}$. 
By the \textit{isotopy extension theorem} \cite[Theorem 1.3]{Hi 76} there exists a diffeotopy (isotopy through diffeomorphisms) $\Psi_{t}: \R^{p+m} \rightarrow \R^{p+m}$ such that $\Psi_{0} = Id_{\R^{p+m}}$ and $\Psi_{t}\circ i_{P} = \psi_{t}$ for all $t \in [0,1]$. 
Denote by $\Phi$ the diffeomorphism of $X\times\R_{+}\times\R^{\bar{n}}\times\R^{p+m}$ given by the formula
$$\Phi(x, t, y, z) = (x, t, y, \Psi_{1}(z)) \quad \text{for $x \in X$, $t \in \R_{+}$, $y \in \R^{\bar{n}}$, and $z \in \R^{p+m}$.}$$
We define a map,
\begin{equation} \label{eq: embedding space homeo}
\mathcal{E}_{P, n}(M, X) \longrightarrow \mathcal{E}_{P', n}(M, X), \quad \varphi \mapsto \Phi\circ\varphi.
\end{equation}
The inverse to (\ref{eq: embedding space homeo}) is given by the formula $\varphi \mapsto \Phi^{-1}\circ\varphi$. 
Thus (\ref{eq: embedding space homeo}) is a homeomorphism. 
This concludes the proof. 
\end{proof}

For each $n$ there is a natural embedding $\mathcal{E}_{P, n}(M) \hookrightarrow \mathcal{E}_{P, n+1}(M)$. 
We then define,
\begin{equation}
\xymatrix{
\mathcal{E}_{P}(M) := \colim_{n\to\infty}\mathcal{E}_{P, n}(M).
}
\end{equation}
We have the following theorem. 
\begin{theorem} \label{theorem: weakly contractible}
Let $M$ be a closed $P$-manifold (i.e.\ $M$ is compact and $\partial_{0}M = \emptyset$). 
Then the space $\mathcal{E}_{P}(M)$ is weakly contractible. 
\end{theorem}
\begin{proof}
For each $n \in \N$, denote by $\mathcal{E}_{\partial, n}(M)$ the space of embeddings $M \rightarrow \R_{+}\times\R^{\bar{n}}\times\R^{p+m}$ that satisfy conditions i. and ii. of Definition \ref{defn: i-P maps} but which may fail to factor as a product on $\partial_{1}M$ as in condition iii. of Definition \ref{defn: i-P maps}. 
We then denote $\mathcal{E}_{\partial}(M) := \colim_{n\to\infty}\mathcal{E}_{\partial, n}(M)$. 
Let 
$$r_{\partial_{1}}: \mathcal{E}_{\partial}(M) \longrightarrow \Emb(\partial_{1}M, \R^{\infty}\times\R^{p+m})$$
be the map defined by restricting embeddings to the boundary
and let
$$r_{\beta_{1}}: \mathcal{E}_{P}(M) \longrightarrow \Emb(\beta_{1}M, \R^{\infty})$$
be the map defined by $\varphi \mapsto \varphi_{\beta_{1}}$. 
Consider the map 
$$T_{i_{P}}: \Emb(\beta_{1}M, \R^{\infty}) \longrightarrow \Emb(\partial_{1}M, \R^{\infty}\times\R^{p+m})$$
defined by sending an embedding $\varphi: \beta_{1}M \rightarrow \R^{\infty}$ to the embedding given by the composition
$$\xymatrix{
\partial_{1}M \ar[rr]^{\phi_{1}}_{\cong} && \beta_{1}M\times P \ar[rr]^{\varphi\times i_{P}} && \R^{\infty}\times\R^{p+m}.
}$$
It follows immediately from Definition \ref{defn: i-P maps} that the diagram 
\begin{equation} \label{eq: embedding pull-back}
\xymatrix{
\mathcal{E}_{P}(M) \ar[d]^{r_{\beta_{1}}} \ar[rr] && \mathcal{E}_{\partial}(M) \ar[d]^{r_{\partial_{1}}} \\
 \Emb(\beta_{1}M, \R^{\infty}) \ar[rr]^{T_{i_{P}}} && \Emb(\partial_{1}M, \R^{\infty}\times\R^{p+m})
 }
\end{equation}
is cartesian, where the top-horizontal map is the inclusion. 
By the main theorem of \cite{L 63}, the restriction map $r_{\partial_{1}}$ is a locally trivial fibre-bundle. 
It follows from this that the diagram (\ref{eq: embedding pull-back}) is \textit{homotopy cartesian}. 
By \cite[Theorem 2.7]{G 08} the spaces $\mathcal{E}_{\partial}(M)$, $\Emb(\beta_{1}M, \R^{\infty})$, and $\Emb(\partial_{1}M, \R^{\infty}\times\R^{p+m})$ are all weakly contractible. 
This together with the fact that (\ref{eq: embedding pull-back}) is homotopy cartesian implies that $\mathcal{E}_{P}(M)$ is homotopy cartesian as well. 
This concludes the proof. 
\end{proof}

We now define similar mapping spaces for $P$-manifold cobordism triples. 
Let $X$ be a smooth manifold without boundary as above. 
Let 
$$\begin{aligned}
q_{0}: [0, 1]\times X\times\R_{+}\times\R^{\bar{n}}\times\R^{p+m} &\longrightarrow X\times[0,1]\times\R_{+}\times\R^{n}\times\R^{p+m}\\
q_{1}: \R_{+}\times X\times[0,1]\times\R^{\bar{n}}\times\R^{p+m} &\longrightarrow X\times[0,1]\times\R_{+}\times\R^{n}\times\R^{p+m}
\end{aligned}$$
be the ``permutation'' maps defined by
$$q_{0}(t, x, s, y, z) \; = \; (x, t, s, y, z) \quad \text{and} \quad q_{1}(s, x, t, y, z) = (x, t, s, y, z)$$
where $x \in X$, $t \in [0,1]$, $s \in \R_{+}$, $y \in \R^{n}$, and $z \in \R^{p+m}$. 
\begin{defn} \label{defn: maps P-cobrdism}
Let $(W; M_{a}, M_{b})$ be a $P$-manifold bordism triple and let $n$ be a positive integer. 
We define $\mathcal{E}_{P, n}((W; M_{a}, M_{b}), X)$ to be the space of smooth embeddings 
$$\varphi: W \longrightarrow X\times[0,1]\times\R_{+}\times\R^{\bar{n}}\times\R^{p+m}$$
that satisfy the following conditions:
\begin{enumerate}
\item[i.] The following containments hold,
$$\begin{aligned}
\varphi(M_{a}) &\subset X\times\{0\}\times\R_{+}\times\R^{\bar{n}}\times\R^{p+m}, \\  
\varphi(M_{b}) &\subset X\times\{1\}\times\R_{+}\times\R^{\bar{n}}\times\R^{p+m}, \\
\varphi(\partial_{1}W) &\subset X\times[0,1]\times\{0\}\times\R^{\bar{n}}\times\R^{p+m}.
\end{aligned}$$ 
\item[ii.]  There exists a positive real number $\varepsilon$ such that 
$$\begin{aligned}
\varphi(h_{0}(x, t)) = q_{0}(t, \varphi(x)) &\quad \text{if $(x, t) \in M_{a}\times[0,\varepsilon)$,}\\
\varphi(h_{0}(x, t)) = q_{0}(1- t, \varphi(x)) &\quad \text{if $(x, t) \in M_{b}\times[0,\varepsilon)$,}\\
\varphi(h_{1}(y, s)) = q_{1}(s, \varphi(y)) &\quad \text{if $(y, s) \in \partial_{1}W\times[0,\varepsilon)$.}
\end{aligned}$$
\item[iii.] There exists a map $\varphi_{\beta_{1}}: \beta_{1}W \longrightarrow X\times[0,1]\times\{0\}\times\R^{\bar{n}}$ such that 
the restriction of $\varphi$ to $\partial_{1}W$ has the factorization,
$$\xymatrix{
\partial_{1}W \ar[rr]^{\phi_{1}}_{\cong} && \beta_{1}W\times P \ar[rrr]^{\varphi_{\beta_{1}}\times i_{P}\ \ \ \ \ } &&& (X\times[0,1]\times\{0\}\times\R^{\bar{n}})\times\R^{p+m}.
}$$
\end{enumerate}
The space $\mathcal{E}_{P, n}((W; M_{a}, M_{b}), X)$ is topologized in the $C^{\infty}$-topology. 
We let $\mathcal{E}_{P, n}(W; M_{a}, M_{b})$ denote the space $\mathcal{E}_{P, n}((W; M_{a}, M_{b}), \text{pt.})$.
We then set, 
$$\xymatrix{
\mathcal{E}_{P}(W; M_{a}, M_{b}) := \colim_{n\to\infty}\mathcal{E}_{P, n}(W; M_{a}, M_{b}).
}$$
\end{defn}

\begin{remark} \label{remark: analogous theorems}
There are analogues of Proposition \ref{prop: independence of embedding} and Theorem \ref{theorem: weakly contractible} for the spaces \\
$\mathcal{E}_{P, n}((W; M_{a}, M_{b}), X)$ of embeddings of $P$-bordisms and are proven in the same way.
\end{remark}

The terminology given in the next definition will be useful later on when we define the cobordism category of $P$-manifolds and related constructions. 
 \begin{defn} \label{defn: P-submanifold}
 Let $X$ be a smooth manifold without boundary. 
Let $M$ be a $P$-manifold with $\partial_{0}M = \emptyset$ that is embedded as a submanifold of $X\times\R_{+}\times\R^{\bar{n}}\times\R^{p+m}$ for some $n$,
such that the inclusion map 
$$M \hookrightarrow X\times\R_{+}\times\R^{\bar{n}}\times\R^{p+m}$$ 
is an element of the space $\mathcal{E}_{P, n}(M, X).$ 
Then $M$ is called a \textit{$P$-submanifold}. 
Similarly, let $(W; M_{a}, M_{b})$ be a $P$-manifold bordism triple with $W$ embedded as a submanifold of $X\times[0,1]\times\R_{+}\times\R^{\bar{n}}\times\R^{p+m}$ such that the inclusion
$$W \hookrightarrow X\times[0,1]\times\R_{+}\times\R^{\bar{n}}\times\R^{p+m}$$
is an element of the space $\mathcal{E}_{P, n}((W; M_{a}, M_{b}), X)$. 
Then $W$ is called a \textit{$P$-subcobordism}.
\end{defn}

\begin{remark}[Normal Bundles] \label{remark: normal bundles}
Let $M \subset X\times\R_{+}\times\R^{\bar{n}}\times\R^{p+m}$ be a closed $P$-submanifold. 
Let $\pi: M \longrightarrow X$ denote the restriction of the projection 
$X\times\R_{+}\times\R^{\bar{n}}\times\R^{p+m} \longrightarrow X$
onto $M$. 
It follows immediately from condition iii. of Definition \ref{defn: i-P maps} that $\pi$ is a smooth $P$-map (see Definition \ref{defn: P-map}). 
Denote by $N \rightarrow M$ the normal bundle. 
The factorization $\partial_{1}M = \beta_{1}M\times i_{P}(P)$ for $\beta_{1}M \subset  X\times\{0\}\times\R^{\bar{n}}$ and $i_{P}(P) \subset \R^{p+m}$ implies that the restriction of $N$ to $\partial_{1}M$ has the factorization 
\begin{equation}
N|_{\partial_{1}M} \; = \; (N_{\beta_{1}}\times N_{P})\oplus\epsilon^{1}
\end{equation}
where $N_{\beta_{1}} \rightarrow \beta_{1}M$ and $N_{P} \rightarrow i_{P}(P)$ are the normal bundles for $\beta_{1}M$ and $i_{P}(P)$ respectively. 
It follows that the normal bundle of a $P$-submanifold has the structure of a $P$-vector bundle as in Definition \ref{defn: P-vector bundles}. 
\end{remark}

\subsection{$P$-Manifold Fibre Bundles}
Let $M$ be a closed $P$-manifold. 
Consider the the space $\mathcal{E}_{P}(M)$. 
There is continuous group action
$$\xymatrix{
\Diff_{P}(M)\times\mathcal{E}_{P}(M) \longrightarrow \mathcal{E}_{P}(M), \quad (g, \varphi) \mapsto \varphi\circ g.
}$$
It is clear that this action is a free action. 
We let $\mathcal{M}_{P}(M)$ denote the orbit space $\dfrac{\mathcal{E}_{P}(M)}{\Diff_{P}(M)}$. 
Similarly, if $(W; M_{a}, M_{b})$ is a $P$-manifold bordism triple, there is a continuous group action 
$$\xymatrix{
\Diff_{P}(W; M_{a}, M_{b})\times\mathcal{E}(W; M_{a}, M_{b}) \longrightarrow \mathcal{E}(W; M_{a}, M_{b}), \quad (g, \varphi) \mapsto \varphi\circ g
}$$
which is clearly a free action. 
We let $\mathcal{M}_{P}(W; M_{a}, M_{b})$ denote the orbit space $\dfrac{\mathcal{E}(W; M_{a}, M_{b})}{\Diff_{P}(W; M_{a}, M_{b})}$. 
We have the following theorem whose proof we differ to Appendix \ref{Appendix A}. 
\begin{theorem} \label{theorem: local triviality}
The quotient maps
$$\xymatrix{
\mathcal{E}_{P}(M) \longrightarrow \mathcal{M}_{P}(M), \quad \text{and} \quad 
\mathcal{E}(W; M_{a}, M_{b}) \longrightarrow \mathcal{M}_{P}(W; M_{a}, M_{b})
}$$
are locally trivial fibre-bundles. 
\end{theorem}

\begin{remark} \label{remark: classifying spaces}
Combining the above theorem with Theorem \ref{theorem: weakly contractible} (and the corresponding version of Theorem \ref{theorem: weakly contractible} of $P$-cobordisms)
we have weak homotopy equivalences,
\begin{equation} \label{eq: classifying space eq 1}
\xymatrix{
\BDiff_{P}(M) \simeq \mathcal{M}_{P}(M), \quad \BDiff_{P}(W; M_{a}, M_{b}) \simeq \mathcal{M}_{P}(W; M_{a}, M_{b})
}
\end{equation}
where $\BDiff_{P}(M)$ and $\BDiff_{P}(W; M_{a}, M_{b})$ are the \textit{classifying spaces} of the topological groups $\Diff_{P}(M)$ and $\Diff_{P}(W; M_{a}, M_{b})$. 
\end{remark}

Using the Borel construction we define
 \begin{equation} \label{defn: associated bundle} 
\xymatrix{
\widehat{\mathcal{E}}_{P}(M) := \mathcal{E}_{P}(M)\times_{\Diff_{P}(M)}M.
}
 \end{equation}
The standard projection yields a fibre-bundle,
$\widehat{\mathcal{E}}_{P}(M) \longrightarrow \mathcal{M}_{P}(M)$
with fibre $M$ and structure group $\Diff_{P}(M)$. 
This fibre-bundle comes with a natural embedding 
$$
\widehat{\mathcal{E}}_{P}(M) \hookrightarrow \mathcal{M}_{P}(M)\times\R_{+}\times\R^{\infty}\times\R^{p+m}
$$ 
defined by the formula, $(\varphi, x) \mapsto ([\varphi], \; \varphi(x))$.
Now, if $f: X \longrightarrow \mathcal{M}_{P}(M)$ is a smooth map (when treating $\mathcal{M}_{P}(M)$ as a \textit{Banach manifold})
from a smooth manifold $X$, then the pullback
  $$f^{*}\hat{\mathcal{E}}_{P}(M) \; = \; \{(x, v) \in X\times(\R_{+}\times\R^{\infty}\times\R^{p+m})\;\; | \;\; (f(x), v) \in \widehat{\mathcal{E}}_{P}(M)\}$$
  is a smooth $\dim(M) + \dim(X)$-dimensional $P$-submanifold 
$$
E \subseteq X\times\R_{+}\times\R^{\infty}\times\R^{p+m}
$$
such that the projection $\pi: E \rightarrow X$ is a fibre-bundle with structure group $\Diff_{P}(M)$ and fibre $M$.
The fibre over $x$, denoted by $E_{x}$, has the structure of $P$-manifold diffeomorphic to $M$ such that the inclusion map 
$E_{x} \hookrightarrow \R_{+}\times\R^{\infty}\times\R^{p+m}$
is an element of the space $\mathcal{E}_{P}(M)$. 
We have the following lemma. 
\begin{lemma} \label{lemma: fibre-bundles}
Let $X$ be a smooth manifold without boundary and let $M$ be a closed $P$-manifold. 
There is a one-to-one correspondence between smooth maps, 
$$X \longrightarrow \mathcal{M}_{P}(M)$$ 
and closed $P$-submanifolds 
$$E \subset X\times\R_{+}\times\R^{\infty}\times\R^{p+m}$$
for some $n \in \N$, such that the projection $\pi: E \longrightarrow X$ is a smooth fibre-bundle bundle with fibre $M$ and structure group $\Diff_{P}(M)$. 
\end{lemma}
\begin{proof}
Let $f: X \longrightarrow \mathcal{M}_{P}(M)$ be a smooth map. 
By the discussion above, the pull-back $f^{*}\hat{\mathcal{E}}_{P}(M)$ comes with a canonical embedding into $X\times\R_{+}\times\R^{\infty}\times\R^{p+m}$ such that the projection onto $X$ is fibre-bundle with structure group $\Diff_{P}(M)$ and fibre $M$. 
This describes one direction of the correspondence. 

In the other direction, let $E \subset X\times\R_{+}\times\R^{\infty}\times\R^{p+m}$ be a $P$-submanifold such that the projection onto $X$, $\pi: E \longrightarrow X$ is a fibre-bundle with fibre $M$ and structure group $\Diff_{P}(M)$. 
We obtain a map 
\begin{equation} \label{eq: classifying map}
\xymatrix{
X \longrightarrow \mathcal{M}_{P}(M), \quad x \mapsto E_{x}
}
\end{equation}
where $E_{x} \subset \{x\}\times\R_{+}\times\R^{\bar{n}}\times\R^{p+m}$ is the fibre of the projection $\pi$ over $x \in X$. 
It follows easily that (\ref{eq: classifying map}) is the inverse to the correspondence given by $f \mapsto f^{*}(\hat{\mathcal{E}}_{P}(M))$ described above. 
\end{proof}
We have a similar lemma for $P$-manifold cobordism triples which is proven in the same way. 
\begin{lemma} \label{lemma: fibre-bundles 2}
Let $X$ be a smooth manifold without boundary and let $(W; M_{a}, M_{b})$ be a $P$-manifold cobordism triple. 
There is a one-to-one correspondence between smooth maps, 
$$X \longrightarrow \mathcal{M}_{P}(W; M_{a}, M_{b})$$ 
and $P$-subcobordisms
$$E \subset X\times[0,1]\times\R_{+}\times\R^{\infty}\times\R^{p+m}$$
such that the projection $\pi: E \longrightarrow X$ is a smooth fibre-bundle bundle with fibre $W$ and structure group $\Diff_{P}(W; M_{a}, M_{b})$. 
\end{lemma}

\section{The Cobordism Category of $P$-Manifolds}

\subsection{The Cobordism Category $\mathbf{Cob}^{P}_{d+1}$}
  \label{Category}
  Let $i_{P}: P \longrightarrow  \R^{p+m}$
 be the embedding specified in (\ref{eq: chosen embedding of P}) used to construct the spaces of embeddings in the previous section. 
  We now give a rigorous construction of the category $\mathbf{Cob}^{P}_{d+1}$ that was discussed in the introduction.
  
An object of $\mathbf{Cob}^{P}_{d+1}$ is a pair $(M, a)$
  where $a \in \R$ and 
  $$M \subseteq \{a\}\times\R_{+}\times\R^{\infty}\times\R^{p+m}$$ 
  is a closed $d$-dimensional $P$-submanifold. 
The space of objects is topologized by the identification 
\begin{equation} \label{eq: top category 1}
\Ob(\mathbf{Cob}^{P}_{d+1}) \; = \; \bigsqcup_{M}\left(\mathcal{M}_{P}(M)\times\R\right)
\end{equation}
where the disjoint union is taken over the diffeomorphism classes of closed $P$-manifolds of dimension $d$.

 A non-identity morphism of $\mathbf{Cob}^{P}_{d+1}$ from
  $(M_{a}, a)$ to $(M_{b}, b)$ is a triple $(W; a, b)$ with $(a, b)
  \in \R^{2}_{<} := \{(a, b) \in \R^{2} \; | \; a < b\}$ and $W \subseteq [a,b]\times\R_{+}\times\R^{\infty}\times\R^{p+m}$ 
  is a $(d+1)$-dimensional, compact, $i_{P}$-submanifold such that, 
  $$M_{a} = W\cap(\{a\}\times\R_{+}\times\R^{\infty}\times\R^{p+m}) \quad \text{and} \quad M_{b} = W\cap(\{b\}\times\R_{+}\times\R^{\infty}\times\R^{p+m}).$$
  It follows that $(W; M_{a}, M_{b})$ is a $P$-manifold bordism triple. 
The morphisms $(W_{1}; a, b)$ and $(W_{2}; b, c)$
can be composed if
  $$W_{1} \cap (\{b\}\times\R^{\infty}\times\R_{+}\times\R^{p+m}) \; = \; W_{2}\cap (\{b\}\times\R^{\infty}\times\R_{+}\times\R^{p+m}).$$
  In this case the composition is given by is $(W_{1}\cup W_{2}; a,
  c)$. The collars from condition ii. of Definition \ref{defn: i-P maps} ensure that this 
   union is a smooth 
  submanifold with a canonical smooth structure induced by the smooth structure on the ambient space. 
  The identity morphisms are
  identified with the space of objects.
  The space of morphisms is topologized by the identification 
  \begin{equation} \label{eq: top category 2}
   \Mor(\mathbf{Cob}^{P}_{d+1}) \; = \; \Ob(\mathbf{Cob}^{P}_{d+1}) \sqcup \bigg(\bigsqcup_{(W; M_{a}, M_{b})}\left(\R^{2}_{<}\times\mathcal{M}_{P}(W; M_{a}, M_{b})\right)\bigg).
   \end{equation}
The disjoint union (on the right) is taken over the diffeomorphism classes of $(d+1)$-dimensional $P$-manifold bordism triples. 
It follows easily from (\ref{eq: top category 1}) and (\ref{eq: top category 2}) that composition, and the source and target maps are continuous and thus $\mathbf{Cob}^{P}_{d+1}$ defined in this way is a topological category.
   
\section{Sheaf Models}
\label{D Sheaf}
\subsection{A recollection from \cite{MW 07} of sheaves} \label{subsection: Recollection of Sheaves}
Let $\mathcal{X}$ denote the category of smooth manifolds without
boundary with morphisms given by smooth maps.  By a sheaf
on $\mathcal{X}$ we mean a contravariant functor $\mathcal{F}$ from
$\mathcal{X}$ to \textbf{Sets} which satisfies the following
condition. For any open covering $\{U_{i} \; | \; i \in \Lambda
\}$ of some $X \in \Ob(\mathcal{X})$, and every collection $s_{i} \in
\mathcal{F}(U_{i})$ satisfying 
$$s_{i}\mid_{U_{i}\cap U_{j}} = s_{j}\mid_{U_{i}\cap U_{j}} \quad  \text{for all $i, j \in \Lambda$,}$$
there is a
unique $s \in \mathcal{F}(X)$ such that $s\mid_{U_{i}} = s_{i}$ for
all $i \in \Lambda$. 

\begin{defn}
Let $\mathcal{F}$ be a sheaf on $\mathcal{X}$. Two elements $s_{0}$
and $s_{1}$ of $\mathcal{F}(X)$ are said to be concordant if there
exists $s \in \mathcal{F}(X\times\R)$ that agrees with $\text{pr}^{*}(s_{0})$
in an open neighborhood of $X\times(-\infty, 0]$ and agrees with
$\text{pr}^{*}(s_{1})$ in an open neighborhood of $X\times[1,\infty)$, where
  $\text{pr}: X\times\R\to X$ is the projection onto the first factor.
\end{defn} 
We denote the set of concordance
classes of $\mathcal{F}(X)$ by $\mathcal{F}[X]$.  The correspondence
$X \mapsto \mathcal{F}[X]$ is clearly functorial in $X$. 

\begin{defn}For a sheaf $\mathcal{F}$ we define  the \textit{representing space} $|\mathcal{F}|$ to be the geometric realization of the simplicial set given by the formula
$k \; \mapsto \; \mathcal{F}(\triangle^{k}_{e})$
where 
$$\triangle^{k}_{e} := \{(x_{0}, x_{1}, \dots, x_{n}) \in \R^{n+1} \; | \; \sum x_{i} = 1\}$$ 
is the standard extended $k$-simplex. \end{defn}

From this definition it is easy to see that any map of sheaves
$\mathcal{F} \rightarrow \mathcal{G}$ induces a map between the
representing spaces $|\mathcal{F}| \rightarrow |\mathcal{G}|$.
\begin{defn} Let $\mathcal{F}$ be a sheaf on $\mathcal{X}$. Let
$A \subset X$ is a closed subset, and let $s$ be a germ near $A$,
  i.e., $s \in \colim_{U}\mathcal{F}(U)$ with $U$ ranging over all
  open sets containing $A$ in $X$. Then we define $\mathcal{F}(X, A;
  s)$ to be the set of all $t \in \mathcal{F}(X)$ whose germ near $A$
  coincides with $s$. Then two elements $t_{0}$ and $t_{1}$ are
  concordant \textit{relative} to $A$ and $s$ if they are related by a
  concordance whose germ near $A$ is the constant concordance equal to 
  $s$. The set of such relative concordance classes is denoted by
  $\mathcal{F}[X, A; s]$.
\end{defn} 
Any element
  $z \in \mathcal{F}(\text{pt.})$ determines a point in $|\mathcal{F}|$
  which we also denote by $z$. For any $X \in \Ob(\mathcal{X})$,
  such an element $z \in \mathcal{F}(\text{pt.})$ determines an element, which we give the same
  name, $z \in \mathcal{F}(X)$, by pulling back by the constant map. In
  \cite[2.17]{MW 07} it is proven that there is a natural bijection
  of sets
\begin{equation} \label{concordance relative iso} [(X, A), (|\mathcal{F}|, z)] \cong \mathcal{F}[X, A; z] .
\end{equation}
Here the set on the left hand side is the set of homotopy classes of
maps of pairs. The non-relative case of this isomorphism with $A$ the
empty set holds as well.
Using these observations we define the homotopy
  groups of a sheaf by setting
\begin{equation} \label{eq: sheaf homotopy group} \pi_{n}(\mathcal{F}, z) := \mathcal{F}[S^{n}, \text{pt.}; z]. 
\end{equation}
By (\ref{concordance relative iso}) we get $\pi_{n}(\mathcal{F}, z)
\cong \pi_{n}(|\mathcal{F}|, z)$ for any choice of $z \in
\mathcal{F}(\text{pt.})$.
Using this definition of homotopy group, a map of sheaves is said to
be a \textit{weak equivalence} if it induces a isomorphisms on all homotopy groups.

\subsection{The Sheaf $\mathbf{D}^{P}_{d+1}$} In this section we define a sheaf $\mathbf{D}^{P}_{d+1}$ on $\mathcal{X}$. It will be seen in Section \ref{The Classifying Space}  that the representing space $|\mathbf{D}^{P}_{d+1}|$ is weakly homotopy equivalent to the classifying space $B\mathbf{Cob}^{P}_{d+1}$. 

For what follows, let $i_{p}: P \longrightarrow \R^{p+m}$ be the embedding specified in (\ref{eq: chosen embedding of P}) that was used in the construction of $\mathbf{Cob}^{P}_{d+1}$ and the mapping spaces of Section \ref{eq: P-embeddings}. 
For an integer $n$ we will use the same notation, 
$$\bar{n} = n - p - m - 1$$
as was used in the previous sections. 
Before defining $\mathbf{D}^{P}_{d+1}$ we must fix some more terminology and notation. 

 Let $d$ and $n$ be non-negative integers and let $X \in \mathcal{X}$. 
We will need to consider $(d+1)$-dimensional, $P$-submanifolds,
$$W \subset   X\times\R\times\R_{+}\times\R^{d+\bar{n}}\times\R^{p+m}$$
with $\partial_{0}W = \emptyset$, where $W$ is not assumed to be compact. 
Recall, this means that:
\begin{enumerate}
\item[i.] $\partial_{1}W$ is embedded in $X\times\R\times\{0\}\times\R^{d+\bar{n}}\times\R^{p+m}$ with a collar as in condition ii. of Definition \ref{defn: i-P maps}.
\item[ii.] There is the factorization, $\partial_{1}W = \beta_{1}W\times i_{P}(P)$ where $\beta_{1}W \subset X\times\R\times\{0\}\times\R^{d+\bar{n}}$ is a submanifold. 
\end{enumerate}
In other words, the inclusion map $W \hookrightarrow X\times\R\times\R_{+}\times\R^{d+\bar{n}}\times\R^{p+m}$ is an element of the space $\mathcal{E}_{P, d+n}(X\times\R)$.

Denote by 
\begin{equation} \label{eq: projection}
(\pi, f): W \longrightarrow X\times \R
\end{equation}
the restriction of the projection, $X\times \R \times\R_{+}\times 
\R^{d+\bar{n}}\times\R^{p+m} \longrightarrow X\times\R$
to $W$. 
It follows from Definition \ref{defn: P-map} that 
$$\pi: W \rightarrow X, \quad f: W \rightarrow \R, \quad \text{and} \quad (\pi, f): W \rightarrow X\times \R$$ 
are all smooth $P$-maps. 
Notice that if $K \subset X$ is a submanifold and $\pi$ is transverse to $K$, then 
$$\pi^{-1}(K) \subset K\times \R\times\R_{+}\times\R^{d+\bar{n}}\times\R^{p+m}$$
is a $P$-submanifold.
We are now ready to define $\mathbf{D}^{P}_{d+1}$.
  \begin{defn} \label{defn: D sheaf} Let $X \in \Ob(\mathcal{X})$. 
For non negative integers $n$ and $d$ we
define $\mathbf{D}^{P}_{d+1, n}(X)$ to be the set of $(d+1+\dim(X))$-dimensional 
$P$-submanifolds 
$$W \; \subset \; X\times \R\times\R_{+}\times\R^{d+\bar{n}}\times\R^{p+m},$$ 
with $\partial_{0}W = \emptyset$,
 which satisfy the following:
 \begin{enumerate} 
\item[i.] The map $\pi: W \longrightarrow X$ is a $P$-submersion. 
\item[ii.] The map $(\pi, f): W \longrightarrow X\times \R$ is a proper $P$-map (recall that a map is proper if the pre-image of any compact subset is compact).
\end{enumerate}
\end{defn} 
Let $X, Y \in \mathcal{X}$ and let $f: X \longrightarrow Y$ be a smooth map. 
If $W \in \mathbf{D}^{P}_{d+1, n}(Y)$, then the pullback,
$$f^{*}(W) = \{(x, w) \in X\times W \; | \; f(x) = \pi(w)\}$$
naturally embeds in $X\times\R\times\R_{+}\times\R^{d+ \bar{n}}\times\R^{p+m}$
so that the projections onto $X$ is a submersion,
and thus yields an element of $\mathbf{D}^{P}_{d+1, n}(X)$ (see the discussion of pull-backs in \cite[Sections 2.2]{GMTW 09} for details). 
The correspondence $W \mapsto f^{*}W$ yields a map $\mathbf{D}^{P}_{d+1, n}(Y) \longrightarrow \mathbf{D}^{P}_{d+1, n}(X)$.
In this way, it follows that the assignment $X \mapsto \mathbf{D}^{P}_{d+1, n}(X)$ is a contraviant functor from $\mathcal{X}$ to $\mathbf{Sets}$. 
It is easy to verify that this functor satisfies the sheaf condition from Section \ref{subsection: Recollection of Sheaves}. 
 
To eliminate dependence on $n$, we define $\mathbf{D}^{P}_{d+1} := \displaystyle\colim_{n \to \infty}\mathbf{D}^{P}_{d+1, n}$ where the colimit is understood to be taken in the category of sheaves on $\mathcal{X}$. 
Concretely, $\mathbf{D}^{P}_{d+1}(X)$ is the set of all $(d+1)$-dimensional $P$-submanifolds $W \; \subset \; X\times \R \times\R_{+}\times\R^{d + \infty}\times\R^{p+m}$ satisfying conditions i. and ii. of Definition \ref{defn: D sheaf}, such that for any compact subset $K \subset X$ there exists $n \in \N$ such that, 
$$\pi^{-1}(K) \subset K\times\R\times\R_{+}\times\R^{d+\bar{n}}\times\R^{p+m}$$ 
(compare to \cite[Definition 3.3]{GMTW 09}). 
It follows from this characterization that for each non-negative integer $k$, the natural map $\displaystyle\colim_{n \to \infty}\mathbf{D}^{P}_{d+1, n} \longrightarrow \mathbf{D}^{P}_{d+1}$ induces an isomorphism on all homotopy groups and thus there is a homotopy equivalence,
$|\mathbf{D}^{P}_{d+1}| \simeq \displaystyle\colim_{n \to \infty}|\mathbf{D}^{P}_{d+1, n}|.$

The following lemma is proven in the same way as \cite[2.20]{MW 07}.
\begin{lemma} \label{lemma: concordance representation} 
For all $X \in \mathcal{X}$, 
every
  concordance class in $\mathbf{D}^{P}_{d+1, n}(X)$ has a
  representative  
  $$W \subset X\times\R\times\R_{+}\times\R^{d +\bar{n}}\times\R^{p+m}$$
  such that the map $f: W \longrightarrow \R$ a bundle projection, and thus there is a diffeomorphism $W \cong f^{-1}(0)\times \R.$
\end{lemma}

\section{A cofibre of Thom-spectra}
\label{Thom Spectrum}
\subsection{The spectrum $\MT_{P}(d+1)$} In this section we define a
spectrum $\MT_{P}(d+1)$ as the cofibre of a
map between $\Sigma^{-1}\MT(d-p)$ and $\MT(d+1)$,
where $\MT(d+1)$ is the spectrum defined in \cite{GMTW 09}. We use the
same notation for Grassmannian manifolds and their canonical bundles as
in \cite{GMTW 09}.
 
Let $i_{P}: P \hookrightarrow \R^{p+m}$ be the embedding from (\ref{eq: chosen embedding of P}). 
Denote by $G(p, m)$ the Grassmannian manifold of $p$-dimensional vector subspaces of $\R^{p+m}$.
Denote by $U_{p,m} \rightarrow G(p, m)$ the canonical vector bundle (which has fibre-dimension $p$) and denote by $U^{\perp}_{p,m} \rightarrow G(p, m)$ the orthogonal compliment to $U_{p,m}$, which has fibre-dimension $m$. 
The normal bundle $N_{P} \rightarrow P$ associated to the embedding $i_{P}: P \hookrightarrow \R^{p+m}$ has Gauss map,
\begin{equation} \label{eq: Gauss map 1}
\xymatrix{
  N_{P} \ar[rr]^{\hat{\gamma}} \ar[d] && U_{p,m}^{\perp} \ar[d] \\
  P \ar[rr]^{\gamma} && G(p,m)}
  \end{equation}
  which induces a map of the Thom
spaces,
$\Th(\hat{\gamma}): \Th(N_{p}) \longrightarrow \Th(U_{p,m}^{\perp})$.
Fix an embedding of the normal bundle $N_{P}$ as a tubular neighborhood, 
\begin{equation} \label{eq: P tubular nbh} 
e_{P}:N_{P} \hookrightarrow \R^{p+m}.
\end{equation}
This tubular neighborhood together with $\Th(\hat{\gamma})$ yields the Pontryagin-Thom map
\begin{equation} \label{mult P} 
\xymatrix{S^{p+m} \ar[rr]^{c_{P}} && \Th(U_{p,m}^{\perp})}.
\end{equation}
We now consider the standard multiplication-map
$$\xymatrix{\mu : G(d-p,n)\times G(p, m) \ar[rr] && G(d,n+m)}$$
given by $(V,W)\mapsto V\times W.$
The multiplication map $\mu$ is covered by a bundle map
$$
\xymatrix{U_{d-p,n-m}^{\perp}\times U_{p,m}^{\perp} \ar[d] \ar[rr]^{\ \ \ \ \hat{\mu}} && U_{d,n}^{\perp} \ar[d] \\
  G(d-p,n-m)\times G(p,m)  \ar[rr]^{\ \ \ \ \ \mu} && G(d,n)}$$ 
 which induces,
$$ \xymatrix{\Th(U_{d-p,n-m}^{\perp})\wedge \Th(U_{p,m}^{\perp}) \ar[rr]^{\ \ \ \ \Th(\hat{\mu})} && \Th(U_{d,n}^{\perp})}.
$$

Putting this together with $c_{P}$ from (\ref{mult P}) we define:
\begin{equation} \label{P mult n}
\xymatrix{ \tau_{P,n} := \Th(\hat{\mu})\circ (Id\wedge c_{P}): \Th(U_{d-p,n-m}^{\perp})\wedge S^{p+m} \ar[rrr] &&& \Th(U_{d,n}^{\perp}).}
 \end{equation}
As defined in \cite{GMTW 09}, $\Th(U_{d,n}^{\perp})$ is the
 $(d+n)$th space of the spectrum $\MT(d)$. The structure maps in this
 spectrum $\MT(d)$ come from applying $\Th(\underline{\hspace{.3cm}})$ to the bundle map
  $$\xymatrix{
  U_{d,n}^{\perp}\oplus\epsilon^{1} \ar[d] \ar[rr]^{\hat{i_{n}}} && U_{d,n+1}^{\perp} \ar[d]\\
  G(d,n) \ar[rr]^{i_{n}} && G(d,n+1)}
$$
where the map $i_{n}$ is the standard embedding and $\epsilon^{1}$ is the trivial line bundle.
The map from (\ref{P mult n}),
 yields a map of spectra
\begin{equation} \label{P Mult}
\xymatrix{\tau_{P}: \MT(d-p) \ar[rr] && \MT(d)}.
\end{equation}
Here we are taking into account that the spectrum with $(d+p)$th space
$\Th(U_{d-p,n-m}^{\perp})\wedge S^{p+m}$ is homotopy equivalent to
$\MT(d-p)$. 
Now, consider the map
$G(d,n) \longrightarrow G(d+1,n)$
given by sending a $d$-dimensional vector subspace $V \subset \R^{p+n}$ to the subspace $\R\times V \subset \R\times\R^{d+n}$. 
This is covered by a bundle map $U_{d,n}^{\perp} \longrightarrow U_{d+1,n}^{\perp}$.
This induces a map on Thom-spaces 
\begin{equation} \label{eq: adjacent thom space map}
j_{d, n}: \Th(U_{d,n}^{\perp}) \longrightarrow \Th(U_{d+1,n}^{\perp})
\end{equation}
and in turn a map of spectra which we denote by 
\begin{equation} \label{eq: adjacent spectra map}
\hat{j}_{d}: \Sigma^{-1}\MT(d) \longrightarrow  \MT(d+1).
\end{equation}
Consider the map of spectra given by the composition,
$$
\xymatrix{
\Sigma^{-1}\MT(d-p) \ar[rr]^{\Sigma^{-1}\tau_{P}} &&  \Sigma^{-1}\MT(d) \ar[rr]^{\hat{j}_{d}} && \MT(d+1).
}
$$
Finally we define,
\begin{equation} 
\xymatrix{
\MT_{P}(d+1):= \Cofibre(\hat{j}_{d} \circ(\Sigma^{-1}\tau_{P})). 
}
\end{equation} 
The spectrum $\MT_{P}(d+1)$ defined above is our primary spectrum of interest and is the spectrum that appears in the statement of Theorem \ref{thm: Main} from the introduction. 
Applying the \textit{infinite loop-space} functor to the cofibre sequence 
$$\Sigma^{-1}\MT(d-p) \longrightarrow \MT(d+1) \longrightarrow \MT_{P}(d+1)$$
yields a homotopy-fibre sequence
\begin{equation} \label{eq: Loops Fibre Sequence} \Omega^{\infty}\MT(d+1) \longrightarrow \Omega^{\infty}\MT_{P}(d+1) \longrightarrow \Omega^{\infty}\MT(d-p). \end{equation}
\begin{remark} \label{remark: homotopy invariance under bordism class}
Since the map of spectra $\tau_{P}: \MT(d-p) \longrightarrow \MT(d)$ is induced by the Pontryagin-Thom construction applied to the embedding $i_{P}: P \hookrightarrow \R^{p+m}$, it follows that the homotopy class of $\tau_{P}$ is an invariant of the cobordism class of the manifold $P$. 
It follows that if $P'$ is a manifold which is cobordant to $P$ then the spectrum $\MT_{P'}(d+1)$ is homotopy equivalent to $\MT_{P}(d+1)$.
\end{remark}

\subsection{A filtration of $\MO_{P}$} \label{A Filtration}
We now describe how the the spectrum $\MT_{P}(d+1)$ constructed in the previous section relates to the spectrum $MO_{P}$, which classifies the homology theory associated to the cobordism groups of $P$-manifolds. 
There is a direct system of spectra 
\begin{equation}\label{filtration} \xymatrix{\cdots \ar[r] &
    \Sigma^{(d-1)}\MT(d-1) \ar[r] & \Sigma^{d}\MT(d) \ar[r] &
    \Sigma^{(d+1)}\MT(d+1) \ar[r] & \cdots} 
\end{equation}
where the $d$-th map is the $d$th suspension of the map defined in (\ref{eq: adjacent spectra map}). The direct limit is
a spectrum which we denote by $\MTO$.
The following lemma is proven in \cite[Page 14]{GMTW 09}. 
We provide the proof here for completeness.
\begin{lemma} \label{Filtration Lemma} There is a homotopy equivalence
$\MTO \simeq \MO.$
\end{lemma}
\begin{proof}
  There is a homeomorphism $G(d,n) \longrightarrow G(n,d)$ given by $V
  \mapsto V^{\perp}$. This map is covered by a bundle isomorphism
  $U_{d,n}^{\perp} \longrightarrow U_{n,d}$ and thus yields maps
$$\xymatrix{\Th(U_{d,n}^{\perp}) \ar[rr]^{\perp}_{\cong} && \Th(U_{n,d}) \ar[rr]^{i} && \Th(U_{n,\infty})},$$
where $\Th(U_{n,\infty}) := \colim_{d\to\infty}\Th(U_{n,d})$. 
The space $\Th(U_{n,\infty})$ is the $n$th space of the spectrum $\MO$, thus the above map induces a map of spectra,
$\Sigma^{d}\MT(d) \longrightarrow \MO$
(or a degree $d$ map $\MT(d) \rightarrow \MO$). 
The space $\Th(U_{n,\infty})$ is the $n$-th space in the spectrum $\MO$. Now, $\Th(U_{n,\infty})$ can be built out of $\Th(U_{n,d})$ by attaching cells of dimension greater than dimension $n+d$. This implies that $\Sigma^{d}\MT(d) \longrightarrow \MO$ induces an isomorphism on $\pi_{k}$ for $k < d$ and a surjection for $k = d$. This proves that $\MTO \simeq \MO$. 
\end{proof}
Since $\pi_{d-1}\MO \cong \Omega_{d-1}$ where $\Omega_{d-1}$ is
the cobordism group of unoriented $(d-1)$-dimensional manifolds, the
above lemma implies that $\pi_{-1}\MT(d) \cong \Omega_{d-1}$.
For each $d$, the diagram
$$\xymatrix{
\Sigma^{(d+p)}\MT(d) \ar[ddd]_{\Sigma^{(d+1+p)}[\hat{j}_{d} \circ (\Sigma^{-1}\tau_{P})]} \ar[rrrr]^{\Sigma^{(d+p)}\hat{j}_{d}} &&&& \Sigma^{(d+1+p)}\MT(d+1) \ar[ddd]^{\Sigma^{(d+2+p)}[\hat{j}_{d+1} \circ (\Sigma^{-1}\tau_{P})]}\\ 
\\
\\
\Sigma^{(d+1+p)}\MT(d+1+p) \ar[rrrr]^{\Sigma^{(d+1+p)}\hat{j}_{d+1+p}} &&&& \Sigma^{(d+2+p)}\MT(d+2+p)}$$
commutes up to homotopy. Passing to the cofibres of the vertical maps induces a map of spectra,
\begin{equation} \label{cofibre filtration} 
\xymatrix{
\Sigma^{d+1+p}\MT_{P}(d+1+p) \ar[rr] && \Sigma^{d+2+p}\MT_{P}(d+2+p)
}
\end{equation}
for each $d$. These maps form a direct system similar to (\ref{filtration}). We
denote the direct limit of this direct system by $\MTO_{P}$.
 \begin{lemma} There is a homotopy equivalence 
 $\MTO_{P} \simeq \MO_{P}$
 where $\MO_{P}$ is the classifying spectrum for the cobordism theory
 $\Omega_{*}^{P}$ for manifolds with type
 $P$-singularities.
 \end{lemma}
 \begin{proof}
   The spectrum $\MO_{P}$ is given as the cofibre of the map 
   $\times P: \Sigma^{p}\MO \longrightarrow \MO$ which is induced by the degree $p$ homomorphism $\Omega_{*} \stackrel{\times P} \longrightarrow\Omega_{*}$. 
   On the level of spectra, this map is defined concretely as follows. 
The map 
$$
\mu: G(n,d)\times G(m,p) \longrightarrow G(n+m, d+p), \quad (V, W) \mapsto V\times W
$$
induces 
$$\mu':
   G(n,\infty)\times G(m,p) \longrightarrow G(n+m, \infty)
$$ 
in the limit as $d \to \infty$. The map $\mu'$ is covered
by a bundle map $U_{n, \infty}\times U_{m,p} \longrightarrow
U_{n+m,\infty}$ which induces a map of Thom spaces,
$$ 
\Th(U_{n, \infty})\wedge\Th(U_{m,p}) \longrightarrow
\Th(U_{n+m,\infty}).
$$ 
The normal bundle $N_{P}$ for $i_{P}(P) \subset \R^{p+m}$ has
Gauss map
$$\xymatrix{
  N_{P} \ar[d] \ar[rr] && U_{m,p} \ar[d]\\
  P \ar[rr] && G(m,p).}$$ We emphasize that this map is
  different than the map (\ref{eq: Gauss map 1}) where the target space was
  $G(p,m)$ with bundle $U_{p,m}^{\perp}$. The Pontryagin-Thom map 
$S^{p+m} \longrightarrow \Th(U_{m,p})$
associated to the Gauss map for $i_{P}(P)$ yields
$$\xymatrix{\Th(U_{n,\infty})\wedge S^{p+m} \ar[rr] &&  \Th(U_{n,\infty})\wedge \Th(U_{m,p}) \ar[rr] && \Th(U_{n+m, \infty}). }$$
Since the spectrum with $(n+m)$-th space equal to
$\Th(U_{n,\infty})\wedge S^{m}$ is equivalent to $\MO$, the 
map above induces a map of spectra, $\Sigma^{p}\MO \longrightarrow \MO$
which defines $\times P$.  Upon inspection, it can be seen that for all $d$ the
following diagram commutes
\begin{equation} \label{P filt diagram}
\xymatrix{
\Sigma^{d+p}\MT(d) \ar[d]^{\hat{j}_{d} \circ \tau_{P}} \ar[rr] && \Sigma^{p}\MO \ar[d]^{\times P} \\
\Sigma^{d+p+1}\MT(d+p+1) \ar[rr] && \MO} 
\end{equation}
where the horizontal maps are induced by the composition,
$$\xymatrix{\Th(U_{d,n}^{\perp}) \ar[rr]^{\perp}_{\cong} && \Th(U_{n,d}) \ar[rr]^{i} && \Th(U_{n,\infty})}.$$
As was used in the proof of Lemma \ref{Filtration Lemma}, the Thom space 
$\Th(U_{n,\infty})$ can be
built out of $\Th(U_{n,d})$ by attaching cells of dimension greater
than $n+d$. This implies that the lower and upper horizontal
maps above in (\ref{P filt diagram}) induce isomorphisms on $\pi_{k}$
for $k < d+p$ and surjections on $\pi_{k}$ for $k = d+p$. By
applying the \textit{Five Lemma} to the long exact sequence on
homotopy groups associated to the cofibres of the vertical maps, we
see that the induced map
$$
\xymatrix{\Sigma^{d+p+1}\MT_{P}(d+p+1) \ar[rr] && \MO_{P}}$$
induces an isomorphism on $\pi_{k}$ for $k < d+p$ and a surjection
on $\pi_{k}$ for $k = d+p$. Taking the direct limit as $d \to
\infty$, we see that $\MTO_{P} \simeq \MO_{P}$. \end{proof}
\begin{corollary}\label{coefficients}
  There is an isomorphism $\pi_{-1}\MT_{P}(d+1) \cong \Omega^{P}_{d}$, where $\Omega^{P}_{d}$ is the classical cobordism-group of $d$-dimensional manifolds with singularities of type $P$. 
\end{corollary}

\subsection{Infinite Loop Spaces}
Our main result, Theorem \ref{thm: Main}, establishes a weak homotopy equivalence, $B\mathbf{Cob}^{P}_{d+1} \simeq \Omega^{\infty-1}\MT_{P}(d+1).$
It is difficult to construct a map from a space directly into the infinite loop-space $\Omega^{\infty-1}\MT_{P}(d+1)$. 
It will be useful for us to construct certain auxiliary models for the homotopy type of $\Omega^{\infty-1}\MT_{P}(d+1)$ which will be easier to map into. 
Recall the maps
$$\xymatrix{
\Th(U_{d-p,n-m}^{\perp})\wedge S^{p+m} \ar[rr]^{ \ \ \ \ \tau_{P, n}} && \Th(U_{d, n}^{\perp}) \ar[rr]^{j_{d, n}} && \Th(U^{\perp}_{d+1, n})
}$$
from (\ref{P mult n}) and (\ref{eq: adjacent thom space map}). 
\begin{defn} \label{defn: product mapping space} 
For non-negative integers $n$ and $d$, we define two spaces 
$$\Omega_{\partial}^{d+n}\widehat{\Th(U_{d+1,n}^{\perp})} \quad \text{and} \quad \Omega_{P}^{d+n}\widehat{\Th(U_{d+1,n}^{\perp})}$$
as follows.
$\Omega_{\partial}^{d+n}\widehat{\Th(U_{d+1,n}^{\perp})}$ is defined
to be the space of pairs $(\hat{f}, f)$ of based maps,
$$\begin{aligned}
\hat{f}: D^{d+n} &\longrightarrow \Th(U_{d+1,n}^{\perp}),\\
f: S^{d+n-1} &\longrightarrow  \Th(U_{d-p,n-m}^{\perp})\wedge S^{p+m},
\end{aligned}$$
which make the diagram
\begin{equation} \label{eq: defining diagram}
\xymatrix{
D^{d+n} \ar[rr]^{\tilde{f}} && \Th(U_{d+1,n}^{\perp}) \\
\\
S^{d+n-1} \ar@{^{(}->}[uu] \ar[rr]^{f} && \Th(U_{d-p,n-m}^{\perp})\wedge S^{p+m}  \ar[uu]_{j_{d, n}\circ \tau_{P,n}}
}
\end{equation}
commute, where the left-vertical map is the standard inclusion. 
Now, let 
$$\alpha: S^{d+n-1} \stackrel{\cong} \longrightarrow S^{d + \bar{n}}\wedge S^{p+m}$$ 
be the standard identification (where $\bar{n} = n - p - m - 1$ as in the previous section).
The space  $\Omega_{P}^{d+n}\widehat{\Th(U_{d+1,n}^{\perp})}$ is defined to be the subspace of $\Omega_{\partial}^{d+n}\widehat{\Th(U_{d+1,n}^{\perp})}$ consisting of all pairs $(\widehat{f}, f)$ for which there exists a map 
$$f_{0}: S^{d  + \bar{n}} \longrightarrow \Th(U_{d-p, n-m}^{\perp})$$
such that $f: S^{d-1+n} \longrightarrow \Th(U_{d-p,n-m}^{\perp})\wedge S^{p+m}$ has the factorization,
$$\xymatrix{
S^{d+n-1} \ar[rr]^{\alpha \ \ \ \ }_{\cong \ \ \ \ } && S^{d + \bar{n}}\wedge S^{p+m} \ar[rrr]^{f_{0}\wedge Id_{S^{p+m}}} &&& \Th(U_{d-p, n-m}^{\perp})\wedge S^{p+m}.
}$$ 
It follows that the map $f_{0}$ is uniquely determined. 
\end{defn}
We then define,
\begin{equation} 
\begin{aligned}
\Omega_{P}^{\infty -1}\widehat{\Th(U_{d+1,\infty}^{\perp})} &:= 
\colim_{n\to \infty} \Omega_{P}^{d+n}\widehat{\Th(U_{d+1,n}^{\perp})}, \\
\Omega_{\partial}^{\infty -1}\widehat{\Th(U_{d+1,\infty}^{\perp})} &:= 
\colim_{n\to \infty} \Omega_{\partial}^{d+n}\widehat{\Th(U_{d+1,n}^{\perp})}.
\end{aligned}
\end{equation}
\begin{proposition} \label{prop: homotopy eq of mapping spaces 1}
The natural embedding 
$\Omega_{P}^{\infty -1}\widehat{\Th(U_{d+1,\infty}^{\perp})} \longrightarrow \Omega_{\partial}^{\infty -1}\widehat{\Th(U_{d+1,\infty}^{\perp})}$
is a homotopy equivalence. 
\end{proposition}
\begin{proof}
For each $n$, the space $\Omega^{d+n}_{P}\widehat{\Th(U_{d+1,n}^{\perp})}$ can be realized as the pullback,
 \begin{equation} \label{mapping space pullback}
 \xymatrix{
\Omega^{d+n}_{P}\widehat{\Th(U_{d+1,n}^{\perp})} \ar[d]_{r_{0}} \ar@{^{(}->}[rrr] &&& \Omega^{d+n}_{\partial}\widehat{\Th(U_{d+1,n}^{\perp})} \ar[d]_{r} \\
 \Omega^{d+\bar{n}}\Th(U^{\perp}_{d-p,n-m}) \ar[rrr]^{\underline{\hspace{.3cm}}\wedge Id_{S^{p+m}}\ \ \ } &&&  \Omega^{d+n-1}(\Th(U_{d-p,n-m}^{\perp})\wedge S^{p+m}),}
 \end{equation}
 where $r(\hat{f}, f) = f$ and $r_{0}(\hat{f}, f) = f_{0}$. 
 The bottom horizontal map is the standard $(p+m)$-fold suspension map. 
 The top horizontal map in the diagram is the inclusion. 
 It will suffice to show that this upper-horizontal map is highly connected and that its connectivity approaches infinity as $n \to\infty$.
 
 Now, the map $r$ is a Serre-fibration.
 It follows from this that the pull-back square (\ref{mapping space pullback}) is \textit{homotopy cartesian}. 
The Thom-space $\Th(U^{\perp}_{d-p,n-m})$ is
$(n-m-1)$-connected and so its connectivity approaches $\infty$ with $n$. 
The \textit{Freudenthal suspension theorem} implies that the connectivity of the 
bottom horizontal map, $\underline{\hspace{.3cm}}\wedge Id_{S^{p+m}}$ of (\ref{mapping space pullback}) approaches $\infty$ (notice that as $n$ increases without bound, the integers $p$ and $m$ are held fixed). 
Since the diagram is a homotopy pull-back square, it follows from this that the connectivity of the upper horizontal map tends to $\infty$ with $n$. 
This implies the proposition and completes the proof. 
\end{proof}
We now compare $\Omega_{P}^{\infty -1}\widehat{\Th(U_{d+1,\infty}^{\perp})}$ to the infinite loop-space $\Omega^{\infty -1}\MT_{P}(d+1)$.
For each $n$ there is a map
$\sigma_{n}: \Omega_{\partial}^{d+n}\widehat{\Th(U_{d+1,n}^{\perp})} \longrightarrow \Omega^{d+n}\Cofibre(j_{d, n}\circ \tau_{P,n})$
defined by sending a pair of maps
$$\hat{f}: D^{d+n} \longrightarrow \Th(U_{d+1,n}^{\perp}), \quad f: S^{d+n-1} \longrightarrow  \Th(U^{\perp}_{d-p, n-m})\wedge S^{p+m}$$
which make diagram (\ref{eq: defining diagram}) commute, to its induced map, $D^{d+n}/S^{d+n-1} \longrightarrow \Cofibre(j_{d, n}\circ \tau_{P,n})$. 
In the limit $n \to \infty$, the maps $\sigma_{n}$ induce
\begin{equation} \label{eq: mapping space eq}
\sigma: \Omega_{\partial}^{\infty -1}\widehat{\Th(U_{d+1,\infty}^{\perp})} \longrightarrow \Omega^{\infty -1}\MT_{P}(d+1).
\end{equation}
\begin{proposition} \label{prop: homotopy eq of mapping spaces 2}
The map $\sigma$ of (\ref{eq: mapping space eq}) is a homotopy equivalence. 
\end{proposition}
\begin{proof}
For each $k$, we have a commutative diagram,
$$\xymatrix{
\pi_{k}(\Omega_{\partial}^{d+n}\widehat{\Th(U_{d+1,n}^{\perp})}) \ar[dd] \ar[rrr]^{(\sigma_{n})_{*}} &&& \pi_{k}(\Omega^{d+n}\Cofibre(j_{d, n}\circ \tau_{P,n}))\ar[dd] \\
\\
\pi_{k+d+n}(\Th(U_{d+1,n}^{\perp}), \; \Th(U^{\perp}_{d-p, n-m})\wedge S^{p+m}) \ar[rrr]^{(\sigma^{k}_{n})} &&& \pi_{k+d+n}(\Cofibre(j_{d, n}\circ \tau_{P,n}))
}$$
where the vertical maps are given by adjunction, the top horizontal map is induced by $\sigma_{n}$, and the bottom horizontal map $\sigma^{k}_{n}$, is induced by sending a pair of maps
$$\hat{f}: D^{k+d+n} \longrightarrow \Th(U_{d+1,n}^{\perp}), \quad f: S^{k+d+n-1} \longrightarrow  \Th(U^{\perp}_{d-p, n-m})\wedge S^{p+m}$$
which make diagram (\ref{eq: defining diagram}) commute, to its induced map, 
$$S^{k+d+n} \cong D^{k+d+n}/S^{k+d+n-1} \longrightarrow \Cofibre(j_{d, n}\circ \tau_{P,n}).$$
The space $\Th(U^{\perp}_{d-p, n-m})\wedge S^{p+m}$ is $(n+p-1)$-connected and the map 
$$j_{d, n}\circ \tau_{P,n}: \Th(U^{\perp}_{d-p, n-m})\wedge S^{p+m} \longrightarrow \widehat{\Th(U_{d+1,n}^{\perp})}$$
is at least $n-1$-connected. 
It follows from \cite[Proposition 4.28]{H 01} that the bottom-horizontal homomorphism, $\sigma_{n}^{k}$, is an isomorphism when $k \leq n + p - d - 2$. 
Commutativity of the above diagram then implies that $(\sigma_{n})_{*}$ is an isomorphism when $k \leq n + p - d - 2$ as well. 
Passing to the limit $n \to \infty$ then yields the result. 
This completes the proof of the proposition.
 \end{proof}

Combining Propositions \ref{prop: homotopy eq of mapping spaces 1} and \ref{prop: homotopy eq of mapping spaces 2} yields:
\begin{corollary} \label{cor: infinite loop-space eq}
The natural map $\Omega_{P}^{\infty -1}\widehat{\Th(U_{d+1,\infty}^{\perp})} \longrightarrow \Omega^{\infty -1}\MT_{P}(d+1)$
is a homotopy equivalence. 
\end{corollary}

\section{The main Theorem}
\label{The Main Theorem}
This section is devoted to proving the following theorem.
\begin{theorem}\label{Main Theorem}  There is a
homotopy equivalence
$|\mathbf{D}^{P}_{d+1}| \simeq \Omega^{\infty-1}\MT_{P}(d+1).$
\end{theorem}
This theorem is proven in a way similar to \cite[Theorem 3.4]{GMTW 09}. 
\subsection{Isomorphism of concordance class functors} \label{A Natural Transformation}
First we prove the following lemma:
\begin{lemma} \label{lemma: pontryagin thom}
For $X \in \Ob(\mathcal{X})$ and $n \in \N$, there is a natural map 
$$T_{n}: \mathbf{D}^{P}_{d+1,n}[X] \longrightarrow [X, \Omega_{P}^{d+n}\widehat{\Th(U_{d+1,n}^{\perp})}].$$
\end{lemma}
\begin{proof}
Let $X \in \Ob(\mathcal{X})$. We construct the map $T_{n}$ as follows.
 Let 
 $$W \subset X\times \R \times\R_{+}\times\R^{d +\bar{n}}\times\R^{p+m}$$
be an element of $\mathbf{D}^{P}_{d+1,n}(X)$. 
We may assume that $f$ is transverse as a smooth $P$-map to $0 \in \R$, in the sense of Definition \ref{defn: P-transversality}. 
This means that both $f$ and $f_{\beta}$ are transverse to $0$. 
It follows from Proposition \ref{prop: P-transverse} that $f^{-1}(0)$ has the structure of a $P$-submanifold of $X\times\{0\}\times\R_{+}\times\R^{d +\bar{n}}\times\R^{p+m}$. 
We have,
$$M := f^{-1}(0), \quad \partial_{1}M := (f|_{\partial_{1}W})^{-1}(0), 
\quad  \beta_{1}M := (f|_{\beta_{1}W \times \{0\}})^{-1}(0)$$
and 
 $$\partial_{1}M = \beta_{1}M\times i_{P}(P).$$
Denote by $N_{M}$ and  $\; N_{\beta_{1}M}$ the normal bundles of $M$ and $\beta_{1}M$ in 
$$X\times\{0\}\times\R_{+}\times\R^{d +\bar{n}}\times\R^{p+m},\quad    \text{and} \quad  X\times\{0\}\times\{0\}\times\R^{d +\bar{n}}$$
respectively. 
It follows from Remark \ref{remark: normal bundles} that there is a bundle isomorphism 
\begin{equation} \label{eq: bundle factorization}
N_{M}|_{\partial_{1}M} \stackrel{\cong} \longrightarrow (N_{\beta_{1}M}\times N_{P})\oplus\epsilon^{1}
\end{equation}
where $N_{P} \rightarrow P$ is the normal bundle for $i_{P}(P) \subset \R^{p+m}$. 
These normal bundles yield Gauss maps,
\begin{equation} \label{Gauss M}
\xymatrix{
N_{M} \ar[rr]^{\hat{\gamma_{M}}} \ar[d] && U_{d+1,n}^{\perp} \ar[d] \\
M \ar[rr]^{\gamma_{M}} && G(d+1,n)}
\end{equation}
and
\begin{equation} \label{Gauss d M}
\xymatrix{
 N_{\beta M}\times N_{P} \ar[d] \ar[rrr]^{\hat{\gamma}_{\beta M}\times\hat{\gamma}_{P}} &&& U_{d-p,n-m}^{\perp}\times U_{p,m}^{\perp} \ar[d] \ar[rrr]^{\hat{\mu}} &&& U_{d,n}^{\perp} \ar[d]\\
 \beta M\times P \ar[rrr]^{\gamma_{\beta M}\times\gamma_{P}} &&& G(d-p,n-m)\times G(p,m) \ar[rrr]^{\mu} &&& G(d,n)}
\end{equation}
where $\mu$ and $\hat{\mu}$ are the maps defined in Section \ref{Thom Spectrum}. These bundle maps induce maps on Thom spaces,
$$\xymatrix{\Th(N_{M}) \ar[rr]^{\Th({\hat{\gamma_{M}}})} && \Th(U_{d+1,n}^{\perp})}$$
and 
$$
\xymatrix{\Th(N_{\beta M})\wedge \Th(N_{P}) \ar[rrr]^{\Th({\hat{\gamma_{\beta M}})\wedge \Th(\hat{\gamma_{P}})}} &&& \Th(U_{d-p,n-m}^{\perp})\wedge\Th(U_{p,m}^{\perp})}.
$$
There are tubular neighborhood embeddings of the normal
bundles $N_{M}$ and $N_{\beta_{1}M}$ into
$$X\times\{0\}\times\R_{+}\times\R^{d+\bar{n}}\times\R^{p+m} \quad \text{and} \quad X\times\{0\}\times\{0\}\times\R^{d + \bar{n}}$$ 
respectively
which yield collapsing maps,
$$
\begin{aligned}
  c_{M}:& \ X_{+}\!\wedge\! D^{d+n} \longrightarrow \Th(N_{M}),\\
   c_{\beta_{1}M}: &  \ X_{+}\!\wedge\! S^{d+\bar{n}} \longrightarrow \Th(N_{\beta_{1}M}),
\end{aligned}
$$
where $X_{+}$ denotes the \textit{one-point compactification} of $X$.
Composing with $\Th({\hat{\gamma_{M}}})$ and $\Th(\hat{\gamma_{\beta
    M}})$ with the above collapsing maps yields the diagram,
$$\xymatrix{
  X_{+}\wedge D^{d+n} \ar[rrrrr]^{\Th({\hat{\gamma_{N_{M}}}})\circ c_{M}} &&&&& \Th(U_{d+1,n}^{\perp}) \\
  \\
  X_{+}\wedge S^{d +\bar{n}}\wedge S^{p+m} \ar@{^{(}->}[uu] \ar[rrrrr]^{[\Th(\hat{\gamma}_{N_{\beta M}}) \circ c_{\beta M}]\wedge
Id_{S^{p+m}}} &&&&& \Th(U_{d-p,n-p}^{\perp})\wedge S^{p+m}. \ar[uu]_{j_{d, n}\circ \tau_{P,n}}}$$ 
It follows from the bundle factorization of (\ref{eq: bundle factorization}) that this diagram does indeed commute. 
By adjunction this commutative diagram yields,
$f: X \longrightarrow \Omega_{P}^{d+n}\widehat{\Th(U_{d+1,n}^{\perp})}.$
By standard Pontryagin-Thom theory (see \cite[Section 2]{St 68}) it follows that choosing a different representative of the concordance class of $W$ yields a map homotopic to the one which we just produced; just run the same process on a concordance. We then define $T_{n}([W]) := [f]$. It is easy to check that this definition is natural in the variable $X$. 
\end{proof}
For $X \in \mathcal{X}$, denote by, 
$$T: \mathbf{D}^{P}_{d+1}[X] \longrightarrow \large[X , \Omega^{\infty-1}_{P}\widehat{\Th(U_{d+1,\infty}^{\perp})}]$$
the map induced in the limit $n \to \infty$ by the maps $T_{n}$ constructed in the previous lemma. 
\begin{lemma} \label{lemma: limiting map equivalence}
For compact $X$, the map
$T: \mathbf{D}^{P}_{d+1}[X] \longrightarrow \large[X , \Omega^{\infty-1}_{P}\widehat{\Th(U_{d+1,\infty}^{\perp})}]$
is an isomorphism of sets. 
\end{lemma}
\begin{proof}
We now make the assumption that $X$ is compact. We construct an inverse to 
$$T: \mathbf{D}^{P}_{d+1}[X] \longrightarrow \large[X , \Omega^{\infty-1}_{P}\widehat{\Th(U_{d+1,\infty}^{\perp}})]$$
which we will denote by $H$. Let
\begin{equation} \label{eq: initial input map}
\xymatrix{
X_{+}\wedge D^{d+n} \ar[rrrr]^{\tilde{f}} &&&& \Th(U^{\perp}_{d+1,n}) \\
\\
(X_{+}\wedge S^{\bar{d}-1+\bar{n}})\wedge S^{p+m} \ar[rrrr]^{f := f_{0}\wedge Id_{S^{p+m}}}  \ar[uu] &&&& \Th(U^{\perp}_{d-p,n-m})\wedge S^{p+m} \ar[uu]_{j_{d, n}\circ\tau_{P, n}}}
\end{equation}
represent an element $[(\tilde{f}, f)] \in \large[X , \Omega_{P}^{d+n}\widehat{\Th(U_{d+1,n}^{\perp}})].$ 

By applying an appropriate homotopy, we may assume the following about $(\tilde{f}, f)$:
\begin{enumerate}
\item[i.] The maps $\tilde{f}$ and $f$ are both smooth away from the pre-image of the base-point. Furthermore, both $\tilde{f}$ and $f$ are transverse to
$G(d+1,n)$ and $G(d-p,n-m)$ as submanifolds of $U_{d+1,n}^{\perp}$ and
$U_{d-p,n-m}^{\perp}$ respectively. 
\item[ii.] By transversality in i. we
have a pair of submanifolds
$$(M, \partial_{1}M) \; \subseteq \; (X\times \R_{+}\times\R^{d+\bar{n}}\times\R^{p+m}, \; X\times\{0\}\times\R^{d+\bar{n}}\times\R^{p+m})$$
where 
$$M  := \tilde{f}^{-1}(G(d+1,n) \quad \text{and} \quad \partial_{1}M :=  (\tau_{P,n}\circ f)^{-1}(G(d,n)).$$ 
Furthermore, $\partial_{1}M$ factors as 
$\partial_{1}M = \beta_{1}M\times i_{P}(P)$
where 
$$\beta_{1}M := f_{0}^{-1}(G(d-p, n-m)) \subset X\times\{0\}\times\R^{d+\bar{p}}$$ 
is a closed submanifold. 
\item[iii.] There exists $\varepsilon >0$ such that
\begin{equation} \label{eq: collar} M \cap \left(X\times
    [0,\varepsilon)\times\R^{d+\bar{n}}\times\R^{p+m}\right) = \partial_{1}M\times [0,\varepsilon).
\end{equation}
\end{enumerate}
The above conditions imply that $M \subset X\times\R_{+}\times\R^{d+\bar{n}}\times\R^{p+m}$ is a closed $P$-submanifold in the sense of Definition \ref{defn: P-submanifold}. 

The submanifolds $M$, $\partial_{1}M$, and $\beta_{1}M$ have normal bundles given by the pull-backs,
 \begin{equation}\label{normal bundles}
 N_{M} = \tilde{f}^{*}(U^{\perp}_{d+1,n}), \quad
 N_{\partial_{1}M} = (\tau_{P, n}\circ f)^{*}(U^{\perp}_{d, n}), \quad \text{and} \quad
 N_{\beta_{1}M} = f_{0}^{*}(U_{d-p,n-m}^{\perp}).
 \end{equation}
 Furthermore, the normal bundle $N_{\partial_{1}M}$ 
 has the factorization $N_{\partial_{1}M} = N_{\beta_{1}M}\times N_{P},$  where $N_{P}$ is the normal bundle for $P \subset \R^{p+m}$.
We define vector bundles,
 \begin{equation} \label{vertical bundles}
T^{\pi}M = \tilde{f}^{*}(U_{d+1,n}), \quad T^{\pi}\partial_{1}M =  (\tau_{P, n}\circ f)^{*}(U_{d,n}), \quad T^{\pi}\beta_{1}M =  f_{0}^{*}(U_{d-p,n-m}),
\end{equation}
over $M$, $\partial_{1}M$, and $\beta_{1}M$ respectively. Below we construct bundle epimorphisms for which these bundles are the kernels. 
By the definition of these bundles and the factorization of $f$, it follows that there are bundle splittings,
\begin{equation}
T^{\pi}M\mid_{\partial_{1}M} = T^{\pi}\partial_{1}M \oplus \epsilon^{1}, \quad \quad T^{\pi}\partial_{1}M = T^{\pi}\beta_{1}M \times TP,
\end{equation}
 and bundle isomorphisms
\begin{equation} \label{bundle comp}
N_{M}\oplus T^{\pi}M  \cong \epsilon^{d+n+1}, \quad N_{\partial_{1}M}\oplus T^{\pi}\partial_{1}M \cong \epsilon^{d+n},\quad  N_{\beta_{1}M}\oplus T^{\pi}\beta_{1}M \cong \epsilon^{d + \bar{n}}.
\end{equation}
Denote by 
$$
(i_{M}, i_{\partial_{1}M}): (M, \partial_{1}M) \longrightarrow (X\!\times \R_{+}\!\times\R^{d+\bar{n}}\times\R^{p+m}, \; X\!\times \{0\}\!\times\R^{d +\bar{n}}\times\R^{p+m})
$$
the inclusion map and let 
$$\pi_{0}: M \longrightarrow X \quad \text{and} \quad \pi_{0}^{\partial}: \partial_{1}M \longrightarrow X$$
denote the projections onto $X$.  
Pulling back the tangent bundle of $X\times\R_{+}\!\times\R^{d+\bar{n}}\times\R^{p+m}$ by 
$(i_{M}, i_{\partial_{1}M})$ yields isomorphisms, 
$$\begin{aligned}
i_{M}^{*}(TX)\oplus\epsilon^{d+n} \; \cong \; TM\oplus N_{M}, \quad & \quad  i_{\partial_{1}M}^{*}(TX)\oplus\epsilon^{d+n-1} \; \cong \; T\partial_{1}M \oplus N_{\partial M}, \\
i_{M}^{*}(TX)\oplus\epsilon^{d+n} \; \cong \; \pi_{0}^{*}(TX)\oplus \epsilon^{d+n},  \quad & \quad i_{\partial_{1}M}^{*}(TX)\oplus\epsilon^{d+n-1} \; \cong \; (\pi_{0}^{\partial})^{*}(TX)\oplus
\epsilon^{d+n-1}.
\end{aligned}$$
Combining these isomorphisms yields,
\begin{equation} \label{bundle iso 2}
 \pi_{0}^{*}(TX)\oplus \epsilon^{d+n} \; \cong \;  TM\oplus N_{M} \quad \text{and} \quad 
 (\pi_{0}^{\partial})^{*}(TX)\oplus \epsilon^{d+n-1} \;  \cong \; T\partial_{1}M\oplus N_{\partial_{1}M}
\end{equation}
both of which cover the identity on $M$ and $\partial_{1}M$. 
By adding $T^{\pi}M$ and $T^{\pi}\partial_{1}M$ via Whitney-sum to both sides of
the equations in (\ref{bundle iso 2}) and using the isomorphisms of
(\ref{bundle comp}), we obtain a commutative diagram,
\begin{equation} \label{bundle diagram}
\xymatrix{
TM\oplus \epsilon^{n+d+1}  \ar[rrr]^{\hat{\pi}_{0}}_{\cong} &&& \pi_{0}^{*}TX\oplus T^{\pi}M \oplus \epsilon^{d+n} \\
(T\partial_{1}M \oplus \epsilon^{n+d}) \oplus \epsilon^{1}   \ar[u] \ar[rrr]^{\hat{\pi}^{\partial}_{0}\oplus Id_{\epsilon^{1}}}_{\cong} &&& (\pi_{0}^{*}TX\oplus T^{\pi}\partial_{1}M \oplus \epsilon^{d+n-1})\oplus \epsilon^{1}  \ar[u]}
\end{equation}
of bundle isomorphisms where the horizontal maps cover the identity on 
$M$ and the vertical maps cover the inclusion of $\partial_{1}M$ into $M$. 
Furthermore the bundle isomorphism (coming from the bottom-horizontal arrow),
$$\hat{\pi}^{\partial}_{0}:T\partial_{1}M \oplus \epsilon^{n+d} \longrightarrow \pi_{0}^{*}TX\oplus T^{\pi}\partial_{1}M \oplus \epsilon^{d+n-1}
$$
has the factorization: 
\begin{equation} \label{factor bundle map}
\hat{\pi}^{\partial}_{0} = \hat{\pi}^{\beta}_{0} \times Id_{TP} \end{equation}
where,
$$\hat{\pi}^{\beta}_{0}:T\beta_{1}M\oplus \epsilon^{d+n} \longrightarrow \pi_{0}^{*}TX\oplus T^{\pi}\beta_{1}M \oplus\epsilon^{d+n-1} $$
is a bundle isomorphism which covers the identity on $\beta_{1}M$.
We will need to use the following \textit{destabilization}.
We postpone the proof of the following result to Section \ref{The Stabilization Map}.
\begin{claim} \label{claim: stabilization map} The bundle isomorphism
  pair $(\hat{\pi}_{0}, \hat{\pi}^{\partial}_{0}\oplus
  Id_{\epsilon^{1}})$ from {\rm (\ref{bundle diagram})} is induced by
  a pair of bundle isomorphisms,
$$\xymatrix{
TM\oplus \epsilon^{1}  \ar[rrr]^{\hat{\pi}_{1}}_{\cong} &&& \pi_{0}^{*}TX\oplus T^{\pi}M \\
(T\partial_{1}M \oplus \epsilon^{1}) \oplus \epsilon^{1} \ar[u] \ar[rrr]^{\hat{\pi}^{\partial}_{1}\oplus Id_{\epsilon^{1}}}_{\cong} &&& (\pi_{0}^{*}TX\oplus T^{\pi}\partial_{1}M)\oplus \epsilon^{1} \ar[u],}
$$
with factorization
$$\hat{\pi}^{\partial}_{1} = \hat{\pi}^{\beta}_{1} \times Id_{TP}$$ 
where
$$ \hat{\pi}^{\beta}_{1}:T\beta_{1}M \oplus \epsilon^{1} \longrightarrow \pi_{0}^{*}TX\oplus T^{\pi}\beta_{1}M $$
is a bundle isomorphism covering the identity on $\beta_{1}M$. 
Furthermore, $(\hat{\pi}_{1}, \hat{\pi}^{\partial}_{1}\oplus
Id_{\epsilon^{1}})$ is unique up to homotopy through bundle map pairs
with the factorization specified above.
\end{claim}
We now define spaces
$$ 
W := M\times \R, \quad \partial_{1}W := \partial_{1}M\times \R \quad
\quad \text{and} \quad \quad \beta_{1}W := \beta_{1}M\times \R.
$$ 
We define the bundles $T^{\pi}W$, $T^{\pi}\partial_{1}W$, $T^{\pi}\beta_{1}W$ to be the pullbacks of the bundles $T^{\pi}M$, $T^{\pi}\partial_{1}M$,
$T^{\pi}\beta_{1}M$, over the projections of $W, \; \partial W, \;
\beta W$ onto $M$, $\partial_{1}M$, $\beta_{1}M$ respectively. 
Denote by
\begin{equation} \label{eq: W inclusion}
i_{W}: W \hookrightarrow  X\times\R\times\R_{+}\times\R^{d+\bar{n}}\times\R^{p+m}
\end{equation}
the inclusion map.
Let $s_{0},
s^{\partial}_{0},$ and $s^{\beta}_{0}$ denote the projections of $W, \;
\partial_{1}W,$ and $\beta_{1}W$ onto the factor $X$ (the reason for the
notation $s_{0}$ used for this projection onto $X$ will become clear
momentarily). The result of claim \ref{claim: stabilization map} yields bundle
isomorphisms, 
\begin{equation} \label{bundle iso W}
\begin{aligned}
TW &\cong s_{0}^{*}(TX)\oplus T^{\pi}W, \\ 
T\partial_{1}W &\cong (s^{\partial}_{0})^{*}(TX)\oplus T^{\pi}\partial_{1}W, \\
T\beta_{1}W  &\cong (s^{\beta}_{0})^{*}(TX)\oplus T^{\pi}\beta_{1}W,
\end{aligned}
\end{equation} 
all which cover the identity.  Using
(\ref{bundle iso W}) we obtain a bundle epimorphism with kernel $T^{\pi}W$, 
\begin{equation} \label{bundle epimorphism}
\xymatrix{
TW \ar[d] \ar[rr]^{\hat{s}_{0}} && TX\ar[d] \\
W \ar[rr]^{s_{0}} &&  X}
\end{equation}
such that the restriction $(\hat{s}_{0}, s_{0})\mid_{\partial_{1}W}$ has the factorization,
\begin{equation}
\label{ep factorization}
\xymatrix{
T\partial_{1}W \ar[d] \ar[rr]^{pr} && T\beta_{1}W \ar[d] \ar[rr]^{\hat{s}^{\beta}_{0}} && TX\ar[d]\\
\partial_{1}W \ar[rr]^{pr} && \beta_{1}W  \ar[rr]^{s_{0}^{\beta}} && X, }
\end{equation}
where $(\hat{s}^{\beta}_{0}, s_{0}^{\beta})$ is a bundle-epimorphism covering $s_{0}^{\beta}$ which is the projection onto $X$. 
\begin{claim} \label{claim: submersion prop} There exists a homotopy
  $(\hat{s}_{t}, s_{t})$ through bundle-epimorphisms such that:
\begin{enumerate}
\item[i.]  At $t = 0$, $(\hat{s}_{0}, s_{0})$ is equal to the bundle epimorphism given in {\rm (\ref{bundle epimorphism})}. 
\item[ii.] The bundle epimorphism $(\hat{s}_{1}, s_{1})$ is integrable, i.e. $D s_{1} = \hat{s}_{1}$ and thus $s_{1}$ is a submersion.
\item[iii.] For all $t$, $(\hat{s}_{t}, s_{t})$ has the factorization
  given in {\rm (\ref{ep factorization})}.
\end{enumerate}
Moreover, the integrable bundle-epimorphism $(\hat{s}_{1}, s_{1})$, is unique up to homotopy though integrable bundle-epimorphisms. 
\end{claim}
We provide a proof of this claim in Section \ref{The Submersion
  Theorem}. This is essentially a relative version of
\textit{Phillips' Submersion Theorem} \cite{Ph 67} adapted for
$P$-manifolds.

Let $(\hat{s}_{t}, s_{t})$ be the desired family of bundle
epimorphisms with the above stated properties. The map $s_{1}$ is now
a submersion of $W$ onto $X$. In order to obtain an element of
$\mathbf{D}^{P}_{d+1}[X]$, we need to realize $s_{1}$ as the
composition of some embedding 
$$W \hookrightarrow X\times\R\times\R_{+}\times\R^{d+\bar{n}}\times\R^{p+m}$$
followed by projection onto $X$. 
Recall that $W = \R\times M$ where $M$ is a closed $P$-manifold. 
By Theorem \ref{theorem: weakly contractible} it follows that for some integer $n'$ (possibly larger than $n$) there exists a $P$-embedding 
$$j: M \longrightarrow \R_{+}\times\R^{d+\bar{n}'}\times\R^{p+m}$$
(by $P$-embedding here we mean element of the space $\mathcal{E}_{P, n' + d}(M)$, see Definition \ref{defn: i-P maps}). 
Furthermore, it follows from Theorem \ref{theorem: weakly contractible} again that if $n'$ is large enough then any two choices of embeddings $j$ are isotopic through $P$-embeddings. 
Now consider the embedding 
\begin{equation} \label{eq: new embedding}
W = \R\times M \longrightarrow X\times\R\times(\R_{+}\times\R^{d+\bar{n}'}\times\R^{p+m}),\quad (t, x) \mapsto (s_{1}(t, x), \; t, \; j(x))
\end{equation}
where $s_{1}(t, x) \in X$, $t \in \R$, and $j(x) \in \R_{+}\times\R^{d+\bar{n}'}\times\R^{p+m}$. 
Denote by $W'$ the image of the above embedding. 
It follows that $W' \in \mathbf{D}^{P}_{d, n'}(X)$. 

We define 
$$H: [X , \Omega^{\infty-1}_{P}\widehat{\Th(U_{d+1, \infty}^{\perp})}] \longrightarrow \mathbf{D}^{P}_{d+1}[X]$$ 
by sending $[\bar{f}] \in [X , \Omega^{\infty-1}_{P}\widehat{\Th(U_{d+1, \infty}^{\perp})}]$ (the class that we started with from (\ref{eq: initial input map})), to the concordance class containing the image of $W^{'}$ in $\mathbf{D}^{P}_{d+1}(X)$. This map is well defined because all choices made in the construction of $W^{'}$ were shown to be unique up to homotopy (namely the submersion found in Claim \ref{claim: submersion prop} and the embedding constructed above). 
One can verify directly that $T\circ H =
Id$. To see that $H\circ T = Id$, recall Lemma \ref{lemma: concordance representation} which states that any concordance class $\mathbf{D}^{P}_{d+1, n}[X]$, has a representative 
$$W \subset X\times\R\times(\R_{+}\times\R^{d+\bar{n}}\times\R^{p+m})$$ 
such that $f: W \longrightarrow \R$ is a bundle projection and thus $W$ is diffeomorphic to the product $f^{-1}(0)\times\R$. 
It is easy to check that $H\circ T$ acts as the identity on such elements. 
This concludes the proof of the lemma.
\end{proof}

\subsection{A Parametrized Thom-Pontryagin Construction} \label{subsection: parametrized pontryagin thom} 
Lemma \ref{lemma: pontryagin thom} establishes a bijection $\mathbf{D}^{P}_{d+1}[X] \cong [X, \Omega^{\infty-1}_{P}\widehat{\Th(U_{d+1, \infty}^{\perp})}]$ for any closed manifold $X$. 
In order to obtain a the weak equivalence $|\mathbf{D}^{P}_{d+1}| \simeq \Omega^{\infty-1}_{P}\widehat{\Th(U_{d+1, \infty}^{\perp})}$ we need to show that this bijection is induced by an actual natural transformation of sheaves, $\mathbf{D}^{P}_{d+1} \longrightarrow \Maps(\; \underline{\hspace{.3cm}} \; , \Omega^{\infty-1}_{P}\widehat{\Th(U_{d+1, \infty}^{\perp})}).$
We proceed in a way very similar to \cite[Page 868-869]{MW 07}. We start with
a definition:
\begin{defn} 
Let $\pi: Y \longrightarrow X$ be a submersion. Let $i_{C}: C
\hookrightarrow Y$ be a smooth submanifold and suppose that
$\pi\mid_{C}$ is still a submersion. A \textit{vertical} tubular
neighborhood for $C$ in $Y$ consists of an open embedding $e: N
\longrightarrow Y$ of the normal bundle of $i_{C}(C) \subset Y$, 
such that $e\circ s = i_{C}$ (where $s$ is the zero-section), and $\pi\circ e = \pi\circ i_{C} \circ q$.
\end{defn}

Using this definition we define a variant of the sheaf
$\mathbf{D}^{P}_{d+1, n}$ which we will denote by
$\widehat{\mathbf{D}}^{P}_{d+1, n}(X)$.
\begin{defn} For $X \in \Ob(\mathcal{X})$ we define $\widehat{\mathbf{D}}^{P}_{d+1, n}(X)$ to be the set of pairs $(W, e)$ where $W \in \mathbf{D}^{P}_{d+1, n}(X)$ and 
$$e: N_{W} \longrightarrow X\times\R\times\R_{+}\times\R^{d+\bar{n}}\times\R^{p+m}$$
is a vertical tubular neighborhood for $W$ with respect to the submersion $\pi: W \longrightarrow X$, 
  subject to the following extra condition. 
The restriction of $e$ to $N_{\partial_{1}W} = N_{\beta_{1}W}\times N_{P}$ (which is the normal bundle of $\partial_{1}W$ in $X\times\R\times\{0\}\times\R^{d+\bar{n}}\times\R^{p+m}$) is
  equal to the product $e_{\beta}\times e_{P}$ where
$$e_{\beta}: N_{\beta_{W}} \longrightarrow
X\times\R\times\{0\}\times\R^{d+\bar{n}}$$ is a vertical tubular
neighborhood for $\beta_{1}W$ and $e_{P}: N_{P} \longrightarrow \R^{p+m}$
is the tubular neighborhood embedding, specified in (\ref{eq: P tubular
  nbh}) that was used in our construction of the spectrum
$\MT_{P}(d+1).$
\end{defn}
We define
\begin{equation} \widehat{\mathbf{D}}^{P}_{d+1} := \colim_{n\to\infty}\widehat{\mathbf{D}}^{P}_{d+1, n}
\end{equation}
where the colimit is taken in the category of sheaves on $\mathcal{X}$. 
It follows easily that for any closed manifold $X$, $\widehat{\mathbf{D}}^{P}_{d+1} = \displaystyle \colim_{n\to\infty}(\widehat{\mathbf{D}}^{P}_{d+1, n}(X))$ and thus there is a weak homotopy equivalence, 
$$|\widehat{\mathbf{D}}^{P}_{d+1}| \simeq \displaystyle \colim_{n\to\infty}|\widehat{\mathbf{D}}^{P}_{d+1, n}|.$$

For each $n$ there is a forgetful map
 $F_{n}: \widehat{\mathbf{D}}^{P}_{d+1, n} \longrightarrow \mathbf{D}^{P}_{d+1, n}$
 defined by sending an element $(W, e)$ to $W$. Passing to the direct limit as $n \to \infty$ yields a natural transformation
 \begin{equation} \label{eq: forgetful map D} 
 F: \widehat{\mathbf{D}}^{P}_{d+1} \longrightarrow \mathbf{D}^{P}_{d+1} 
 \end{equation}
 which by the existence and uniqueness up to isotopy of tubular neighborhoods, induces a homotopy equivalence, 
$|\widehat{\mathbf{D}}^{P}_{d+1}| \; \simeq \; |\mathbf{D}^{P}_{d+1}|.$

\begin{defn} For positive integers $n$ and $d$, we define a sheaf $\mathcal{Z}^{P}_{d+1, n}$ on $\mathcal{X}$ by setting
$$\mathcal{Z}^{P}_{d+1, n}(X) := \Maps(X\times\R, \; \Omega_{P}^{d+n}\widehat{\Th(U_{d+1,n}^{\perp})})$$ 
    for $X \in \Ob(\mathcal{X})$. On the right-hand side, $\Maps(\underline{\hspace{.3cm}}\; , \; \underline{\hspace{.3cm}} )$ simply means here the set of maps with no topology given. 
    These are strictly set valued sheaves.
\end{defn}
For each $n$ the natural map 
$$\Omega_{P}^{d+n}\Th(U_{d+1,n}^{\perp}) \longrightarrow \Omega_{P}^{d+n+1}\Th(U_{d+1,n+1}^{\perp})$$
induces the natural transformation $\mathcal{Z}^{P}_{d+1, n} \rightarrow \mathcal{Z}^{P}_{d+1, n+1}$.
We define, 
$\mathcal{Z}^{P}_{d+1} := \displaystyle\colim_{n\to \infty} \mathcal{Z}^{P}_{d+1, n}.$
For each $n$, the map  
$$j_{0}: \mathcal{Z}^{P}_{d+1, n} \longrightarrow \Maps(X, \; \Omega^{d+n}_{P}\widehat{\Th(U^{\perp}_{d+1, n})})$$
given by restricting a map in $\mathcal{Z}^{P}_{d+1, n}(X)$  to $X\times\{0\}$ is a weak equivalence of sheaves.
In the limit $n \to \infty$, the maps $j_{0}$ induce a 
 homotopy equivalence,
$|\mathcal{Z}^{P}_{d+1}| \simeq \Omega^{\infty -1}_{P}\widehat{\Th(U^{\perp}_{d+1, \infty})}.$ 
    
The \textit{Pontryagin-Thom} construction yields a  map of sheaves,
$\widehat{T}_{n}: \widehat{\mathbf{D}}^{P}_{d+1, n} \longrightarrow \mathcal{Z}^{P}_{d+1, n}$ for each $n$
which we describe in detail. Let $(W, e)$ be an element of
$\widehat{\mathbf{D}}^{P}_{d+1, n}(X)$. Since 
$$(\pi, f): W \longrightarrow X\times\R$$ 
is a proper map, for each $(x, t) \in X\times\R$ there exists a positive real number denoted by 
$\lambda(x, t) > 0$, such that the element 
$$(x, t, z) \in X\times\R\times(\R_{+}\times\R^{d+\bar{n}}\times\R^{p+m})$$ 
lies in the compliment of $W \subset X\times\R\times\R_{+}\times\R^{d+\bar{n}}\times\R^{p+m}$
whenever $|z| \geq \lambda(x, t)$ (where $|z|$ denotes the length of the vector $z$ in the Euclidean metric). 
The numbers $\lambda(x, t)$ can be chosen to make $(x, t) \mapsto \lambda(x, t)$
a continuous function. It follows that the
collapsing map 
$$X\times\R\times\R_{+}\times\R^{d+\bar{n}}\times\R^{p+m} \longrightarrow \Th(N_{W})$$
induced by the vertical tubular embedding $e$,
extends to 
$$X\times\R\times D^{d-1+n} \longrightarrow \Th(N_{W})$$ where
$D^{d-1+n}$ is viewed as the one-point-compactification of
$\R_{+}\times\R^{d+\bar{n}}\times\R^{p+m}$.  Putting together the above collapsing map
and the functor $\Th(\; \underline{\hspace{.3cm}}\; )$ applied to the Gauss-maps of the
normal bundles $N_{W}$ and $N_{\partial_{1}W} = N_{\beta_{1}W}\times N_{P}$, 
we obtain the commutative diagram,
$$\xymatrix{
X\times\R\times D^{d+n} \ar[rr] && \Th(N_{W}) \ar[rr] && \Th(U^{\perp}_{d+1, n}) \\
X\times\R\times S^{d+\bar{n}}\wedge S^{m+p} \ar[u] \ar[rr] && \Th(N_{\beta W})\wedge S^{m+p} \ar[rr] && \Th(U^{\perp}_{d-p, n-m})\wedge S^{m+p} \ar[u]_{j_{d, n}\circ {\tau}_{P,n}}}$$
which yields an element of $\mathcal{Z}^{P}_{d+1, n}$ via adjunction. We denote by $\widehat{T}: \widehat{\mathbf{D}}^{P}_{d+1} \longrightarrow \mathcal{Z}^{P}_{d+1}$
the induced map in the direct limit as $n \to \infty$. 
 The concordance class functors associated to our newly defined sheaves fit into the commutative diagram,
\begin{equation} \label{eq: concordance diagram}
\xymatrix{ 
\widehat{\mathbf{D}}^{P}_{d+1}[\;  \underline{\hspace{.3cm}}\;] \ar[rrr]^{\widehat{T}_{*}} \ar[d]_{\cong} &&& \mathcal{Z}^{P}_{d+1}[\;  \underline{\hspace{.3cm}} \; ] \ar[d]_{\cong} \\
\mathbf{D}^{P}_{d+1}[\; \underline{\hspace{.3cm}} \;] \ar[rrr]^{T} &&& [\; \underline{\hspace{.3cm}} \;, \Omega^{\infty-1}_{P}\widehat{\Th(U^{\perp}_{d+1, \infty})}]}
\end{equation}
where $T$ is the natural transformation defined in Section \ref{A Natural Transformation}, the vertical maps are induced by $F$ and $j_{0}$. The bottom row is an isomorphism whenever applied to a compact manifold and so by commutativity, the top map is as well. It follows that for each $k \geq 0$, the map $|\widehat{T}|$ induces a bijection
$$\xymatrix{
|\widehat{T}|: [S^{k}\; , \; |\widehat{\mathbf{D}}^{P}_{d+1}|] \ar[rrr]^{\cong} &&& [S^{k}\; , \;  \Omega^{\infty-1}_{P}\widehat{\Th(U^{\perp}_{d+1, \infty})}].
}$$ 
However, these are isomorphisms of sets of homotopy classes of unbased maps and not isomorphisms of the actual homotopy groups. In order to prove that the map $|\widehat{T}|$ is a homotopy equivalence, we need the following proposition whose proof is similar to  \cite[Theorem 3.8]{MW 07}. We give the proof in Section \ref{section:monoidal}. 
\begin{proposition} \label{prop: homotopy monoid} The map $|\widehat{T}|$ induces an isomorphism on all homotopy groups with respect to any choice of basepoint.\end{proposition}
Proposition \ref{prop: homotopy monoid} together with Corollary \ref{cor: infinite loop-space eq} implies Theorem \ref{Main Theorem} which states that there is a weak homotopy equivalence,
$$|\mathbf{D}^{P}_{d+1}| \simeq \Omega^{\infty-1}\MT_{P}(d+1).$$
By Remark \ref{remark: homotopy invariance under bordism class} we have that if $P$ is cobordant to a closed manifold $P'$ then there is a homotopy equivalence, $\Omega^{\infty-1}\MT_{P}(d+1) \simeq \Omega^{\infty-1}\MT_{P'}(d+1)$. 
This implies the following corollary. 
\begin{corollary}
If the closed manifolds $P$ and $P'$ are cobordant, then there is a weak homotopy equivalence $|\mathbf{D}^{P}_{d+1}| \simeq |\mathbf{D}^{P'}_{d+1}|$. 
\end{corollary}
 
\section{The classifying space of $\mathbf{Cob}_{d+1}^{P}$}
 \label{The Classifying Space}
 In this section we construct a weak homotopy equivalence, $|\mathbf{D}^{P}_{d+1}| \simeq \mathbf{BCob}^{P}_{d+1}.$ 
 Combining this weak equivalence with the results of the previous section implies Theorem \ref{thm: Main}. 
 \subsection{Category Valued Sheaves} \label{subsection: Cat-valued sheaves}
 We will need to consider sheaves on $\mathcal{X}$ that are valued in the category of small categories, which is denoted by $\mathbf{CAT}$. 
 \begin{defn} \label{defn: cat-valued sheaf}
A contravariant functor $\mathcal{F}: \mathcal{X} \longrightarrow \mathbf{CAT}$ is said to be a $\mathbf{CAT}$-valued sheaf if for any $X \in \Ob(\mathcal{X})$, $\mathcal{F}(X)$ satisfies the same gluing condition described in Section \ref{subsection: Recollection of Sheaves}  for $\mathbf{Set}$-valued sheaves, with respect to any open cover of $X$.  
 \end{defn}
 We now recall some important facts about $\mathbf{CAT}$-valued sheaves.  
Let $\mathcal{F}$ be a $\mathbf{CAT}$-valued sheaf. 
For each non-negative integer $k$, there is an auxiliary $\mathbf{Set}$-valued sheaf denoted by $N_{k}\mathcal{F}$ which is defined by setting $N_{k}\mathcal{F}(X)$ equal to the $k$-th level of the nerve of the category $\mathcal{F}(X)$. 
The representing space $|\mathcal{F}|$ naturally obtains the structure of a topological category by setting,
\begin{equation} \label{eq: cat valued sheaf} 
\Ob(|\mathcal{F}|) = |N_{0}\mathcal{F}| \quad \text{and} \quad \Mor(|\mathcal{F}|) = |N_{1}\mathcal{F}|.
\end{equation}
The classifying space $B|\mathcal{F}|$ can be constructed by taking the geometric realization of the \textit{diagonal} simplicial set,
$k \mapsto N_{k}\mathcal{F}(\triangle_{e}^{k}).$

If $\mathcal{C}$ is a topological category with a smooth structure  (i.e. $\mathcal{C}$ is a smooth manifold or an infinite dimensional \textit{Banach manifold} like in the case with $\mathbf{Cob}_{d+1}^{P}$) then for $X \in \Ob(\mathcal{X})$, the set of smooth maps $C^{\infty}(X, \mathcal{C})$ has the structure of a small category by pointwise composition. 
The contravariant functor $C^{\infty}(\; \underline{\hspace{.3cm}} \; , \mathcal{C})$ defines a $\mathbf{CAT}$-valued sheaf on $\mathcal{X}$. 
As with all $\mathbf{CAT}$-valued sheaves, the representing space $|C^{\infty}(\; \underline{\hspace{.3cm}} \;, \mathcal{C})|$ has the structure of a topological category. 
In this way, the natural map,
 \begin{equation} \label{eq: cannonical map}
|C^{\infty}(\; \underline{\hspace{.3cm}}\;, \mathcal{C})| \longrightarrow \mathcal{C}
 \end{equation}
defined between the geometric realization of the singular complex of a space and the space itself (see \cite{M 57}), is a continuous functor. 
It follows from \cite{GMTW 09} that (\ref{eq: cannonical map}) induces a weak-homotopy equivalence
\begin{equation} \label{eq: canonical classifying space equivalence}
B|C^{\infty}(\;\underline{\hspace{.3cm}}\;, \mathcal{C})| \longrightarrow B\mathcal{C}.
\end{equation}
Note that the above construction and weak homotopy equivalence holds for the $\mathbf{CAT}$-valued sheaf $\Maps(\; \underline{\hspace{.3cm}}\;, \mathcal{C})$ as well. 

\subsection{Cocycle Sheaves}
We will need to use another important construction from \cite{MW 07} and \cite{GMTW 09} relating to $\mathbf{CAT}$-valued sheaves . 
For the next definition, fix once and for all an uncountable set $J$. 
For elements $X \in \Ob(\mathcal{X})$ we will need to consider open covers $\mathcal{U}$ of $X$ indexed by the set $J$. 
For subset $S \subset J$, we denote $U_{S} := \cap_{i \in S}U_{i}$. 
\begin{defn} \label{defn: cocycle sheaf}
Let $\mathcal{F}$ be a $\textbf{CAT}$-valued sheaf on $\mathcal{X}$. 
For each $X \in \Ob(\mathcal{X})$, $\bar{\beta}\mathcal{F}(X)$ is defined to be the set of pairs $(\mathcal{U}, \Phi)$ where $\mathcal{U} = \{U_{i} \; | \; j \in J\}$ is an open cover of $\mathcal{X}$ indexed by $J$, and $\Phi$ is collection of morphisms, $\varphi_{RS} \in N_{1}\mathcal{F}(U_{S})$, indexed by the pairs $R \subseteq S$ of non-empty finite subsets of $J$, subject to the following conditions:
\begin{enumerate}
\item[i.] $\varphi_{RR} = Id_{C_{R}}$ for some object $C_{R} \in N_{0}\mathcal{F}(U_{R})$.
\item[ii.] For each non-empty finite $R \subseteq S$, $\varphi_{RS}$ is a morphism from $C_{S}$ to $C_{R}|_{U_{S}}$. 
\item[iii.] For all triples $R \subseteq S \subseteq T$ of finite non-empty subsets of $J$, we have 
$$\varphi_{RT} = (\varphi_{RS}|_{U_{T}})\circ \varphi_{ST}.$$
\end{enumerate}
It can be verified that for any $\mathbf{CAT}$-valued sheaf $\mathcal{F}$ on $\mathcal{X}$, the assignment $X \mapsto \bar{\beta}\mathcal{F}(X)$ defines a \textbf{Set}-valued sheaf on $\mathcal{X}$. 
The sheaf $\bar{\beta}\mathcal{F}$ is called the \textit{cocycle-sheaf} associated to $\mathcal{F}$. 
\end{defn}

In \cite{MW 07} it is proven that for any $\mathbf{CAT}$-valued sheaf $\mathcal{F}$ on $\mathcal{X}$, there is a weak homotopy equivalence,
\begin{equation}
|\bar{\beta}\mathcal{F}| \simeq B|\mathcal{F}|
\end{equation}
where $|\bar{\beta}\mathcal{F}|$ is the representing space of the $\mathbf{Set}$-valued sheaf $\bar{\beta}\mathcal{F}$ and $B|\mathcal{F}|$ is the classifying space of the topological category $|\mathcal{F}|$. 
 This homotopy equivalence is natural in the following sense. 
  \begin{remark} \label{remark: set valued sheaf}
 Since any set may by considered a category with only identity morphisms, 
 a $\mathbf{Set}$-valued sheaf $\mathcal{F}$ on $\mathcal{X}$ may be considered a $\mathbf{CAT}$-valued sheaf by considering $\mathcal{F}(X)$ a category with only identity morphisms, for $X \in \mathcal{X}$. 
 In this way, we may consider $\bar{\beta}\mathcal{F}$. 
For any $X \in \Ob(\mathcal{X})$, $\bar{\beta}\mathcal{F}(X)$ reduces to the set of pairs $(\mathcal{U}, \Phi)$ where $\mathcal{U}$ is a cover of $X$ (indexed by the designated uncountable set $J$ from Definition \ref{defn: cocycle sheaf}) and $\Phi$ is a collection of elements $\varphi_{S} \in \mathcal{F}(U_{S})$ for $S \subset J$, which are compatible under restrictions. 
Using the sheaf gluing condition, any element $(\mathcal{U}, \Phi) \in \bar{\beta}\mathcal{F}(X)$ induces a unique element $\varphi \in \mathcal{F}(X)$
such that $\varphi|_{U_{S}} = \varphi_{S}$ for all subsets $S \subset J$. 
This correspondence $(\mathcal{U}, \Phi) \mapsto \varphi$ yields a natural isomorphism,
$\bar{\beta}\mathcal{F} \stackrel{\cong} \longrightarrow \mathcal{F}.$
\end{remark}
 
 \begin{notation}
In \cite{MW 07} and in \cite{GMTW 09} the notation $\beta\mathcal{F}$ for the above cocycle-sheaf construction. 
Since in this paper, the greek letter $\beta$ is already used so heavily, we denote the cocycle-sheaves by $\bar{\beta}\mathcal{F}$ so as to not have any conflict of notation.
\end{notation}

\subsection{Sheaf Models for $\mathbf{Cob}^{P}_{d+1}$} 
We will need to consider the $\mathbf{CAT}$-valued sheaf,  $C^{\infty}(\; \underline{\hspace{.3cm}} \;, \mathbf{Cob}^{P}_{d+1})$ with composition defined pointwise. By (\ref{eq: canonical classifying space equivalence}) there is a weak homotopy equivalence, 
$$B|C^{\infty}(\; \underline{\hspace{.3cm}} \;, \mathbf{Cob}^{P}_{d+1})| \simeq B\mathbf{Cob}^{P}_{d+1}$$ 
induced by the natural map
$|C^{\infty}(\; \underline{\hspace{.3cm}} \;, \mathbf{Cob}^{P}_{d+1})| \longrightarrow  \mathbf{Cob}^{P}_{d+1}$.
In this section we will define three new $\mathbf{CAT}$-valued sheaves
$\mathbf{C}^{P}_{d+1}$, $\mathbf{C}^{P,\pitchfork}_{d+1},$ and $\mathbf{D}^{P, \pitchfork}_{d+1}$, along with a zig-zag of natural transformations
\begin{equation} \label{eq: zig-zag of natural transformations}
\xymatrix{
C^{\infty}(\; \underline{\hspace{.3cm}} \;, \mathbf{Cob}^{P}_{d+1}) \ar[rr]^{\eta} &&
 \mathbf{C}^{P}_{d+1} \ar[rr]^{i} && \mathbf{C}^{P,\pitchfork}_{d+1} && \mathbf{D}^{P, \pitchfork}_{d+1} \ar[ll]_{\alpha} \ar[rr]^{\gamma} && \mathbf{D}^{P}_{d+1}
}
\end{equation}
(on the right, $\mathbf{D}^{P}_{d+1}$ is considered a $\mathbf{CAT}$-valued sheaf with only identity morphisms) which induce  weak
homotopy equivalences, 
 $$\xymatrix{
 B\mathbf{Cob}^{P}_{d+1} \ar[rr]^{\simeq} &&
 B|\mathbf{C}^{P}_{d+1}| \ar[rr]^{\simeq} && B|\mathbf{C}^{P,\pitchfork}_{d+1}| && B|\mathbf{D}^{P, \pitchfork}_{d+1}| \ar[ll]_{\simeq} \ar[rr]^{\simeq} && |\mathbf{D}^{P }_{d+1}|.}$$ 
These weak equivalences together with Theorem \ref{Main Theorem} will
 imply Theorem \ref{thm: Main}. The constructions of this section closley follow
follow \cite{GMTW 09} and \cite{G 08}. 
 \begin{remark} \label{remark: non-singular sheaves}
 The $\mathbf{CAT}$-valued sheaves $\mathbf{C}^{P}_{d+1}$, $\mathbf{C}^{P,\pitchfork}_{d+1},$ and $\mathbf{D}^{P, \pitchfork}_{d+1}$ correspond directly to the $\mathbf{CAT}$-valued sheaves from \cite{GMTW 09} denoted by 
 $\mathbf{C}_{d+1}$, $\mathbf{C}^{\pitchfork}_{d+1}$, and $\mathbf{D}^{ \pitchfork}_{d+1}.$
 In particular, these sheaves are isomorphic to the ones defined in this section in the case that $P = \emptyset$. 
 \end{remark}
 
Let $ i_{P}: P \hookrightarrow \R^{p+m}$ used in the construction of $\mathbf{Cob}^{P}_{d+1}$ and throughout the paper.
\begin{Notation}For $X \in \Ob(\mathcal{X})$ and smooth functions $a, b: X \longrightarrow \R$ with $a(x) \leq b(x)$ for all $x \in X$, we denote
$$\begin{aligned}
  X \times [a, b] := \{(x, u) \in X\times \R \; |\; a(x)< u < b(x) \},\\
  X \times (a, b) := \{(x, u) \in X\times \R \; |\; a(x) \leq u \leq b(x)\}.
\end{aligned}$$ 
\end{Notation}
\begin{defn} Let $\varepsilon > 0$ be real number. Let $X\in \Ob(\mathcal{X})$ and $a, b: X \longrightarrow \R$
  be smooth functions with $a(x) \leq b(x)$ for all $x \in X$.  We
  define, $\mathbf{C}^{P, \pitchfork}_{d+1}(X; a, b, \varepsilon)$ to be the set of $(d+1+\dim(X))$-dimensional $P$-submanifolds (see Definition \ref{defn: P-submanifold}),
 $$W \subset X\times (a-\varepsilon, b+\varepsilon)
\times\R_{+}\times\R^{d+\infty}\times\R^{p+m}$$ 
which satisfy the conditions:
  \begin{enumerate}
\item[i.] The projection of $\pi: W\rightarrow X$ is a $P$-submersion with
  $(d+1)$-dimensional fibres. 
  \item[ii.] The projection of $W$ onto $X\times(a-\varepsilon, b+\varepsilon)$, denoted by $(\pi, f): W \rightarrow X \times (a-\varepsilon,
  b+\varepsilon)$, is a proper $P$-map. 
\item[iii.] The restriction of $(\pi, f)$ to $(\pi,f)^{-1}(X\times (\nu -\varepsilon, \nu +\varepsilon))$ for $\nu = a, b$ are $P$-submersions. 
 \end{enumerate}
 \end{defn}

\begin{remark} Condition iii. of the above definition implies that the map $\pi$ is a local-trivial fibre-bundle and not just a submersion. The manifold $\widehat{W} := (\pi, f)^{-1}(X \times [a, b])$ is a $P$-manifold of dimension $d+1$ with boundary given by,
$$\partial_{0}\widehat{W} =  (\pi, f)^{-1}(X\times\{a\})\bigsqcup (\pi, f)^{-1}(X\times\{b\}).$$ 
The restriction of $\pi$ to $\partial_{0}\widehat{W}$ is a fibre-bundle with closed $P$-manifold fibres. 
\end{remark}
We eliminate dependence on $\varepsilon$ by setting
 \begin{equation} \mathbf{C}^{P, \pitchfork}_{d+1}(X; a, b) := \colim_{\varepsilon \to 0}\mathbf{C}^{P, \pitchfork}_{d+1}(X; a, b, \varepsilon).\end{equation}
\begin{defn} We define $\mathbf{C}^{P, \pitchfork}_{d+1}(X) := \bigsqcup\mathbf{C}^{P, \pitchfork}_{d+1}(X; a, b)$ with union ranging over
  all pairs of smooth real-valued functions $(a, b)$ with $a \leq b$. \end{defn} 
  This definition makes
$\mathbf{C}^{P, \pitchfork}_{d+1}$ into a $\mathbf{CAT}$-valued
sheaf. 
 \begin{defn} \label{defn: C-sheaf} For $X \in \mathcal{X}$, smooth functions $a, b: X \longrightarrow \R$ with $a(x) \leq b(x)$ for all $x \in X$, and a real number $\varepsilon > 0$,
 we define $\mathbf{C}^{P}_{d+1}(X, a, b, \varepsilon) \subset \mathbf{C}^{P, \pitchfork}_{d+1}(X, a, b, \varepsilon)$ be the subset 
 consisting of all elements $W \in \mathbf{C}^{P, \pitchfork}_{d+1}(X, a, b, \varepsilon)$ which satisfy the further condition:
\begin{enumerate}
\item[iv.]
For $\nu = a, b$ and $x \in X$, let $J_{\nu}$ be the interval $((\nu - \epsilon_{0})(x), (\nu + \epsilon_{0})(x)) \subseteq \R$ and let
  $$
V_{\nu} \; = \; (\pi, f)^{-1}(\{x\}\times J_{\nu}) \; \subset \; \{x\}\times J_{\nu}\times\R_{+}\times \R^{d+\infty}\times\R^{p+m}.
$$ 
Then the following holds,
 $$
V_{\nu} \; = \; \{x\} \times J_{\nu}\times M  \; \subset \; \{x\}\times J_{\nu}\times\R_{+}\times \R^{d+\infty}\times\R^{p+m}
$$
for some $d$-dimensional submanifold $M \subset \R_{+}\times \R^{d+\infty}\times\R^{p+m}$. 
\end{enumerate}
 It follows that the boundary $\partial M = M \cap
 \left(\{x\}\times J_{a}\times\R_{+}\times
   \R^{d+\infty}\times\R^{p+m}\right)$ has the factorization
 $\partial_{1}M = \beta_{1}M \times i_{P}(P)$.
 \end{defn}  
 The above definition should be compared to \cite[Definition 2.8]{GMTW 09}. 
 We then may define  $\mathbf{C}^{P}_{d+1}$ in the same way as for $\mathbf{C}^{P, \pitchfork}_{d+1}$ 
 by taking the limits as $\epsilon_{0}, \epsilon_{1} \to 0$ and taking the disjoint union over all pairs of real valued
 functions $a, b: X \longrightarrow \R$ such that $a(x) \leq b(x)$.

Using the topological structure on $\mathbf{Cob}^{P}_{d+1}$ given in (\ref{eq: top category 1}) together with Lemmas \ref{lemma: fibre-bundles} and \ref{lemma: fibre-bundles 2} regarding fibre-bundles with $P$-manifold fibres, 
there is a natural isomorphism 
\begin{equation}
\beta: C^{\infty}(\; \underline{\hspace{.3cm}} \;, \mathbf{Cob}^{P}_{d+1}) \stackrel{\cong} \longrightarrow \mathbf{C}^{P}_{d+1},
\end{equation}
 given by sending a smooth map $f: X \longrightarrow \mathbf{Cob}^{P}_{d+1}$ to the fibre-bundle of $P$-manifold cobordisms over $X$ that $f$ induces by pull-back. 
By (\ref{eq: canonical classifying space equivalence}), this isomorphism induces a weak homotopy equivalence, $B \mathbf{Cob}^{P}_{d+1} \simeq B|\mathbf{C}^{P}_{d+1}|$. 
 
 Denote by $i: \mathbf{C}^{P}_{d+1} \rightarrow \mathbf{C}^{P, \pitchfork}_{d+1}$ the natural transformation induced by inclusion. 
 The following proposition is proven in exactly the same way as \cite[Proposition 4.4]{GMTW 09}.
 \begin{proposition} The inclusion map $i: \mathbf{C}^{P}_{d+1} \rightarrow \mathbf{C}^{P, \pitchfork}_{d+1}$ induces a weak homotopy equivalence,
    $$B|\mathbf{C}^{P}_{d+1}| \simeq  B|\mathbf{C}^{P, \pitchfork}_{d+1}|.$$
  \end{proposition}
  
  We now define a new sheaf which can be compared directly to $\mathbf{D}^{P}_{d+1}$. 
  \begin{defn} For $X \in \Ob(\mathcal{X})$, we define $\mathbf{D}^{P,\pitchfork}_{d+1}(X)$ to be the set
   of pairs $(W, a) \in \mathbf{D}^{P}_{d+1}(X)\times C^{\infty}(X, \R)$ such that 
   for all $x \in X$, the restriction map $f|_{\pi^{-1}(x)}: \pi^{-1}(x) \longrightarrow \R$ is transverse to $a(x) \in \R$.
We say that $f$ is \textit{fibrewise transverse} to $a: X \rightarrow \R$ with respect to the submersion $\pi$.
   The set $\mathbf{D}^{P,\pitchfork}_{d+1}(X)$ is then given the structure of a partially ordered set by declaring 
   $(W, a) \leq (V, b)$
   if $W = V$ and $a(x) \leq b(x)$ for all $x \in X$.
   By considering the partially ordered set $\mathbf{D}^{P,\pitchfork}_{d+1}(X)$ a category, $\mathbf{D}^{P,\pitchfork}_{d+1}$ is a $\mathbf{CAT}$-valued sheaf of $\mathcal{X}$. 
   \end{defn}  
  There is a natural transformation,
 \begin{equation} \label{eq: alpha} 
 \alpha: \mathbf{D}^{P,\pitchfork}_{d+1} \longrightarrow \mathbf{C}^{P, \pitchfork}_{d+1}
 \end{equation}
 which is defined in exactly the same way as the map $\mathbf{D}^{\pitchfork}_{d+1} \rightarrow \mathbf{C}^{\pitchfork}_{d+1}$ in from \cite[Page 17]{GMTW 09} for the corresponding sheaves for non-singular manifolds. 
The following proposition is proved in exactly the same way as the map
\cite[Proposition 4.4]{GMTW 09}.
\begin{proposition} The natural transformation $\alpha: \mathbf{D}^{P,\pitchfork}_{d+1} \longrightarrow \mathbf{C}^{P, \pitchfork}_{d+1}$ induces a weak homotopy equivalence $B|\mathbf{D}^{P, \pitchfork}_{d+1}| \simeq B|\mathbf{C}^{P, \pitchfork}_{d+1}|.$
\end{proposition}
  
We now compare $\mathbf{D}^{P, \pitchfork}_{d+1}$ to $\mathbf{D}^{P}_{d+1}$. By considering $\mathbf{D}^{P}_{d+1}$ as a $\mathbf{CAT}$-valued sheaf with only identity morphisms, there is a forgetful functor,
$\gamma: \mathbf{D}^{P, \pitchfork}_{d+1} \longrightarrow \mathbf{D}^{P}_{d+1}$
defined by sending $(W, a) \in \mathbf{D}^{P, \pitchfork}_{d+1}(X)$ to $W \in \mathbf{D}^{P}_{d+1}(X)$. 
This induces a natural transformation 
$$\bar{\beta}\gamma: \bar{\beta}\mathbf{D}^{P, \pitchfork}_{d+1} \longrightarrow \bar{\beta}\mathbf{D}^{P}_{d+1} \cong \mathbf{D}^{P}_{d+1}$$
where $\bar{\beta}\mathbf{D}^{P}_{d+1} \cong \mathbf{D}^{P}_{d+1}$ is the isomorphism from Remark \ref{remark: set valued sheaf}.  
The next proposition is proven in the same way as \cite[Proposition 4.2]{GMTW 09}.
 \begin{proposition}
   The map $\bar{\beta}\gamma: \bar{\beta}\mathbf{D}^{P, \pitchfork}_{d+1} \longrightarrow \mathbf{D}^{P}_{d+1}$ induces a weak homotopy equivalence $B|\mathbf{D}^{P, \pitchfork}_{d+1}| \simeq |\mathbf{D}^{P}_{d+1}|$.
 \end{proposition}
The last four propositions imply that there is a weak
homotopy equivalence, 
$$|\mathbf{D}^{P}_{d+1}| \simeq B\mathbf{Cob}^{P}_{d+1}.$$ 
Combining this with Theorem \ref{Main Theorem} yields the weak homotopy equivalence,
$$B\mathbf{Cob}^{P}_{d+1} \; \simeq \; \Omega^{\infty-1}_{P}\MT(d+1),$$
thus completing the proof of Theorem \ref{thm: Main}. 

\subsection{The Bockstein Functor} \label{The Boundary Map}
Recall the cobordism category $\mathbf{Cob}_{d+1}$ from \cite{GMTW 09}. 
Setting $P$ equal to $\emptyset$ we have an isomorphism of topological categories,
$\mathbf{Cob}^{\emptyset}_{d+1} \cong \mathbf{Cob}_{d+1}.$
We now consider the functor 
$\beta_{1}: \mathbf{Cob}^{P}_{d+1} \longrightarrow \mathbf{Cob}_{d-p}$
defined by sending a $P$-subcobordism 
$$W \subset [a, b]\times\R_{+}\times\R^{d+\infty}\times\R^{p+m}$$
to the embeded (non-singular) cobordism given by, 
$\beta_{1}W \subset [a, b]\times\{0\}\times\R^{d+\infty}.$ 
The functor $\beta_{1}$ is defined similarly on objects.
Furthermore, the category $\mathbf{Cob}_{d+1}$ is isomorphic to the subcategory of $\mathbf{Cob}^{P}_{d+1}$ consisting of all submanifolds 
$$W \subset [a, b]\times\R_{+}\times\R^{d+\infty}\times\R^{p+m}$$
such that 
$$W\cap ([a, b]\times\{0\}\times\R^{d+\infty}\times\R^{p+m}) \; = \; \emptyset.$$ 
We denote by 
$$i: \mathbf{Cob}_{d+1} \longrightarrow \mathbf{Cob}^{P}_{d+1}$$
the inclusion functor.
Theorem \ref{thm: homotopy fibre} states that the composition of functors
$$\xymatrix{
\mathbf{Cob}_{d+1} \ar[rr]^{i} && \mathbf{Cob}^{P}_{d+1} \ar[rr]^{\beta_{1}} && \mathbf{Cob}_{d-p}
}$$
induces a homotopy-fibre sequence on the level of classifying spaces,
$$\xymatrix{
B\mathbf{Cob}_{d+1} \ar[rr]^{B(i)} && B\mathbf{Cob}^{P}_{d+1} \ar[rr]^{B(\beta_{1})} && B\mathbf{Cob}_{d-p}.
}$$
We now give a proof of Theorem \ref{thm: homotopy fibre}.
\begin{proof}[Proof of Theorem \ref{thm: homotopy fibre}.] \label{proof: homotopy fibre}
There are versions of the functors $i_{d}$ and $\beta_{1}$ defined on the sheaf level which yields a commutative diagram of 
natural transformations,
\begin{equation} \label{eq: diagram of fibre-sequences}
\xymatrix@-.3pc@R-.2pc{
C^{\infty}(\; \cdot \;, \mathbf{Cob}_{d+1}) \ar[rr]^{\eta} \ar[d]^{i_{d}} &&
 \mathbf{C}_{d+1} \ar[rr]^{i} \ar[d]^{i_{d}} && \mathbf{C}^{\pitchfork}_{d+1} \ar[d]^{i_{d}} && \mathbf{D}^{\pitchfork}_{d+1} \ar[d]^{i_{d}} \ar[ll]_{\alpha} \ar[rr]^{\gamma} && \mathbf{D}_{d+1} \ar[d]^{i_{d}}\\
 C^{\infty}(\; \cdot \;, \mathbf{Cob}^{P}_{d+1}) \ar[rr]^{\eta} \ar[d]^{\beta_{1}} &&
 \mathbf{C}^{P}_{d+1} \ar[rr]^{i} \ar[d]^{\beta_{1}} && \mathbf{C}^{P,\pitchfork}_{d+1} \ar[d]^{\beta_{1}} && \mathbf{D}^{P, \pitchfork}_{d+1} \ar[d]^{\beta_{1}} \ar[ll]_{\alpha} \ar[rr]^{\gamma} && \mathbf{D}^{P}_{d+1} \ar[d]^{\beta_{1}}\\
 C^{\infty}(\; \cdot \;, \mathbf{Cob}_{d-p}) \ar[rr]^{\eta} &&
 \mathbf{C}_{d-p} \ar[rr]^{i} && \mathbf{C}^{\pitchfork}_{d-p} && \mathbf{D}^{\pitchfork}_{d-p} \ar[ll]_{\alpha} \ar[rr]^{\gamma} && \mathbf{D}_{d-p}.
}
\end{equation}
This diagram of natural transformations then induces a commutative diagram of maps of spaces,
\begin{equation} \label{eq: classifying space commutativity}
\xymatrix@C-.3pc@R-.2pc{
B\mathbf{Cob}_{d+1} \ar[rr]^{B\eta}_{\simeq} \ar[d]^{Bi_{d}} &&
 B|\mathbf{C}_{d+1}| \ar[rr]^{Bi}_{\simeq} \ar[d]^{Bi_{d}} && B|\mathbf{C}^{\pitchfork}_{d+1}| \ar[d]^{Bi_{d}} && B|\mathbf{D}^{\pitchfork}_{d+1}| \ar[d]^{Bi_{d}} \ar[ll]_{B\alpha}^{\simeq} \ar[rr]^{B\gamma}_{\simeq} && |\mathbf{D}_{d+1}| \ar[d]^{Bi_{d}}\\
 B\mathbf{Cob}^{P}_{d+1} \ar[rr]^{B\eta}_{\simeq} \ar[d]^{B\beta_{1}} &&
 B|\mathbf{C}^{P}_{d+1}| \ar[rr]^{Bi}_{\simeq} \ar[d]^{B\beta_{1}} && B|\mathbf{C}^{P,\pitchfork}_{d+1}| \ar[d]^{B\beta_{1}} && B|\mathbf{D}^{P, \pitchfork}_{d+1}| \ar[d]^{\beta_{1}} \ar[ll]_{B\alpha}^{\simeq} \ar[rr]^{B\gamma}_{\simeq} && |\mathbf{D}^{P}_{d+1}| \ar[d]^{B\beta_{1}}\\
B\mathbf{Cob}_{d-p} \ar[rr]^{B\eta}_{\simeq} &&
 B|\mathbf{C}_{d-p}| \ar[rr]^{Bi}_{\simeq} && B|\mathbf{C}^{\pitchfork}_{d-p}| && B|\mathbf{D}^{\pitchfork}_{d-p}| \ar[ll]_{B\alpha}^{\simeq} \ar[rr]^{B\gamma}_{\simeq} && |\mathbf{D}_{d-p}|
}
\end{equation}
such that all horizontal maps are weak homotopy equivalences.
Now consider the fibre-sequence of infinite loop-spaces,
$$\Omega^{\infty-1}\MT(d+1) \longrightarrow \Omega^{\infty-1}\MT_{P}(d+1) \longrightarrow \Omega^{\infty-1}\MT(d-p)$$
which is induced by the cofibre sequence of spectra,
$$\Sigma^{-1}\MT(d-p) \longrightarrow \MT(d+1) \longrightarrow \MT_{P}(d+1).$$
The weak homotopy equivalences constructed in the proof of Theorem \ref{The Main Theorem} yield a homotopy commutative diagram,
\begin{equation} \label{eq: homotopy commutative diagram}
\xymatrix{
|\mathbf{D}_{d+1}| \ar[d]^{|i|} \ar[rr]^{\simeq \ \ } && \Omega^{\infty-1}\MT(d+1) \ar[d] \\
|\mathbf{D}^{P}_{d+1}| \ar[d]^{|\beta_{1}|} \ar[rr]^{\simeq \ \ } && \Omega^{\infty-1}\MT_{P}(d+1) \ar[d] \\
|\mathbf{D}_{d-p}| \ar[rr]^{\simeq \ \ } && \Omega^{\infty-1}\MT(d-p) 
}
\end{equation}
where the horizontal maps are weak equivalences. 
Homotopy commutativity of (\ref{eq: homotopy commutative diagram}) together with the fact that the right column is a fibre-sequence, implies that 
$$\xymatrix{
|\mathbf{D}_{d+1}| \ar[rr]^{|i_{d}|} && |\mathbf{D}^{P}_{d+1}| \ar[rr]^{|\beta_{1}|} && |\mathbf{D}_{d-p}|
}$$
is a homotopy fibre-sequence. 
Then, commutativity of (\ref{eq: classifying space commutativity}) implies that 
$$\xymatrix{
B\mathbf{Cob}_{d+1} \ar[rr]^{B|i_{d}|} && B\mathbf{Cob}^{P}_{d+1} \ar[rr]^{B|\beta_{1}|} && B\mathbf{Cob}_{d-p}
}$$
is a homotopy fibre sequence. 
This completes the proof of Theorem \ref{thm: homotopy fibre}.
 \end{proof}
 
\section{The Space of $P$-Submersions}\label{The Submersion Theorem}
In this section we prove a result which implies Claim \ref{claim:
  submersion prop} used in the proof of Lemma \ref{lemma: limiting map equivalence}. 
  Our result is a version of \textit{Phillips'
  Submersion Theorem} \cite{Ph 67} for $P$-manifolds.
  
  Before stating our main result, we must review some submersion theory. 
  Let $X$ and $Y$ be smooth manifolds. 
  We denote by $\Sub(X, Y)$ the space of submersions $X \rightarrow Y$, topologized in the $C^{\infty}$-topology. 
  We denote by $\Sub^{f}(X, Y)$ the space of fibrewise surjective bundle maps, $TX \longrightarrow TY$. 
  We will express elements of $\Sub^{f}(X, Y)$ as pairs $(\hat{f}, f)$ where $\hat{f}: TX \longrightarrow TY$ is a fibrewise surjective bundle map and $f: X \rightarrow Y$ is the continuous map which underlies $\hat{f}$, i.e. $f$ is the unique map from $X$ to $Y$ such that the diagram 
  $$\xymatrix{
  TX \ar[rr]^{\hat{f}} \ar[d] && TY \ar[d] \\
  X \ar[rr]^{f} && Y
  }$$
  commutes, where the vertical maps are the bundle projections. 
  Elements of $\Sub^{f}(X, Y)$ are referred to as \textit{formal submersions}.
  Since $f$ is uniquely determined by $\hat{f}$ there is redundancy in this notation, however it will be useful to keep track of the underlying map.
 There is a map 
 $$\xymatrix{
 D: \Sub(X, Y) \longrightarrow \Sub^{f}(X, Y)
 }$$
 given by sending a submersion $f: X \rightarrow Y$ to the pair $(df, f)$ where $df: TX \rightarrow Y$ is the differential of $f$. 
 Clearly, the map $D$ is an embedding. 
 Now, a smooth manifold is said to be an \textit{open manifold} if it has no compact components.
 The main theorem from \cite{Ph 67} is the following. 
 \begin{theorem}[Phillips 1967] \label{theorem: phillips submersion theorem}
 Let $X$ be an open smooth manifold and let $Y$ be a smooth manifold without boundary.
 Then the embedding 
 $D: \Sub(X, Y) \longrightarrow \Sub^{f}(X, Y)$
 is a weak homotopy equivalence. 
  \end{theorem}

  For what follows, let $W$ be a $P$-manifold and let $X$ be a smooth manifold. 
  Recall from Section \ref{Singularities} that a smooth $P$-map $f: W \longrightarrow X$ is said to be a \textit{$P$-submersion} if both $f$ and $f_{\beta_{1}}: \beta_{1}W \rightarrow X$ are submersions. 
  We denote by $\Sub_{P}(W, X)$ the space of $P$-submersions $W \rightarrow X$, topologized as a subspace of the space of smooth maps from $W$ to $X$. 
 Below we define the space of \textit{formal $P$-submersions}.
 \begin{defn}
 We denote by $\Sub^{f}_{P}(W, X)$ the subspace of $\Sub^{f}(W, X)$ consisting of all formal submersions $(\hat{f}, f)$ which satisfy the following:
  \begin{enumerate}
  \item[i.] The underlying map $f: W \longrightarrow X$ is a $P$-map. 
  \item[ii.] There exists a fibrewise surjective bundle map $\hat{f}_{\beta_{1}}: T\beta_{1}W \longrightarrow TX$ such that the restriction of the bundle map $\hat{f}: TW \longrightarrow TX$ to the sub-bundle $T\partial_{1}W \subset TW|_{\partial_{1}W}$ (defined over $\partial_{1}W$) has the factorization,
  $$\xymatrix{
  T\partial_{1}W \ar[rr]^{d\phi_{1}} && T\beta_{1}W\times TP \ar[rr]^{\text{proj.}} && T\beta_{1}W \ar[rr]^{\hat{f}_{\beta}} && TX}$$
where $d\phi_{1}$ is the differential of the structure map $\phi_{1}: \partial_{1}W \stackrel{\cong} \longrightarrow \beta_{1}W\times P$.
  \end{enumerate}
  \end{defn}
  Notice that condition ii. of the above definition implies that $\hat{f}|_{T\partial_{1}W}: T\partial_{1}W \longrightarrow TX$ is fibrewise surjective. 
  Notice also that the correspondence $(\hat{f}, f) \mapsto (\hat{f}_{\beta_{1}}, f_{\beta_{1}})$ defines a continuous map 
  \begin{equation} \label{eq: formal P-sub restriction}
  \xymatrix{
\beta_{1}:  \Sub^{f}_{P}(W, X) \longrightarrow \Sub^{f}(\beta_{1}W, X).
  }
  \end{equation}

  The embedding $D: \Sub(W, X) \longrightarrow \Sub^{f}(W, X)$ from Theorem \ref{theorem: phillips submersion theorem}
restricts to an embedding, 
  $$\xymatrix{
 D_{P}: \Sub_{P}(W, X) \longrightarrow \Sub^{f}_{P}(W, X).
  }$$
  \begin{defn} A $P$-manifold $W$ is said to be an \textit{open $P$-manifold} if $\partial_{0}W = \emptyset$ and both $\beta_{1}W$ and $W$ have no compact components. 
\end{defn}
We have the following generalization of Theorem \ref{theorem: phillips submersion theorem}. 
\begin{theorem} \label{theorem: h-principle P-mfds}
Let $W$ be an open $P$-manifold and let $X$ be a smooth manifold without boundary. 
Then the embedding $D_{P}: \Sub_{P}(W, X) \longrightarrow \Sub^{f}_{P}(W, X)$ is a weak homotopy equivalence. 
\end{theorem}
The above theorem is proven in stages.
We must first derive some intermediate results. 
\begin{lemma} \label{lemma: restriction fibration 1}
The map 
$$\xymatrix{
\beta_{1}: \Sub^{f}_{P}(W, X) \longrightarrow \Sub^{f}(\beta_{1}W, X), \quad (\hat{f}, f) \mapsto (\hat{f}_{\beta_{1}}, f_{\beta_{1}})
}$$
is a Serre-fibration. 
\end{lemma}
\begin{proof}
Denote by $\Sub^{f}(W, X)|_{\partial_{1}W}$ the space of pairs $(\hat{g}, g)$ where 
$$\hat{g}: TW|_{\partial_{1}W} \longrightarrow TX$$
is a fibrewise surjective bundle map and $g: \partial_{1}W \longrightarrow X$ is map which makes the diagram 
$$\xymatrix{
TW|_{\partial_{1}W} \ar[rr]^{\hat{g}} \ar[d] && TX \ar[d] \\
\partial_{1}W \ar[rr]^{g} && X
}$$
commute. 
We have a restriction map,
\begin{equation} \label{eq: restriction map formal sub 1}
\xymatrix{
r|_{\partial_{1}}: \Sub^{f}(W, X) \longrightarrow \Sub^{f}(W, X)|_{\partial_{1}W}, \quad (\hat{f}, f) \mapsto (\hat{f}|_{\partial_{1}W}, f|_{\partial_{1}W})
}
\end{equation}
which is a Serre-fibration by \cite[Lemma 5.3]{Ph 67}.
From Example \ref{example: tangent bundle} we see that $TW$ has the structure of a $P$-vector bundle.
The collar embedding $h_{1}: \partial_{1}W\times\R_{+}\longrightarrow W$ together with the differential of the structure map 
$\phi_{1}: \partial_{1}W \stackrel{\cong} \longrightarrow \beta_{1}W\times P$,
 induces a bundle isomorphism
\begin{equation} \label{eq: bundle iso splitting}
\hat{\phi}_{1}: TW|_{\partial_{1}W} \stackrel{\cong} \longrightarrow (T\beta_{1}W\oplus\epsilon^{1})\times TP,
\end{equation}
which covers $\phi_{1}$ (above, $\epsilon^{1}$ is the one dimensional trivial bundle over $\beta_{1}W$).
 Using (\ref{eq: bundle iso splitting}), we define a map 
 \begin{equation} \label{eq: times P formal sub}
 \xymatrix{
T_{P}:  \Sub^{f}(\beta_{1}W, X) \longrightarrow \Sub^{f}(W, X)|_{\partial_{1}W}
}
 \end{equation}
by sending a fibrewise surjective bundle map $\hat{g}: T\beta_{1}W \longrightarrow TX$, to the fibrewise surjective bundle map $TW|_{\partial_{1}W} \longrightarrow TX$ given by the composition,
$$\xymatrix{
TW|_{\partial_{1}W} \ar[rr]^{\hat{\phi}_{1}\ \ \ \ }_{\cong} && (T\beta_{1}W\oplus\epsilon^{1})\times TP \ar[rr]^{ \ \ \ \text{proj.}} && T\beta_{1}W \ar[rr]^{\hat{g}} && TX.}
$$
It follows immediately from the definition of (\ref{eq: times P formal sub}) that the diagram,
$$\xymatrix{
\Sub^{f}_{P}(W, X) \ar[d]^{\beta_{1}} \ar@{^{(}->}[rr] && \Sub^{f}(W, X) \ar[d]^{r|_{\partial_{1}}}\\
\Sub^{f}(\beta_{1}W, X) \ar[rr]^{T_{P}} && \Sub^{f}(W, X)|_{\partial_{1}W} 
}$$
is a cartesian square. 
Since (\ref{eq: restriction map formal sub 1}) is a Serre-fibration, the fact that the above diagram is cartesian implies that 
$\xymatrix{
\beta_{1}: \Sub^{f}_{P}(W, X) \longrightarrow \Sub^{f}(\beta_{1}W, X)
}$
is a Serre-fibration as well. 
This concludes the proof of the lemma.
\end{proof}

We need to define a an intermediate space which can be compared to be both $\Sub_{P}(W, X)$ and $\Sub^{f}_{P}(W, X)$.
\begin{defn}
For a $P$-manifold $W$ and smooth manifold $X$, denote by $\Sub^{f}_{P, \beta}(W, X)$ the subspace of $\Sub^{f}_{P}(W, X)$ consisting of all formal submersions $(\hat{f}, f)$ such that, $\hat{f}_{\beta_{1}} = d f_{\beta_{1}}$
where $d f_{\beta_{1}}$ denotes the differential of $f_{\beta_{1}}$.
\end{defn}

\begin{lemma}
Let $W$ be an open $P$-manifold and let $X$ be a smooth manifold. 
The inclusion $\Sub^{f}_{P, \beta}(W, X) \hookrightarrow \Sub^{f}_{P}(W, X)$ is a weak homotopy equivalence. 
\end{lemma}
\begin{proof}
We may consider $\Sub(\beta W, X)$ as the subspace of $\Sub^{f}(\beta W, X)$ which consists of all formal submersions $(\hat{g}, g)$ such that $g$ is smooth and $\hat{g} = dg$ ($dg$ here is the differential of $g$). 
Recall the restriction map 
$\xymatrix{
\beta_{1}: \Sub^{f}_{P, \beta_{1}}(W, X) \longrightarrow \Sub^{f}(\beta_{1}W, X)
}$
which is a Serre-fibration by Lemma \ref{lemma: restriction fibration 1}. 
By definition we have 
$$\xymatrix{
\Sub^{f}_{P, \beta}(W, X) = \beta_{1}^{-1}(\Sub(\beta W, X)),
}$$
thus the diagram
$$\xymatrix{
\Sub^{f}_{P, \beta}(W, X) \ar@{^{(}->}[rr] \ar[d]^{\beta_{1}} && \Sub^{f}_{P}(W, X) \ar[d]^{\beta_{1}} \\
\Sub(\beta_{1}W, X) \ar@{^{(}->}[rr] && \Sub^{f}(\beta_{1}W, X)
}$$
is cartesian. 
Since the right-vertical map is a Serre-fibration, it follows that the above diagram is homotopy-cartesian. 
Since $W$ is an open $P$-manifold, it follows that $\beta_{1}W$ is an open manifold. 
By Theorem \ref{theorem: phillips submersion theorem} the bottom horizontal map in the above diagram is a weak equivalence. 
Since it is homotopy cartesian, it follows that the upper-horizontal map is a weak equivalence as well. 
This completes the proof.
\end{proof}

\begin{proof}[Proof of Theorem \ref{theorem: h-principle P-mfds}]
The space $\Sub_{P}(W, X)$ may be realized as the subspace of $\Sub^{f}_{P, \beta}(W, X)$ which consists of all formal $P$-submersions $(\hat{f}, f)$ such that $f$ is a smooth $P$-map and $df = \hat{f}$. 
To prove the theorem it will suffice to show that the relative homotopy group, 
$$\xymatrix{
\pi_{n}(\Sub^{f}_{P, \beta}(W, X), \; \Sub_{P}(W, X)) = 0 \quad \text{for all $n$.}
}$$ 
Let $F: (D^{n}, S^{n-1}) \longrightarrow (\Sub^{f}_{P, \beta}(W, X), \; \Sub_{P}(W, X))$ be a map of pairs. 
By definition of the space $\Sub^{f}_{P, \beta}(W, X)$, for all $x \in D^{n}$ we have $\beta_{1}F(x) \in \Sub(\beta_{1}W, X) \subset \Sub^{f}(\beta_{1}W, X)$. 
In other words, the formal submersion $\beta_{1}F(x)$ is integrable for all $x \in D^{n}$. 
Since $W$ is an open $P$-manifold, it follows automatically that $\partial_{1}W$ is an open manifold, as well as $W$. 
We then may apply the \textit{Relative Parametric H-Principle} (see  \cite[6.2 C]{EM 02} for the definition and \cite[7.2.4]{EM 02} for the statement of the relevant theorem needed to apply it) to obtain a homotopy 
$$\xymatrix{
F_{t}: (D^{n}, S^{n-1}) \longrightarrow (\Sub^{f}_{P, \beta}(W, X), \; \Sub_{P}(W, X)) \quad \text{for $t \in [0,1]$}
}$$
such that:
\begin{enumerate}
\item[i.] $F_{0} = F$,
\item[ii.] $F_{t}|_{\partial_{1}W} = F|_{\partial_{1}W}$ for all $t \in [0,1]$, and
\item[iii.] $F_{1}(D^{n}) \subset \Sub_{P}(W, X)$. 
\end{enumerate}
This completes the proof of the Theorem.
\end{proof}

\section{Stabilization of Sections of Vector Bundles}
\label{The Stabilization Map}
In this section we prove a lemma that implies Claim \ref{claim:
  stabilization map} used in the proof of Lemma \ref{lemma: limiting map equivalence}.  
  This result is essentially a relative version of \cite[Lemma 6]{MW 07}.

For any space $X$ and vector bundles $V_{1}$ and $V_{2}$ over $X$,
let $\Iso(E^{1}, E^{2})$ be the space of bundle-isomorphisms
covering the identity map. We have a stabilization map,
\begin{equation} \label{eq: stabilization}
  \sigma: \Iso(E^{1}, E^{2}) \longrightarrow \Iso(E_{1}\oplus \epsilon^{1}, E^{2}\oplus \epsilon^{1}), \quad   f \mapsto f\oplus Id_{\epsilon^{1}}.
 \end{equation}
 From \cite[Lemma 6]{MW 07} we have the following.
 \begin{lemma} \label{lemma: bundle stabilization map}
 Let $X$ be a manifold and let $E^{1}, E^{2} \longrightarrow V$ be vector bundles of fibre-dimension $k$. 
 The the stabilization map, 
$$ \sigma: \Iso(E^{1}, E^{2}) \longrightarrow \Iso(E^{1}\oplus \epsilon^{1}, E^{2}\oplus \epsilon^{1})$$
 is $(k-\dim(X)-1)$-connected.
 \end{lemma}
 We will need to use a relative version of the above lemma. 
 Let $A \subset X$ be a submanifold and let $E^{1}, E^{2} \longrightarrow X$ be vector bundles. 
 For $g \in \Iso(E^{1}|_{A}, E^{2}|_{A})$, denote by $\Iso(E^{1}, E^{2})_{g}$ the subspace of $\Iso(E^{1}, E^{2})$ consisting of all bundle isomorphisms $f: E^{1} \stackrel{\cong} \longrightarrow E^{2}$ such that the restriction of $f$ to $E^{1}|_{A}$ is equal to $g$. 
 We have the following.
 \begin{proposition} \label{prop: relative stabilization}
 Let $A \subset X$ be a submanifold of positive codimension and let $E^{1}, E^{2} \longrightarrow X$ be vector bundles of fibre-dimension $k$ and 
 let $g \in \Iso(E^{1}|_{A}, E^{2}|_{A})$.
 Then the stabilization map 
 $$\sigma: \Iso(E^{1}, E^{2})_{g} \longrightarrow \Iso(E^{1}\oplus\epsilon^{1}, E^{2}\oplus\epsilon^{1})_{g\oplus Id_{\epsilon^{1}}}$$
 is $(k-\dim(X)-1)$-connected. 
 \end{proposition}
 \begin{proof} 
 Consider the restriction map 
 $$r: \Iso(E^{1}, E^{2}) \longrightarrow \Iso(E^{1}|_{A}, E^{2}|_{A}), \quad f \mapsto f|_{A}.$$
The space $\Iso(E^{1}, E^{2})_{g}$ is equal to the subspace $r^{-1}(g) \subset \Iso(E^{1}, E^{2})$. 
 Furthermore it follows from \cite[Lemma 5.3]{Ph 67} that the map $r$ is a Serre-fibration. 
 The stabilization map induces a map of fibre-sequences,
 $$\xymatrix{
 \Iso(E^{1}, E^{2})_{g} \ar[d] \ar[rr] && \Iso(E^{1}\oplus\epsilon^{1}, E^{2}\oplus\epsilon^{1})_{g\oplus Id_{\epsilon^{1}}} \ar[d] \\
  \Iso(E^{1}, E^{2}) \ar[rr] \ar[d] &&   \Iso(E^{1}\oplus\epsilon^{1}, E^{2}\oplus\epsilon^{1}) \ar[d] \\
  \Iso(E^{1}|_{A}, E^{2}|_{A}) \ar[rr] && \Iso(E^{1}|_{A}\oplus\epsilon^{1}, E^{2}|_{A}\oplus\epsilon^{1}).
 }$$
 By Lemma \ref{lemma: bundle stabilization map} the middle horizontal map is $(k-\dim(X)-1)$-connected and since $\dim(A) < \dim(X)$ the bottom horizontal map is at least
 $(k-\dim(X))$-connected by \ref{lemma: bundle stabilization map}.
It then follows by application of the \textit{five-lemma} to the long exact sequence in homotopy groups that the top-horizontal map is $(k-\dim(X)-1)$-connected as well. 
This completes the proof. 
 \end{proof}
 
To prove Claim \ref{claim: stabilization map} we will need a version of Lemma \ref{lemma: bundle stabilization map} adapted for $P$-vector bundles over $P$-manifolds. 
For what follows, let $M$ be a $P$-manifold. Let $E_{P} \rightarrow P$ be a vector bundle and let 
 $E^{1}, E^{2} \longrightarrow M$ 
 be $P$-vector bundles of the same fibre-dimension, equipped with a specified identification, $E_{P} = E^{1}_{P} = E^{2}_{P}$.
 We denote by $\Iso_{E_{P}}(E^{1}, E^{2})$ the space of pairs $(f, f_{\beta})$ where 
 $$f: E^{1} \stackrel{\cong} \longrightarrow E^{2} \quad \text{and} \quad f_{\beta}: E^{1}_{\beta_{1}} \stackrel{\cong} \longrightarrow E^{2}_{\beta_{1}}$$ 
are bundle isomorphisms that cover the identity maps such that the diagram
$$\xymatrix{
E^{1}|_{\partial_{1}M} \ar[d]^{\widehat{\phi}_{E^{1}}} \ar[rrr]^{f|_{\partial_{1}M}} &&&  E^{2}|_{\partial_{1}M} \ar[d]^{\widehat{\phi}_{E^{1}}} \\
E^{1}_{\beta_{1}}\times E_{P}\oplus \epsilon^{1} \ar[rrr]^{f_{\beta_{1}}\times Id_{E_{P}}\oplus Id_{\epsilon^{1}}} &&& E^{2}_{\beta_{1}}\times E_{P}\oplus \epsilon^{1}  
}$$
 commutes.
 The is a stabilization map 
 \begin{equation} \label{eq: P-bundle stabilization}
 \sigma_{E_{P}}: \Iso_{E_{P}}(E^{1}, E^{2}) \longrightarrow \Iso_{E_{P}}(E^{1}\oplus\epsilon^{1}, E^{2}\oplus\epsilon^{1}), \quad (f, f_{\beta}) \mapsto (f\oplus\epsilon^{1}, f_{\beta}\oplus\epsilon^{1}).
  \end{equation}
  The main result of this section is the following.
  \begin{theorem}
  Let $M$, $E_{P}$, $E^{1}$ and $E^{2}$ be as above and denote by $k$ the fibre-dimension of $E^{1}$ and $E^{2}$. 
  Suppose that the fibre-dimension of $E_{P}$ equal to $p = \dim(P)$. 
  Then the stabilization map from (\ref{eq: P-bundle stabilization}) is $(k - \dim(M) -1)$-connected. 
  \end{theorem}
  \begin{proof} 
  Consider the map 
  $$T_{E_{P}}: \Iso(E^{1}_{\beta_{1}}, E^{2}_{\beta_{2}}) \longrightarrow \Iso(E^{1}|_{\partial_{1}M}, E^{2}|_{\partial_{1}M})$$
  defined by sending a bundle isomorphism $f: E^{1}_{\beta_{1}} \longrightarrow E^{2}_{\beta_{1}}$ to the bundle isomorphism given by the composition,
  $$\xymatrix{
  E^{1}|_{\partial_{1}M} \ar[rr]^{\hat{\phi}_{E^{1}}}_{\cong} && (E^{1}_{\beta_{1}}\times E_{P})\oplus\epsilon^{1} \ar[rrr]^{f_{\beta}\times Id_{E_{P}}\oplus Id_{\epsilon^{1}}} &&& (E^{1}_{\beta_{1}}\times E_{P})\oplus\epsilon^{1} \ar[rr]^{\ \ (\hat{\phi}_{E^{2}})^{-1}}_{\cong} && E^{2}|_{\partial_{1}M}.
  }$$
  We also have the map 
  $$\beta_{1}: \Iso(E^{1}_{\beta_{1}}, E^{2}_{\beta_{2}}) \longrightarrow \Iso(E^{1}_{\beta_{1}}, E^{2}_{\beta_{1}}), \quad (f, f_{\beta}) \mapsto f_{\beta}.$$
  It follows immediately from the definition of the space $\Iso_{E_{P}}(E^{1}, E^{2})$ that the diagram,
\begin{equation} \label{eq: label cartesian bundle map diagram}
\xymatrix{
\Iso_{E_{P}}(E^{1}, E^{2}) \ar[d]^{\beta_{1}} \ar@{^{(}->}[rr] && \Iso(E^{1}, E^{2}) \ar[d]^{r} \\
\Iso(E^{1}_{\beta_{1}}, E^{2}_{\beta_{1}}) \ar[rr]^{T_{E_{P}}} && \Iso(E^{1}|_{\partial_{1}M}, E^{2}|_{\partial_{1}M})
}
\end{equation}
is cartesian, where the top horizontal map is the inclusion map and the right-vertical map $r$ is the restriction map. 
By \cite[Lemma 5.3]{Ph 67} the restriction map $r$ is a Serre-fibration with fibre over $g \in \Iso(E^{1}|_{\partial_{1}M}, E^{2}|_{\partial_{1}M})$ equal to the space
$\Iso(E^{1}, E^{2})_{g}$. 
It follows from this that the diagram (\ref{eq: label cartesian bundle map diagram}) is homotopy cartesian
and that the left-vertical map $\beta_{1}$ is a Serre-fibration as well. 
For $g \in \Iso(E^{1}_{\beta_{1}}, E^{2}_{\beta_{1}})$, the fibre over the map $\beta_{1}$ over $g$ is equal to the space 
$\Iso(E^{1}|_{\partial_{1}M}, E^{2}|_{\partial_{1}M})_{T_{E_{P}}(g)}$. 
We then have a map of fibre sequences,
 $$\xymatrix{
\Iso(E^{1}, E^{2})_{T_{E_{P}}(g)} \ar[d] \ar[rr] && \Iso(E^{1}\oplus\epsilon^{1}, E^{2}\oplus\epsilon^{1})_{T_{E_{P}}(g)\oplus Id_{\epsilon^{1}}} \ar[d] \\
  \Iso_{E_{P}}(E^{1}, E^{2}) \ar[rr] \ar[d] &&   \Iso_{E_{P}}(E^{1}\oplus\epsilon^{1}, E^{2}\oplus\epsilon^{1}) \ar[d] \\
 \Iso(E^{1}_{\beta_{1}}, E^{2}_{\beta_{1}}) \ar[rr] &&  \Iso(E^{1}_{\beta_{1}}\oplus\epsilon^{1}, E^{2}_{\beta_{1}}\oplus\epsilon^{1}).
 }$$
 By Lemma \ref{lemma: bundle stabilization map} and Proposition \ref{prop: relative stabilization} the top horizontal map is $(k-\dim(M)-1)$-connected and the degree of connectivity of the bottom-horizontal map is equal to,
 $$[(k-p-1) - \dim(\beta_{1}M) - 1] \; = \; [(k-p-1) - (\dim(M) - p -1) - 1] \; = \; (k-\dim(M)-1)$$
 as well.
 It then follows by application of the \textit{five-lemma} to the long-exact sequence in homotopy groups associated to the above fibrations that the middle map is $(k-\dim(X)-1)$-connected.
 This completes the proof of the Theorem.
 \end{proof}


\section{Proof of Proposition  \ref{prop: homotopy monoid}}\label{section:monoidal}
In Section \ref{The Main Theorem} we proved that for all $k$
there is an isomorphism of sets
\begin{equation} \label{eq: unbased isomorphism}
\xymatrix{
[S^{k}, |\mathbf{D}^{P}_{d+1}|] \cong [S^{k},
    \Omega^{\infty-1}\MT_{P}(d+1)] 
    }
 \end{equation}
  which is induced by a zig-zag, 
  $\xymatrix{
  |\mathbf{D}^{P}_{d+1}| && |\widehat{\mathbf{D}}^{P}_{d+1}| \ar[ll]_{\simeq} \ar[rr]^{T \ \ \ \ } &&  \Omega^{\infty-1}\MT_{P}(d+1)
  }$
  where the first map is a weak homotopy equivalence. 
  In order to prove Proposition
  \ref{prop: homotopy monoid}, we need to show that we have an
  isomorphism of homotopy groups for any choice of base-point
  on any path component.  In this section we resolve this issue. 
  
  First note that since $\Omega^{\infty-1}\MT_{P}(d+1)$ is an infinite loop-space, it has the structure of a topological monoid.
 In particular, this implies that the identity component of $\Omega^{\infty-1}\MT_{P}(d+1)$ is a connected $H$-space. 
 It then follows from \cite[Example 4A.3]{H 01} that for all $k \in \N$ and $x_{0}$ in the identity component $\Omega^{\infty-1}_{0}\MT_{P}(d+1) \subset \Omega^{\infty-1}\MT_{P}(d+1)$, 
 there is a bijection between the homotopy group
$\pi_{k}(\Omega^{\infty-1}\MT_{P}(d+1), x_{0})$ and the set $[S^{k}, \Omega^{\infty-1}_{0}\MT_{P}(d+1)]$
 induced by the natural map defined by forgetting the base point. 
 Now, the monoid structure on $\pi_{0}(\Omega^{\infty-1}\MT_{P}(d+1))$ is a group. 
This implies that all path components of $\Omega^{\infty-1}\MT_{P}(d+1)$ are homotopy equivalent. 
It follows that for all $x \in \Omega^{\infty-1}\MT_{P}(d+1)$ in any path component and all $k \in \N$, there is an isomorphism
\begin{equation} \label{eq: forget base point 1}
\pi_{k}(\Omega^{\infty-1}\MT_{P}(d+1), x) \stackrel{\cong} \longrightarrow [S^{k}, \Omega^{\infty-1}_{x}\MT_{P}(d+1)],
\end{equation}
where $\Omega^{\infty-1}_{x}\MT_{P}(d+1) \subset \Omega^{\infty-1}\MT_{P}(d+1)$ is the path component that contains the element $x \in \Omega^{\infty-1}\MT_{P}(d+1)$.

  We will need to show that $|\mathbf{D}^{P}_{d+1}|$ has the structure
  of a monoid (with product defined up to homotopy) and that the map
$|T|: |\mathbf{D}^{P}_{d+1}| \rightarrow \Omega^{\infty-1}\MT_{P}(d+1)$ 
is a monoid homomorphism inducing isomorphism on $\pi_{0}$. The method
  of this section is very similar to the proof of \cite[Theorem 3.8]{MW
    07}. 
\begin{proposition} \label{prop: monoid homomorphism} 
The space $|\widehat{\mathbf{D}}^{P}_{d+1}|$ has
the structure of topological monoid up to homotopy.
\end{proposition}
\begin{proof} The monoid (up to homotopy) structure on $|\widehat{\mathbf{D}}^{P}_{d+1}|$ is defined as follows:
Let
$\widehat{\mathbf{D}}^{P}_{d+1}\bar{\times}\widehat{\mathbf{D}}^{P}_{d+1}$
be the sheaf defined by letting
$(\widehat{\mathbf{D}}^{P}_{d+1}\bar{\times}\widehat{\mathbf{D}}^{P}_{d+1})(X)$
consist of all pairs
$$ 
((W_{1}, e_{1}), (W_{2}, e_{2}))
\in
\widehat{\mathbf{D}}^{P}_{d+1}(X)\times\widehat{\mathbf{D}}^{P}_{d+1}(X)$$
such that the image of $e^{1}$ and $e^{2}$ are disjoint in $X\times
\R\times\R_{+}\times\R^{d+\bar{n}}\times\R^{p+m}$. There is a natural map
\begin{equation} \label{eq: partial product} 
\mu:
(\widehat{\mathbf{D}}^{P}_{d+1}\bar{\times}\widehat{\mathbf{D}}^{P}_{d+1})(X)
\longrightarrow \widehat{\mathbf{D}}^{P}_{d+1}(X)
\end{equation}
defined by sending a pair $((W_{1}, e_{1}), (W_{2}, e_{2}))$ to the element $(W_{1}\sqcup W_{2}, e_{1}\sqcup e_{2}).$
This map yields a partially defined product on $\widehat{\mathbf{D}}^{P}_{d+1}(X)$ which is clearly associative and commutative. The identity element is given by the empty set. 
The inclusion map 
$$\xymatrix{j: \widehat{\mathbf{D}}^{P}_{d+1}\bar{\times}\widehat{\mathbf{D}}^{P}_{d+1} \ar[rr] && \widehat{\mathbf{D}}^{P}_{d+1} \times \widehat{\mathbf{D}}^{P}_{d+1}}$$
is a weak equivalence of sheaves. Roughly, given 
$$((W_{1}, e_{1}), (W_{2}, e_{2})) \in \widehat{\mathbf{D}}^{P}_{d+1}(X)\times\widehat{\mathbf{D}}^{P}_{d+1}(X)$$ 
such that the images of $e_{1}$ and $e_{2}$ intersect, after increasing the dimension of the ambient space,  one can find an isotopy of embeddings which pulls $W_{1}$ away from $W_{2}$. 
Letting $|k|$ be a pseudo-inverse for $|j|$, the product described above yields a homotopy monoid structure on the representing space $|\widehat{\mathbf{D}}^{P}_{d+1}|$ with product given by
\begin{equation} \label{eq: product} \xymatrix{
|\widehat{\mathbf{D}}^{P}_{d+1}| \times |\widehat{\mathbf{D}}^{P}_{d+1}| \ar[r]^{\simeq} & |\widehat{\mathbf{D}}^{P}_{d+1} \times \widehat{\mathbf{D}}^{P}_{d+1}| \ar[rr]^{|k|} && |\widehat{\mathbf{D}}^{P}_{d+1}\bar{\times}\widehat{\mathbf{D}}^{P}_{d+1}| \ar[rr]^{|\mu|} && |\widehat{\mathbf{D}}^{P}_{d+1} |} \end{equation}
where the left-most map is some choice of homotopy equivalence. The empty-set element in $\widehat{\mathbf{D}}^{P}_{d+1}(\text{pt.})$ (which induces the empty-set element in $\widehat{\mathbf{D}}^{P}_{d+1}(X)$ for any $X$ by pulling back over the constant map) determines an element  $e \in |\widehat{\mathbf{D}}^{P}_{d+1}|$. This is easy to see by examining the construction of $|\widehat{\mathbf{D}}^{P}_{d+1}|$ as the geometric realization of the simplicial set $(l \mapsto \widehat{\mathbf{D}}^{P}_{d+1}(\triangle^{l}))$. From the fact that the empty set is the identity for the partially defined product in  (\ref{eq: partial product}) it follows that $e$ is the identity (up to homotopy) for the product defined in (\ref{eq: product}). Associativity also follows from associativity of (\ref{eq: partial product}). \end{proof}

Since $|\widehat{\mathbf{D}}^{P}_{d+1}|$ has the structure of a monoid up to homotopy, 
it follows from \cite[Example 4A.3]{H 01} that for each $k \in \N$ and point $x_{0}$ on the identity component $|\widehat{\mathbf{D}}^{P}_{d+1}|_{0} \subset |\widehat{\mathbf{D}}^{P}_{d+1}|$,
there is an isomorphism, $\pi_{k}(|\widehat{\mathbf{D}}^{P}_{d+1}|, x_{0}) \stackrel{\cong} \longrightarrow [S^{k}, |\widehat{\mathbf{D}}^{P}_{d+1}|_{0}]$ induced by the map defined by forgetting base points. 
Now consider the map $|H|\circ|\widehat{T}|: |\widehat{\mathbf{D}}^{P}_{d+1}| \longrightarrow \Omega^{\infty-1}\MT_{P}(d+1)$ from Section \ref{subsection: parametrized pontryagin thom}).
\begin{proposition} \label{prop: monoid iso loops} The map $|H|\circ|\widehat{T}|$
induces a monoid isomorphism,
$$\xymatrix{ \pi_{0}(|\widehat{\mathbf{D}}^{P}_{d+1}|) \stackrel{\cong} \longrightarrow \pi_{0}(\Omega^{\infty-1}\MT_{P}(d+1)). }$$
\end{proposition}
\begin{proof} This proposition is proven by examining the Thom-Pontryagin map
$$\xymatrix{
\widehat{T}: \widehat{\mathbf{D}}^{P}_{d+1}(\text{pt.})  \longrightarrow \mathcal{Z}^{P}_{d+1}(\text{pt.})
}$$ 
and checking that it sends a disjoint union $(W_{1}\sqcup W_{2}, e_{1}\sqcup e_{2}) \in \widehat{\mathbf{D}}^{P}_{d+1}(\text{pt.})$ to a sum of ``loops'' in $\mathcal{Z}^{P}_{d+1}(\text{pt.}) \sim \Omega^{\infty-1}_{P}\widehat{\Th(U^{\perp}_{d+1, n})}$. 
This follows the exact same argument as in the proof of the classical \textit{Pontryagin-Thom Theorem} from \cite{St 68}. We refer the reader there for details. 
\end{proof}  

Now, since $\pi_{0}(\Omega^{\infty-1}\MT_{P}(d+1))$ is a group, it follows from Proposition \ref{prop: monoid iso loops} that $\pi_{0}(|\widehat{\mathbf{D}}^{P}_{d+1}|)$ is a group as well. 
From this group structure it follows that all path components of $|\widehat{\mathbf{D}}^{P}_{d+1}|$ are homotopy equivalent. 
We then have that for every $x \in |\widehat{\mathbf{D}}^{P}_{d+1}|$ in any path component and for all $k \in \N$, there is an isomorphism
\begin{equation} \label{eq: unbased iso 2}
\xymatrix{
 \pi_{k}(|\widehat{\mathbf{D}}^{P}_{d+1}|, x) \stackrel{\cong} \longrightarrow [S^{k}, |\widehat{\mathbf{D}}^{P}_{d+1}|_{x}], 
 }
 \end{equation}
 where $|\widehat{\mathbf{D}}^{P}_{d+1}|_{x} \subset |\widehat{\mathbf{D}}^{P}_{d+1}|$ is the path component containing $x$. 
Consider the commutative diagram,
 $$\xymatrix{
  \pi_{k}(|\widehat{\mathbf{D}}^{P}_{d+1}|, x) \ar[d]^{\cong} \ar[rrr]^{|H|\circ|\widehat{T}|} &&&  \pi_{k}(\Omega^{\infty-1}\MT_{P}(d+1), y) \ar[d]^{\cong} \\
  [S^{k}, |\widehat{\mathbf{D}}^{P}_{d+1}|_{x}] \ar[rrr]^{|H|\circ|\widehat{T}|}_{\cong} &&&  [S^{k}, \Omega^{\infty-1}_{y}\MT_{P}(d+1)], }$$
  where $y = |H|\circ|\widehat{T}|(x)$. 
It follows that the top horizontal map is an isomorphism. 
This concludes the proof of Proposition \ref{prop: homotopy monoid}.
  
 \appendix 
\section{} \label{Appendix A} In this section we prove a result which implies 
Theorem \ref{theorem: local triviality}. This result is a slight modification of the main theorem 
from \cite{BF 81} and our proof is similar.
\begin{lemma} \label{lemma: Principal bundle} Let $(W; M_{a}, M_{b})$ be a $P$-manifold cobordism triple. 
For any positive integer $n$, the quotient map, 
$$\xymatrix{
q: \mathcal{E}_{P, n}(W; M_{a}, M_{b}) \longrightarrow  \dfrac{\mathcal{E}_{P, n}(W; M_{a}, M_{b})}{\Diff_{P}(W; M_{a}, M_{b})}
}$$
is a locally trivial fibre bundle. 
\end{lemma}
\begin{proof} Let $f \in \mathcal{E}_{P, n}(W; M_{a}, M_{b})$ and let $[f]$  denote the class of $f$ in the orbit space, 
$$\xymatrix{
\dfrac{\mathcal{E}_{P, n}(W; M_{a}, M_{b})}{\Diff_{P}(W; M_{a}, M_{b})}.
}$$ 
By how the action is defined, for any $g \in q^{-1}([f])$ we have $g(W) = f(W)$. 
By definition of the space $\mathcal{E}_{P, n}(W; M_{a}, M_{b})$ we have 
$$f(\partial_{1}W) \subset [0,1]\times\{0\}\times\R^{d+\bar{n}}\times\R^{p+m}$$
and there exists a real number $\varepsilon > 0$ such that 
$$[0, \varepsilon)\times f(\partial_{1}W) \subset f(W)$$
where $[0, \varepsilon)\times f(\partial_{1}W)$ is understood to be the set of all
$$(t, \; s, \; x, \; y) \in [0,1]\times\R_{+}\times\R^{d+\bar{n}}\times\R^{p+m}$$
such that $(t, 0, x, y) \in f(\partial_{1}W)$ and $s \in [0, \varepsilon) \subset \R_{+}$. 
Let 
$$N \subset [0,1]\times\R_{+}\times\R^{d+\bar{n}}\times\R^{p+m}$$
be a geodesic neighborhood for $f(W)$. 
Denote by $\pi: N \longrightarrow f(W)$ the projection map.
We have,
\begin{equation} \label{eq: geodesic neighborhood factorization}
N \cap( [0,1]\times\{0\}\times\R^{d+\bar{n}}\times\R^{p+m}) \; = \; N_{\beta_{1}}\times N_{P}
\end{equation}
where $N_{\beta_{1}} \subset \R^{d+\bar{n}}$ is a geodesic neighborhood for $f_{\beta_{1}}(\beta_{1}W) \subset \R^{d+\bar{n}}$
and $N_{P} \subset \R^{p+m}$ is a geodesic neighborhood for $i_{P}(P) \subset \R^{p+m}$. 
We denote by 
$$\pi_{\beta_{1}}: N_{\beta_{1}W} \longrightarrow \beta_{1}W \quad \text{and} \quad  \pi_{P}: N_{P} \longrightarrow P$$ 
the projection maps. 
Now let 
$$z \; \in \; \bigg([0,1]\times [0, \varepsilon)\times \R^{d+\bar{n}}\times i_{P}(P)\bigg)\cap N.$$ 
We write $z$ as
$z = (s, t, x, y)$
with $s \in [0, \varepsilon)$,  $(t, x) \in [0,1]\times \R^{\bar{d}-1+\bar{n}}$, and $y \in i_{P}(P)$
(we permute the factors of $[0, \varepsilon)$ and $[0, 1]$ to make for more convenient notation for the constructions ahead).
It follows from the factorization of (\ref{eq: geodesic neighborhood factorization}) that,
  \begin{equation} \label{eq: proj} 
  \pi(s, t, x, y) \; = \; (s, \pi_{\beta_{1}}(t, x), y).
  \end{equation} 
 Let
 $$\xymatrix{
 U \subset \mathcal{E}_{P, n}(W; M_{a}, M_{b})
 }$$
 be an open neighborhood of $f$ with the property that $g(W) \subset N$ for
 all $g \in U$. By definition of the $C^{\infty}$-topology, such a
 subset does indeed exist.
 Now let $q$ be the quotient map from the statement of the theorem. 
Let let $\bar{U}$ denote the image $q(U)$. 
 For any such $g \in q^{-1}(\bar{U})$, we obtain a smooth map
 $W \longrightarrow W$ given by the formula $x \mapsto f^{-1}\circ\pi\circ g(x)$.
 We will abuse notation and denote this map by 
 $$f^{-1}\circ\pi\circ g: W \longrightarrow W.$$
 It follows from (\ref{eq: proj}) that for all $g \in q^{-1}(\bar{U})$, the map $f^{-1}\circ\pi\circ g$ is an element of the mapping space $C^{\infty}_{P}(W; M_{a}, M_{b})$ introduced in Section \ref{section: mapping spaces}. 
 We have a map 
 $$\alpha: q^{-1}(\bar{U}) \longrightarrow C^{\infty}_{P}(W; M_{a}, M_{b}), \quad g \mapsto f^{-1}\circ\pi\circ g.$$
 Notice that $\alpha(f) = Id_{W}$, which is of course an element of $\Diff_{P}(W; M_{a}, M_{b})$. 
By Proposition \ref{prop: P-diff open set},
 $$\xymatrix{
 \Diff_{P}(W; M_{a}, M_{b}) \subset C^{\infty}_{P}(W; M_{a}, M_{b})
 }$$ 
 is an open subset.
We may then choose a small neighborhood $\bar{U}' \subset \bar{U}$ of $[f]$ such that 
 $$\xymatrix{
 \alpha(q^{-1}(\bar{U}')) \subset \Diff_{P}(W; M_{a}, M_{b}).
 }$$ 
 Using $\alpha$, we define a map 
 $$\xymatrix{
 \Phi: q^{-1}(\bar{U}') \longrightarrow \bar{U}'\times\Diff_{P}(W; M_{a}, M_{b}), \quad g \mapsto ([g], \alpha(g)).
 }$$ 
 It follows easily that this map $\Phi$ is a local trivialization of the projection $q$. 
 This concludes the proof of the Theorem.
  \end{proof}

  \end{document}